\documentclass[final]{siamltex}
\usepackage{subfig, amsmath, amsfonts}
\usepackage{graphicx,color}
\usepackage[ruled]{algorithm2e}

\newcommand{\mc}[1]{\mathcal{#1}}

\newcommand{\abs}[1]{\lvert#1\rvert}
\newcommand{\norm}[1]{\lVert#1\rVert}

\newcommand{\wt}[1]{\widetilde{#1}}

\newcommand{\Or}{\mathcal{O}}

\newcommand{\REV}[1]{#1}

\newtheorem{remark}[theorem]{Remark}

\title{Localized spectrum slicing}

\author{
Lin Lin 
\thanks{
Department of Mathematics, University of California Berkeley
and Computational Research Division, Lawrence Berkeley National
Laboratory, Berkeley, CA 94720. Email: linlin@math.berkeley.edu} 
}

\begin{document}

\maketitle

\begin{abstract}
%  Given a sparse Hermitian matrix $A$ and a real number $\mu$, we
%  develop a new basis set called the localized spectrum slicing (LSS)
%  basis set.  Each vector in the LSS basis set is approximately a linear
%  combination of eigenvectors of $A$ corresponding to eigenvalues near
%  $\mu$, and is approximately a sparse vector, and therefore the LSS
%  basis set is localized both spectrally and spatially.  The spatial
%  locality is derived from the decay properties of a Gaussian matrix
%  function acting on $A$, and the balance of spectral and spatial
%  locality is controlled by a single parameter $\sigma$ which is the
%  width of the Gaussian.  The LSS basis set leads to sparse projected
%  matrices with reduced sizes, which allows the projected problems to be
%  solved efficiently with techniques that require to maintain matrix
%  sparsity.  We demonstrate that the LSS basis set can be used to
%  efficiently solve interior eigenvalue problems for a discretized
%  second order partial differential operator in one-dimensional and
%  two-dimensional domains.
  Given a sparse Hermitian matrix $A$ and a real number $\mu$, we
  construct a set of sparse vectors, each approximately spanned only by
  eigenvectors of $A$ corresponding to eigenvalues near $\mu$. This set
  of vectors spans the column space of a localized spectrum slicing
  (LSS) operator, and is called an LSS basis set.  The sparsity of the
  LSS basis set is related to the decay properties of matrix Gaussian
  functions.  We present a divide-and-conquer strategy with controllable
  error to construct the LSS basis set. This is a purely algebraic
  process using only submatrices of $A$, and can therefore be applied to
  general sparse Hermitian matrices.  
%  We demonstrate that the computational
%  complexity of our method is $\Or(n)$ under the assumption that the
%  spectral radius of $A$ \REV{and the sizes of the submatrices of
%  $A$ do not increase with $n$}.
  The LSS basis set leads to sparse projected matrices with reduced
  sizes, which allows the projected problems to be solved efficiently
  with techniques using sparse linear algebra. As an example, we demonstrate
  that the LSS basis set can be used to solve interior eigenvalue
  problems for a discretized second order partial differential operator
  in one-dimensional and two-dimensional domains, \REV{as well as for a
  matrix of general sparsity pattern}.
\end{abstract}


 \pagestyle{myheadings}
\thispagestyle{plain}
\markboth{L. LIN}{Localized spectrum slicing}

\section{Introduction}\label{sec:intro}

Let $A$ be an $n\times n$ large, sparse, Hermitian matrix.  In many
applications in science and engineering, one would like to find
eigenvalues and eigenfunctions of $A$ near a given real number $\mu$.
As a motivating problem, we consider $A$ to be obtained
from a certain discretization (e.g. finite difference or finite element
discretization) of a second order partial differential operator of the
form $-\Delta + V(x)$, where $\Delta$ is the Laplacian operator, and
$V(x)$ is a potential function.  
Depending on the context and the choice of $V$, this type of problems
can arise from quantum mechanics, wave propagation, electromagnetism
etc.

When $\mu$ locates inside the spectrum of $A$, the
eigenvalues to be computed are called interior eigenvalues.  These
interior eigenvalues and corresponding eigenfunctions are in general
difficult to compute.  Since $n$ is large and $A$ is sparse, iterative
methods such as inverse power method~\cite{GolubVan1996}, preconditioned
conjugate gradient type of
methods~\cite{BradburyFletcher1966,Davidson1975,Knyazev2001}, and
shift-inverse Lanczos type of
methods~\cite{LehoucqSorensen2000,ZhangSmithSternbergEtAl2007} are
desirable.  The effectiveness of such methods often depends on the
availability of a good preconditioner that can approximately apply
$(A-\mu I)^{-1}$ to vectors, and such preconditioner can be difficult to
construct.  

Another type of methods that recently receives increasing amount of
attention is based on the construction of a matrix function
$f_{\mu}(A)$, where the corresponding scalar function $f_{\mu}(z)$ only
takes significant values on a small interval near $\mu$ on the real
line. 
%Here $f_{\mu}(z)$ is viewed as an analytic function inside a region in the
%complex plane so that the spectrum of $A$ lies inside the analytic
%region of $f_{\mu}(z)$.  
Such a matrix function $f_{\mu}(A)$ can be called a
\textit{spectrum slicing} operator, since for any vector
$v\in\mathbb{C}^{n}$, $f_{\mu}(A) v$ is approximately only spanned by eigenvectors
of $A$ corresponding to eigenvalues near $\mu$, and the vector
$f_{\mu}(A) v$ is said to be \textit{spectrally localized}. The spectrum
slicing operator can be simultaneously applied to a set of random
vectors $V=[v_{1},\ldots,v_{p}]$. When $p$ is large enough but is still
small compared to $n$, the subspace spanned
by 
\[
W=f_{\mu}(A)V
\] 
will approximately contain the subspace of all eigenvectors
corresponding to eigenvalues near $\mu$.  Let 
\[
A_{W} = W^{*} A W, \quad B_{W} = W^{*} W, 
\]
then the desired eigenvalues and eigenvectors can be
computed via the solution of a generalized eigenvalue problem
\begin{equation}
  A_{W} C = B_{W} C \Theta.
  \label{eqn:geneig}
\end{equation}
%The dimension of the matrices $A_{W},B_{W}$ is usually small, and it may be possible to
%solve the generalized eigenvalue problem~\eqref{eqn:geneig} via direct
%methods by treating $A_{W},B_{W}$ as dense matrices. 
In practice $f_{\mu}(A)$ can be constructed through relatively high
order Chebyshev polynomials~\cite{SchofieldChelikowskySaad2012}, or
contour integral based methods~\cite{Polizzi2009,SakuraiSugiur2003}. It
should be noted that contour integral based methods still require
solving equations of the form $(A-z I)^{-1}v$ where $z$ is close to
$\mu$ in the complex plane, either through direct methods or iterative
methods.

In general the spectrum slicing operator $f_{\mu}(A)$ is a dense matrix.
Therefore the matrix $W=f_{\mu}(A)V$ is in general a dense matrix, regardless of how
the initial matrix $V\in\mathbb{C}^{n\times p}$ is chosen.  Furthermore, the matrices
$A_{W},B_{W}$ are in general dense matrices that do not reveal much
structure to be further exploited, and the solution of the projected
problem~\eqref{eqn:geneig} may still be expensive when $p$ is large.  

\subsection{Contribution}

%In this paper, we construct a spectrum slicing operator so that
%$f_{\mu}(A)$ have many entries that are small in magnitude, so that
%after truncating these small entries the resulting matrix
%$\wt{f}_{\mu}(A)$ is close to be a spectrum slicing operator
%but is also sparse. In this sense, $f_{\mu}(A)$ is
%called a \textit{localized spectrum slicing (LSS) operator}. 
%More specifically, let the spectral radius of $A$ be $1$, and the LSS
%operator is simply obtained from a Gaussian function
%\begin{equation}
%  f_{\sigma,\mu}(z) = 
%  e^{-\frac{(z-\mu)^2}{\sigma^2}}, \quad \sigma>0,\mu\in (-1,1).
%  \label{eqn:gaussian}
%\end{equation}
%It is clear that the parameter $\sigma$ describes the spectral
%locality, which is given by the support size of $f_{\sigma,\mu}(z)$
%along the real line after truncating small values of
%$f_{\sigma,\mu}(z)$.  We demonstrate that the spatial locality (number of
%nonzero entries of $f_{\mu}(A)$ after truncation) is directly related
%to $\frac{1}{\sigma}$.  Therefore by carefully choosing the value of $\sigma$, 
%spectral locality and spatial locality can be balanced, the LSS operator
%$f_{\mu}(A)$ can be constructed in a divide-and-conquer fashion using
%only a sequence of submatrices of $A$. 

In this paper, we consider the use of a simple choice of Gaussian
function with a positive number $\sigma$
\begin{equation}
  f_{\sigma,\mu}(z) = 
  e^{-\frac{(z-\mu)^2}{\sigma^2}},
  \label{eqn:gaussian}
\end{equation}
and the corresponding matrix Gaussian function $f_{\sigma,\mu}(A)$ is
spectrally localized near $\mu$ with width proportional to $\sigma$.
We demonstrate that under a proper choice of $\sigma$,
$f_{\sigma,\mu}(A)$ can have
many entries that are small in magnitude, so that after truncating these
small entries the resulting matrix is close to be a
spectrum slicing operator but is also sparse. In this sense,
$f_{\sigma,\mu}(A)$ is called a \textit{localized spectrum slicing (LSS)
operator}.   

We demonstrate that the LSS operator $f_{\sigma,\mu}(A)$ can be 
constructed in a divide-and-conquer method with controllable error using
only a sequence of submatrices of $A$ with $\Or(n)$ cost, 
%provided that
%the spectral radius of $A$  \REV{and the sizes of the submatrices of $A$
%do not increase with $n$}. 
\REV{under certain assumptions of the behavior of the sparsity, spectral
radius, and sizes of submatrices of $A$ as $n$ increases.}
The column space of the LSS
operator is spanned by a sparse matrix $U\in \mathbb{C}^{n\times p}$,
and the subspace spanned by $U$ will approximately contain the subspace
of eigenvectors to be computed.  As a result, the projected matrices
\begin{equation}
A_{U} = U^{*} A U, \quad B_{U} = U^{*} U
  \label{eqn:projAB}
\end{equation}
are sparse matrices. In this aspect, the matrix $U$ can be regarded as a
specially tailored basis set for representing the subspace approximately spanned by
eigenvectors of $A$ near $\mu$, and each column of $U$ is localized both
spectrally and spatially.  In the following text $U$ is called a
\textit{localized spectrum slicing (LSS) basis set}. The LSS basis set
can be constructed without explicitly constructing the LSS operator.
The generalized eigenvalue problem for the sparse projected matrices
$A_{U},B_{U}$ may be solved both by direct methods, but also by 
methods using sparse linear algebra techniques.  During the construction
of the LSS operator and/or the LSS basis set, a good global preconditioner for
$(A-\mu I)^{-1}$ is not needed. We demonstrate the construction of the
LSS basis set and its use for solving interior eigenvalues problems for
matrices obtained from discretizing second order partial differential
operators, and find that the use of the LSS basis set can be more
efficient than solving the global problem directly for matrices of large
sizes. \REV{We also apply the LSS method to a general matrix from the
University of Florida matrix collection~\cite{FloridaMatrix}.}

%Here we can chosoe $\mu = \frac{a+b}{2}$, and $\sigma$ is large enough
%so that $f(a),f(b)$ still have significant value.  This
%paper deals with the sparsification of the spanning vectors of the
%spectral projection operators by augmenting them with some additional
%vectors, and sparse representation can be found.  Such sparse
%representation allows
%
%1.  Local computation for the basis of the spanning vectors.
%
%2.  Sparse global matrix allowing fast algorithms compared to existing
%method.
%
%We follow a two step procedure.  We first establish that the decay
%properties of the matrix $f_{\sigma,\mu}(A)$, so that the matrix is
%sparse after truncation.  Then we explore the lowrankness and introduce
%the concept of interpolative decomposition.  This is too expensive. We
%first compute a superset of the
%actual interpolative decomposition is performed by computing the
%localized columns of $f_{\sigma,\mu}(A)$ through local computation. The
%linearly dependent modes are then eliminated through a global rank
%revealing $QR$ factorization.  The complexity of this implementation is
%XXX.  The optimal implementation of this is XXX. Finally we use this as
%a preconditioner.

\subsection{Related work}

The spectral locality of the LSS operator is valid by construction.
Comparatively the spatial locality of the LSS operator is less obvious,
and is given more precisely by the \textit{decay properties} of matrix
functions that are analytic in a certain region in the complex plane
(see e.g.~\cite{BenziBoitoRazouk2013,BenziGolub1999,BenziRazouk2007}).
The decay properties of matrix functions were first realized for matrix
inverse 
$A^{-1}$ (i.e. $f(z)=z^{-1}$), where $A$ is a banded, positive definite
matrix~\cite{Demko1977,DemkoMossSmith1984}.  The method for showing
decay properties relies on whether $f(z)$ can be well approximated by a
low order Chebyshev polynomial evaluated at the eigenvalues of $A$, and
this method is therefore generalizable to any analytic function $f(z)$ for
banded matrices $A$.  In order to generalize from banded
matrices to general sparse matrices, decay properties should be
defined using geodesic distances of the graph induced by $A$. These
techniques have been shown
in~\cite{BenziBoitoRazouk2013,BenziRazouk2007} and references therein,
for demonstrating the decay properties of e.g. Fermi-Dirac operators in
electronic structure theory. These techniques are directly used for
showing the decay properties of the LSS operator in this work, which
then allows the construction of the divide-and-conquer method. In physics
literature, such decay property is dubbed ``near-sightedness property''
and is vastly studied using various
models (see e.g.~\cite{Kohn1996,Nenciu1983,ProdanKohn2005}). The decay
property
is also used for constructing linear scaling
algorithms~\cite{BowlerMiyazaki2012,Goedecker1999} for density
functional theory calculations.

\subsection{Contents}
The rest of this paper is organized as follows. We introduce the decay properties of matrix functions and in
particular the localized spectrum slicing operator in
section~\ref{sec:prelim}.  Based on the decay properties,
section~\ref{sec:lss} describes a divide-and-conquer algorithm for
constructing the LSS operator and the LSS basis set, and provides the
error bound and
computational complexity. 
%Section~\ref{sec:interior} demonstrates the
%use of the LSS basis set for solving interior eigenvalue
%problems. 
\REV{The computation of interior eigenvalues and a domain
partitioning strategy for general sparse matrices are also discussed.
We demonstrate numerical results using the LSS basis set
for solving interior eigenvalue problems in section~\ref{sec:numer},} and discuss
the conclusion and future work in section~\ref{sec:conclusion}.

%the localized 
%
%spectrum slicing method is introduced
%
%
%provide the
%definition 

%In this section we describe the procedure for finding localized spectrum
%slicing.   In section~\ref{subsec:slicing} we introduce the spectrum
%slicing.  Section~\ref{subsec:decay} describes the decay property of the
%Gaussian function spectrally and plays an important role in the anslysi.
%Section~\ref{subsec:localized} describes the algorithm for constructing basis
%functions thart are both localized spatially and spectrally.  In
%section~\ref{subsec:precond} we demonstrate an algorithm for the
%localized spectral slicing for constructing a preconditioner to
%accelerate the solution of indefinite systems.

\section{Preliminaries}\label{sec:prelim}

\subsection{Notation}

The $(i,j)$-th element of a matrix $A\in \mathbb{C}^{n\times n}$ is
denoted by $A_{ij}$.  The submatrix of $A$ corresponding to a set
of row indices $\mc{I}$ and a set of column indices $\mc{J}$ is denoted
by $A_{\mc{I},\mc{J}}$. Using MATLAB
notation, all elements in the $i$-th row of $A$ are denoted by
$A_{i,:}$, and all elements in a set of rows $\mc{I}$ are denoted by
$A_{\mc{I},:}$.  Similarly, all elements in the $j$-th column of $A$ are
denoted by $A_{:,j}$, and all elements for a set of columns $\mc{J}$ are
denoted by $A_{:,\mc{J}}$.  The $k$-th power of $A$ is denoted by
$A^{k}$.  The matrix $p$-norm of $A$ is denoted by $\norm{A}_{p}$, 
and the vector $p$-norm of a vector $u$ is denoted by $\norm{u}_{p}$
($p\ge 1$).  The max norm of a matrix is denoted by
$\norm{A}_{\max}\equiv \max_{i,j}\{\abs{A_{ij}}\}$, which is the same as the
$\infty$-norm of a vector of length $n^2$, formed by all the elements of
$A$.
The Hermitian conjugate of $A$ is denoted by $A^{*}$.
Depending on the context, we may also refer to a matrix as an
\textit{operator}.

A Hermitian matrix $A$ induces an undirected graph $G=(\mc{V},\mc{E})$ with
$\mc{V}=\{i\vert i=1,\ldots,n\}$, and
$\mc{E}=\{(i,j)\vert A_{ij}\ne 0, \quad 1\le i,j\le n\}$. Each
element in $\mc{V}$ is called a vertex, and each element in
$\mc{E}$ is called an edge. The cardinality of a set of indices $\mc{I}$
is denoted by $\abs{\mc{I}}$.

A Hermitian matrix $A\in \mathbb{C}^{n\times n}$ has the eigen-decomposition
\begin{equation}
  AX=X\Lambda.
  \label{eqn:eigenA}
\end{equation}
Here $\Lambda=\mathrm{diag}[\lambda_{1},\ldots,\lambda_{n}]$ is a
diagonal matrix containing the (real) eigenvalues of $A$ and we assume
$\lambda_{1}\le \lambda_{2}\le \cdots\le \lambda_{n}$ are ordered
non-decreasingly.  $X=[x_{1},\ldots,x_{n}]$ and $x_{i}$ is the
eigenvector corresponding to the eigenvalue $\lambda_{i}$.  
If all eigenvalues (and corresponding eigenvectors) to be computed are
with in a small interval $(\mu-c,\mu+c)$ on the real line with
$\lambda_{1}<\mu-c<\mu+c<\lambda_{n}$, then this problem is called an
interior eigenvalue problem.

%We call a vector $u\in \mathbb{C}^{n}$ \textit{spectrally localized} if
%there exists a small number $\varepsilon>0$ and a set of orthonormal
%eigenvectors of $\wt{X}$ of $A$ corresponding to eigenvalues of $A$ near
%$\mu$ so that
%\[
%\norm{u-\wt{X}\wt{X}^{*}u}_{2} \le \varepsilon,
%\]
%i.e. $u$ is approximately spanned by the eigenvectors of $A$
%corresponding to a narrow range of eigenvalues. 
%Similarly, we call a matrix $B\in \mathbb{C}^{n\times n}$ spectrally
%localized if
%\[
%\norm{B-\wt{X}\wt{X}^{*}B}_{2} \le \varepsilon,
%\]
%
%We call a vector $u$ spatially localized if there exists a small
%number $\varepsilon>0$ and a vector $\wt{u}\in \mathbb{C}^{n}$ so that
%$(\wt{u})_{i}$ is only nonzero for a small number of entries $i$, and
%\[
%\norm{u-\wt{u}}_{2} \le \varepsilon.
%\]
%Similarly, we call a matrix $B\in \mathbb{C}^{n\times n}$ spatially
%localized if there exists a matrix $\wt{B}\in \mathbb{C}^{n\times n}$ so
%that the $(\wt{B})_{ij}$ is only nonzero for a small
%number of entries $(i,j)$, such that
%\[
%\norm{B-\wt{B}}_{2} \le \varepsilon.
%\]
%A more detailed description of spatial locality is given by the decay
%properties of matrices given in section~\ref{subsec:decay}.

\subsection{Decay property of matrix functions}\label{subsec:decay}

In this section, we provide a short but self-contained description of
the decay properties of $f_{\sigma,\mu}(A)$. More details on the
description of the decay properties of general matrix functions can be found
in~\cite{BenziBoitoRazouk2013} and references therein.

%The localized spectrum slicing is based on the decay properties of the
%matrix $f_{\sigma,\mu}(A)$ along the off-diagonal directions.  Here we
%provide a self-contained demonstration of such decay properties, and the
%techniques originate from~\cite{BenziSIAMRev} and references therein.

Let $k$ be a non-negative integer, and $\mathbb{P}_{k}$ be the set of
all polynomials of degrees less than or equal to $k$ with real
coefficients. Without loss of generality we assume the eigenvalues of
$A$ are within the interval $(-1,1)$. For a real continuous function $f$ on $[-1,1]$, the best
approximation error is defined as
\begin{equation}
  E_{k}(f) = \REV{\min_{p\in \mathbb{P}_{k}}}\left\{ \norm{f-p}_{\infty}\equiv
  \max_{-1\le x\le 1}\abs{f(x)-p(x)}\right\}.
  \label{eqn:approxerror}
\end{equation}
Consider an ellipse in the complex plane $\mathbb{C}$ with foci in $-1$
and $1$, and $a>1,b>0$ be the half axes so that the vertices of the
ellipse are $a,-a, ib, -ib$, respectively. Let the sum of the half axes
be $\chi=a+b$, then using the identity $a^2-b^2=1$ we have
\[
a = \frac{\chi^2+1}{2\chi}, \quad b = \frac{\chi^2-1}{2\chi}.
\]
Thus the ellipse is determined only by $\chi$, and such ellipse is
denoted by $\mathcal{E}_{\chi}$.
Then Bernstein's theorem~\cite{MeinardusSchumaker1967} is stated  as follows.
\begin{theorem}[Bernstein]
  Let $f(z)$ be analytic in $\mathcal{E}_{\chi}$ with $\chi>1$, and
  $f(z)$ is a real valued function for real $z$. Then
  \begin{equation}
    E_{k}(f) \le \frac{2 M(\chi)}{\chi^{k} (\chi-1)},
    \label{eqn:bernstein}
  \end{equation}
  where
  \begin{equation}
    M(\chi) = \sup_{z\in \mathcal{E}_{\chi}}\abs{f(z)}.
    \label{eqn:Mchi}
  \end{equation}
  \label{thm:bernstein}
\end{theorem}

Using Theorem~\ref{thm:bernstein}, a more quantitative description of
the approximation properties for $f_{\sigma,\mu}(z)$ in
Eq.~\eqref{eqn:gaussian} is given in
Theorem~\ref{thm:approxGaussian}.
\begin{theorem}
  Let $f_{\sigma,\mu}(z)$ be a Gaussian function defined in
  Eq.~\eqref{eqn:gaussian}, then for any $\alpha>0$,
  \begin{equation}
    E_{k}(f_{\sigma,\mu)} \le 
    \frac{2}{\alpha\sigma} e^{\alpha^2}
    (1+\alpha\sigma)^{-k}.
    \label{eqn:approxGaussian}
  \end{equation}
  \label{thm:approxGaussian}
\end{theorem}
\begin{proof}
  For any $\mu\in (-1,1), \sigma>0$, the Gaussian function
  $f_{\sigma,\mu}$ is analytic in any ellipse $\mathcal{E}_{\chi}$ with
  $\chi>1$, then
  \[
  M(\chi)=\sup_{z\equiv x+iy\in \mathcal{E}_{\chi}}\abs{f_{\sigma,\mu}(x+iy)}  \le 
  \sup_{z\equiv x+iy\in \mathcal{E}_{\chi}} 
  e^{\frac{y^2}{\sigma^2}}\le 
  e^{\frac{(\chi-\frac{1}{\chi})^2}{4\sigma^2}}.
  \]
  For any $\alpha>0$, let 
  \begin{equation}
    \chi=1+\alpha \sigma
    \label{eqn:chichoice}
  \end{equation}
  then
  $\chi-\frac{1}{\chi} \le 2\alpha \sigma$, and
  \begin{equation}
    M(1+\alpha \sigma) \le e^{\alpha^2}.
    \label{eqn:Mbound}
  \end{equation}
  Using Theorem~\ref{thm:bernstein},
  Eq.~\eqref{eqn:approxGaussian} is the direct consequence of
  Eq.~\eqref{eqn:Mbound} and the choice of $\chi$ in
  Eq.~\eqref{eqn:chichoice}.
\end{proof}

For the graph $G=(\mc{V},\mc{E})$ associated with the matrix $A$ and
vertices $i,j\in\mc{V}$, a path linking $i,j$ is given by a sequence of
edges $p=\{(i_{0}\equiv
i,i_{1}),(i_{1},i_{2}),\ldots,(i_{l},i_{l+1}\equiv j)\}$ where
$i_{1},\ldots,i_{l}\in \mc{V}$, and each element in $p$ is an edge
in $\mc{E}$.  The length of the path $p$ is defined to be $l+1$.
If $p=\{(i,j)\}$ then the length of $p$ is $1$.  
The \textit{geodesic distance} $d(i,j)$ between vertices $i$ and $j$ is
defined as
the length of the shortest path between $i$ and $j$. 
It should be noted that for
structurally symmetric matrices, i.e. $A_{ij}\ne 0$ implies
$A_{ji}\ne 0$ for all indices $i,j$, the geodesic distance is symmetric,
i.e. $d(i,j)=d(j,i)$.  In particular, Hermitian matrices are
structurally symmetric.
If $d(i,j)>1$, then $A_{ij}=0$.  
If $d(i,j)=\infty$ then there is no path connecting $i$ and $j$.  
%This statement can be straightforwardly generalized that 
\REV{More generally, for any
positive integer $k$, if $d(i,j)>k$ then $(A^{k})_{ij}=0$, where $A^{k}$
is the $k$-th power of the matrix $A$.}

%as in Proposition~\ref{prop:Akpower}. The
%
%\begin{proposition}\label{prop:Akpower}
%  Let $A$ be a sparse and Hermitian matrix.  For any
%  positive integer $k$, if $d(i,j)>k$ then $(A^{k})_{ij}=0$, where
%  $A^{k}$ is the $k$-th power of the matrix $A$.
%\end{proposition}
%\begin{proof}
%  The statement is obviously correct for $k=1$. Assume the statement holds
%  for $k-1$, then if $d(i,j)>k$, 
%  \[
%  (A^{k})_{ij} = \sum_{p} A_{ip} (A^{k-1})_{pj}.
%  \]
%  Each term in the summation is nonzero only if $A_{ip}\ne 0$, i.e.
%  $d(i,p)=1$.  Since $d(i,j)>k$, for such $p$ we have $d(p,j)>k-1$, and we have
%  $(A^{k-1})_{pj} = 0$. Therefore $(A^{k})_{ij}=0$ and the statement holds
%  for $k$.
%\end{proof}

%The decay properties of the off-diagonal elements of the matrix
%function $f_{\sigma,\mu}(A)$ is given as follows.
The precise statement of the spatial locality of the matrix function
$f_{\sigma,\mu}(A)$ is given by the decay properties along the
off-diagonal direction in Theorem~\ref{thm:decaygauss}. For a given
column $j$, the magnitude of each element $f_{\sigma,\mu}(A)_{i,j}$
decays exponentially with respect to the geodesic distance $d(i,j)$.

\begin{theorem}
  Let $A$ be a sparse and Hermitian matrix with all eigenvalues
  contained in the interval $(-1,1)$.  For any $\alpha>0,\sigma>0$, let
  \begin{equation}
    \rho=(1+\alpha\sigma)^{-1},\quad 
    K=\frac{2}{\rho\alpha\sigma} e^{\alpha^2}, 
    \label{eqn:Krho}
  \end{equation}
  then for all $d(i,j)\ge 1, i,j=1,\cdots,n$,
  \begin{equation}
    \abs{f_{\sigma,\mu}(A)_{ij}} \le K \rho^{d(i,j)}, 
    \label{eqn:decaygauss}
  \end{equation}
  where $d(i,j)$ is the geodesic distance between vertices $i$ and
  $j$.
  \label{thm:decaygauss}
\end{theorem}
\begin{proof}
  For any integer $k\ge 0$, there exists a 
  polynomial $p_{k}\in \mathbb{P}_{k}$ such that
  \[
  \norm{f_{\sigma,\mu}(A)-p_{k}(A)}_{2} =
  \norm{f_{\sigma,\mu}-p_{k}}_{\infty} = E_{k}(f_{\sigma,\mu}) \le K
  \rho^{k+1}.
  \]
  \REV{The last inequality follows from Theorem~\ref{thm:approxGaussian}.}
  Now consider all edges $(i,j)$ such that the geodesic distance
  $d(i,j)=k+1$, and then $p_{k}(A)_{ij}=0$.  Therefore
  \[
  \abs{f_{\sigma,\mu}(A)_{ij}} =
  \abs{f_{\sigma,\mu}(A)_{ij}-p_{k}(A)_{ij}} \le
  \norm{f_{\sigma,\mu}(A)-p_{k}(A)}_{2} \le  K \rho^{k+1} = K
  \rho^{d(i,j)}.
  \]
\end{proof}

%\begin{remark}
%  The treatment using
%  geodesic distance is a direct generalization of results for banded
%  matrices~\cite{Demko1977,BenziSeries}. The concept of banded matrix is
%  useful for e.g. discretized PDEs in one-dimensional space, but less
%  useful for higher dimensional case.  
%\end{remark}

\begin{remark}
  As suggested in Eq.~\eqref{eqn:Krho}, $\rho,K$ only depend on
  $\sigma$ but not on $\mu$. Therefore the decay properties of the
  matrix function $\abs{f_{\sigma,\mu}(A)}$ seem to be independent of the
  shift $\mu$.  This is because an upper bound for $M(\chi)$ is given in
  Theorem~\ref{thm:approxGaussian} that is valid for all $\mu$.
  Numerical results in section~\ref{sec:numer} indicate that the
  preconstant of the exponential decay may have a strong dependency on
  $\mu$, \REV{and such dependency can be specific to the matrix under
  study.}
\end{remark}

\begin{remark}
  In Theorem~\ref{thm:decaygauss} there is an arbitrary positive
  constant $\alpha$.
  For any given $\alpha>0$, the
  off-diagonal entries of $\abs{f_{\sigma,\mu}(A)_{ij}}$ should decay
  exponentially with respect to the geodesic distance.  By optimizing
  $\alpha$ together with the degree of the Chebyshev polynomial
  $k$, the actual decay rate can be slightly faster than exponential.
  Fig.~\ref{fig:superexpdecay} gives an example of the magnitude of the first column 
  $\abs{f_{\sigma,\mu}(A)_{:,1}}$ where $A$ is a discretized Laplacian
  operator in 1D with periodic boundary conditions, with
  $\sigma=1.0,\mu=2.0$ and $\sigma=1.0,\mu=10.0$ respectively.
  \REV{Although the discretized 1D Laplacian matrix is a banded matrix, 
  all its eigenfunctions are plane waves which are fully delocalized
  in the global domain. Nonetheless the upper bound of the decay rate of
  the LSS operator is clearly exponential as shown in 
  Fig.~\ref{fig:superexpdecay}.}
  \begin{figure}[h]
    \begin{center}
      \includegraphics[width=0.4\textwidth]{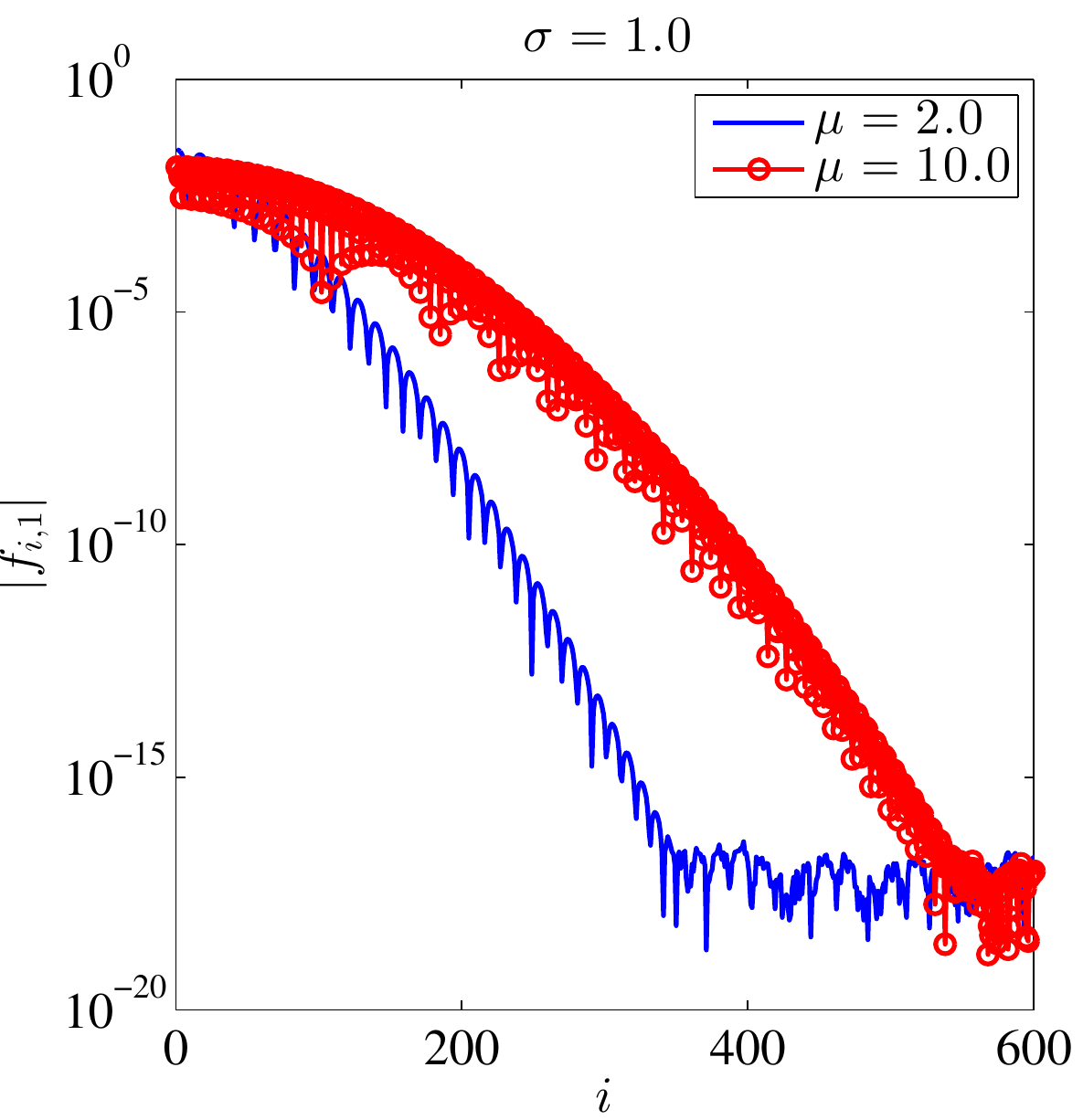}
    \end{center}
    \caption{Log-scale plot of the magnitude of the first column
    $\abs{f_{\sigma,\mu}(A)_{:,1}}$. $A$ is a discretized Laplacian
    operator in 1D with periodic boundary conditions, with $\sigma=1.0$
    and $\mu=2.0,1.0$, respectively.}
    \label{fig:superexpdecay}
  \end{figure}
\end{remark}

\begin{remark}
\REV{In order to limit the numerical rank of 
$f_{\sigma,\mu}(A)$ in practice, it is desirable to use a small $\sigma$.
With fixed $\alpha$ and assume $\alpha\sigma < 1$, we have
\[
\rho^{d(i,j)} = (1+\alpha\sigma)^{-d(i,j)} \le e^{-\frac12 \alpha\sigma
d(i,j)}.
\]
}
Here $\sigma$ reflects the spectral locality, and
$d(i,j)$ reflects the spatial locality, \REV{which} reveals the balance
between the spectral and spatial locality, tuned by one parameter
$\sigma$.
\end{remark}

%decay properties spectrally and spatially, which are
%inversely proportional to each other.  The number of basis functions
%needed to compute $f_{\sigma,\mu}(A)$ is related to the rank of the
%matrix function, or the number of significant sigular vectors of
%$f_{\sigma,\mu}(A)$.  This is not necessarily related to the spatial
%smoothness of the singular vectors.  This provides opportunity to
%provide a small number of localized basis functions to span the range of
%the operator.

\section{Localized spectrum slicing}\label{sec:lss}

\subsection{Algorithm}

%If we take $P$ to be the matrix $f_{\sigma,\mu}(A)$, and denote by the
%rank of the matrix $k$. Even if $k\ll n$, it would be prohibitively
%expensive to compute the rank-revealing $QR$ decomposition to $P$.  This
%procedure requires the storage of the entire matrix $P$ which scales as
%$n^2$, and the computational cost for the $QR$ decomposition is
%$\Or(n^2 k)$.   It has recently been demonstrated that for the case
%where $P$ is explicitly given as a rank-$k$ matrix and the span of the
%singular vectors is known, then a technique called the \textit{selected
%columns of the density matrix (SCDM)} is developed for finding localized
%representation of the matrix.  The difficulty of applying the techinque
%to the present context is that the spanning columns are not known a
%priori, and therefore the new mehtod is needed.

Using the decay properties of the LSS operator $f_{\sigma,\mu}(A)$ in
Theorem~\ref{thm:decaygauss}, a set of basis functions called the LSS
basis set can be constructed in a divide-and-conquer fashion.
%This is possible because the locality of $f_{\sigma,\mu}(A)$ implies
%the columns can be computed accurately even if we modify $A$ far away
%from the columns of interest.  Using the localization results in
%section~\ref{sec:}, using Theorem~\ref{thm:decaygauss}, each column of
%this matrix is exponentially localized, and therefore can be computed
%via local operation.  
Below we demonstrate that  if the smearing parameter $\sigma$ is large
enough, then the localized spectrum slicing operator $f_{\sigma,\mu(A)}$
can be approximately computed using submatrices of $A$. The size of each
submatrix is independent of the size of $A$.
This is important for reducing the computational complexity and for
parallel computation. 

%The basic strategy is a two step procedure: first, we construct a subset of columns of $f_{\sigma,\mu}(A)$ following a divide and conquer
%strategy. This is because as $\sigma$ is large enough, \textit{all}
%columns of the matrix should be localized and it is possible to find a
%spanning set of columns within each element \LL{define element here}.
%The columns generated in such way can be well conditioned locally, but
%poorly conditioned in the global level.  This is because of the fact
%that the columns selected from one element do not know what columns are
%selected in other elements.  Therefore the columns, especially near the
%edge of the element can share many similar information and the global
%matrix can be ill-conditioned or even singular. The number of columns
%taken at the first step is not entirely crucial, because a second step
%will later be applied on the global level to eliminate the
%linearly-dependent columns.

\begin{proposition}
  Let $A$, $B$ be $n\times n$ Hermitian matrices.  The graph
  induced by $B$ is a spanning subgraph of the graph $G$ induced by $A$, and the
  geodesic distance $d(i,j)$ is defined using the graph $G$. We assume
  for a given integer $j,m$ ($1\le j,m\le n)$, 
  \[
  A_{il} = B_{il}, \quad \forall i,l  \quad \mathrm{s.t.}
  \quad d(j,i)\le m, \quad d(j,l) \le m .
  \]
  Then for any integer $k$ ($1\le k\le m$), 
  \[
  (A^{k})_{il} = (B^{k})_{il}, \quad \forall i,l \quad \mathrm{s.t.}
  \quad d(j,i)\le m-k+1, \quad d(j,l)
  \le m-k+1.
  \]
  \label{prop:product}
\end{proposition}
\begin{proof}
  The statement is apparently correct for $k=1$.  Assume
  the statement for $k-1$ is proved, \REV{and we prove the statement is
  true for $k$. First}
  \[
  (A^{k})_{il} = \sum_{p} A_{ip} (A^{k-1})_{pl},\quad
  (B^{k})_{il} = \sum_{p} B_{ip} (B^{k-1})_{pl}.
  \]
  In the summation above, $A_{ip} (A^{k-1})_{pl}$ is nonzero only if
  $A_{ip}\ne 0$. Similarly $B_{ip} (B^{k-1})_{pl}$ is nonzero only if
  $B_{ip}\ne 0$.
  Since the graph induced by $B$ is a subgraph of the graph induced by
  $A$, $B_{ip}\ne 0$ implies $A_{ip}\ne 0$, and therefore we only need
  to consider $p$ such that $A_{ip}\ne 0$, i.e. $d(i,p)=1$.
  Consider $i$ such that
  $d(j,i)\le m-k+1$, then $p$
  satisfies
  \[
  d(j,p)\le d(j,i)+d(i,p)=m-k+2 = m-(k-1)+1.
  \]
  Also \REV{for any $l$ such that $d(j,l)\le m-k+1 <
  m-(k-1)+1$, by the assumption that the statement for $k-1$ is proved,
  $(A^{k-1})_{pl} = (B^{k-1})_{pl}$.    
  Together with $d(j,i)\le m, d(j,p)\le m$, we have
  $A_{ip}=B_{ip}$.}
  Therefore
  \[
  (A^{k})_{il} = \sum_{p} A_{ip} (A^{k-1})_{pl} = \sum_{p} B_{ip} (B^{k-1})_{pl} = (B^{k})_{il}
  \]
  is valid for all $i,l$ such that  $d(j,i)\le m-k+1, d(j,l) \le m-k+1$.
\end{proof}

Using Proposition~\ref{prop:product}, Theorem~\ref{thm:truncate} shows
that the $j$-th column of
$f_{\sigma,\mu}(A)$ can be accurately computed from $f_{\sigma,\mu}(B)$, as long as $A$ and $B$
are sufficiently close in the vicinity of $j$ in the sense  of
geodesic distance.

\begin{theorem}
  Let $A$, $B$ be $n\times n$ Hermitian matrices with eigenvalues in
  $(-1,1)$.  For a given
  $j$ and an even integer $m$ ($1\le j,m\le n$), 
  \[
  A_{il} = B_{il}, \quad \forall i,l \quad \mathrm{s.t.} \quad d(j,i)\le
  m, \quad d(j,l) \le m .
  \]
  Then 
  \[
  \abs{f_{\sigma,\mu}(A)_{ij}-f_{\sigma,\mu}(B)_{ij}} \le  2K
  \rho^{\frac{m}{2}+1},
  \]
  for all $i$ such that $d(j,i)\le m/2+1$, where the constants
  $K,\rho$ are given in Eq.~\eqref{eqn:Krho}.
  \label{thm:truncate}
\end{theorem}
\begin{proof}
  \REV{For any $i,j$ and $k\ge 0$ we have
  \[
  \abs{f_{\sigma,\mu}(A)_{ij}-f_{\sigma,\mu}(B)_{ij}} \le 
  \abs{f_{\sigma,\mu}(A)_{ij} - p_{k}(A)_{ij}} + 
  \abs{p_{k}(A)_{ij} - p_{k}(B)_{ij}}+ 
  \abs{f_{\sigma,\mu}(B)_{ij} - p_{k}(B)_{ij}}.
  \]
  Take $k=\frac{m}{2}$.  For any $i$ such $d(j,i)\le m-k+1 =
  \frac{m}{2}+1$, by
  Proposition~\ref{prop:product}, $p_{k}(A)_{ij} = p_{k}(B)_{ij}$. Also from
  Theorem~\ref{thm:approxGaussian}, we have
  \[
  \abs{f_{\sigma,\mu}(A)_{ij} - p_{k}(A)_{ij}} \le \norm{f_{\sigma,\mu}(A)-p_{k}(A)}_{2}
  \le K \rho^{\frac{m}{2}+1},
  \]
  \[
  \abs{f_{\sigma,\mu}(B)_{ij} - p_{k}(B)_{ij}} \le \norm{f_{\sigma,\mu}(B)-p_{k}(B)}_{2}
  \le K \rho^{\frac{m}{2}+1},
  \]
  and hence the result.
%  \[
%  \abs{f_{\sigma,\mu}(A)_{ij}-f_{\sigma,\mu}(B)_{ij}} \le 2K
%  \rho^{\frac{m}{2}+1}.
%  \]
  }
\end{proof}

%If we collect the columns from all elements $\wt{\CS}=\cup_{\kappa\in \TS}  
%\wt{\CS}_{\kappa}$, and denote by the columns $P_{:,\wt{\CS}}$, 
%then the way to overcome the ill-conditioned problem of the matrix 
%$P_{:,\wt{\CS}}$ is to perform another
%interpolative decomposition for the matrix.  Note that each column of
%$P_{:,\wt{\CS}}$ is compactly supported on the extended element
%$\wt{\kappa}$, and interpolative decomposition can in principle be done
%more efficiently, such as following the tournament pivoting
%procedure~\cite{DemmelGrigoriGuEtal} for parallel rank-revealing QR
%decomposition.  However, here we simply perform a QR decomposition by
%assuming $P_{:,\wt{\CS}}$ as a dense matrix as a proof of principle,
%i.e.
%\begin{equation}
%  P_{:,\wt{\CS}} \Pi = P_{:,\CS}[I T].
%\end{equation}
%Again we use the RRQR for the matrix $P_{:,\wt{\CS}}$, and we only take
%the singificant columns.  The basis functions $P_{:,\CS}$ is the set of
%spanning columns for the operator $f_{\sigma,\mu}(A)$. In fact if needed
%$f_{\sigma,\mu}(A)$ can be reconstructed  as
%\[
%f_{\sigma,\mu}(A) = P_{:,\CS} D P_{:,\CS}^{*},
%\]
%\LL{The diagonal matrix D is to be determined.}
%

Theorem~\ref{thm:truncate} shows that in order to compute any column
$j$ of the matrix $f_{\sigma,\mu}(A)$ up to certain accuracy, it is
only necessary to have a matrix that is the same as $A$ up to a certain
distance away from $j$.  Together with the decay property of each
column of $f_{\sigma,\mu}(A)$, this allows the $j$-th column of
$f_{\sigma,\mu}(A)$ to be constructed in a \textit{divide and conquer}
manner.  For instance, for a given integer $m$ we can define 
\begin{equation}
B_{il} = \begin{cases}
  A_{il}, &\forall i,l \quad \mathrm{s.t.} \quad d(j,i)\le m, d(j,l) \le
  m,\\
  0, & \mathrm{otherwise}
\end{cases}.
\label{eqn:dirichletmodification}
\end{equation}
which is simply a submatrix of $A$. As a submatrix, $\norm{B}_{2}\le
\norm{A}_{2}$ and the assumption of the spectral radius in
Theorem~\ref{thm:truncate} is satisfied.

In practice it would be very time consuming to construct an approximate
matrix for each column of $j$, \REV{since the rank of the LSS operator
$f_{\sigma,\mu}(A)$ is often much less compared to $n$. For structured
matrices such as matrices obtained from finite difference or finite
element discretization of PDE operators, it is often possible to
partition the domain into well structured disjoint columns sets, and
apply the truncated matrix to each column set. The cost for generating
such partition can be very small if the structure of the matrix is known
\textit{a priori}. For the discussion below, we assume that the
partition $\mc{V}=\{1,\ldots,n\}$ into $M$ simply
connected disjoint sets $\{E_{\kappa}\}_{\kappa=1}^{M}$ is given, i.e.
\[
\mc{V} = \bigcup_{\kappa=1}^{M}  E_{\kappa}, \quad \mbox{and}\quad
E_{\kappa}\bigcap
E_{\kappa'} = \emptyset,\quad \kappa\ne \kappa'.
\]
For general sparse matrices, such partition may not be readily
available. We discuss the choice of domain partitioning strategy in
section~\ref{subsec:lssgeneral}.}

For each $E_{\kappa}$ and an integer $m$, we
define an associated set
\begin{equation}
  Q_{\kappa} = \left\{ i \vert d(i,j)\le m, \forall j\in E_{\kappa}
  \right\}.
  \label{eqn:Qkappa}
\end{equation}
Theorem~\ref{thm:truncate} implies that the
submatrix $(f_{\sigma,\mu}(A))_{:,E_{\kappa}}$ can be constructed by
a submatrix of $A$ defined as 
\begin{equation}
  (A_{\kappa})_{ij} = \begin{cases}
    A_{ij}, & i,j\in Q_{\kappa},\\
    0,& \mathrm{otherwise}.
  \end{cases}
  \label{eqn:Akappa}
\end{equation}
In the following discussion, we refer to $E_{\kappa}$ as an
\textit{element}, and to $Q_{\kappa}$ as an \textit{extended element}
associated with $E_{\kappa}$. \REV{It should be noted that the zero
entries of $A_{\kappa}$ outside the index set $Q_{\kappa}$ do not need
to be explicitly stored.}

\begin{remark}
  The choice in Eq.~\eqref{eqn:dirichletmodification} takes a submatrix
  of $A$ to compute the localized spectrum slicing operator.  From the
  point of view of partial differential operators, this is similar to
  imposing zero Dirichlet boundary condition on some local domains.
\end{remark}

Since $A$ is Hermitian and sparse, and so is $A_{\kappa}$, and the
latter has the
eigen-decomposition
\begin{equation}
  A_{\kappa} X_{\kappa} = X_{\kappa} D_{\kappa}.
  \label{eqn:eigenAkappa}
\end{equation}
Here $D_{\kappa}$ is a diagonal matrix.
Note that $A_{\kappa}$ only takes nonzero values on the extended
element $Q_{\kappa}$. \REV{The entries of each column of $X_{\kappa}$
outside the index set $Q_{\kappa}$ can be set to zero, and such zero
entries do not need to be explicitly stored. This is equivalent to
solving an eigenvalue problem of size $|Q_{\kappa}|\times |Q_{\kappa}|$.  }
%\REV{Eq.~\eqref{eqn:eigenAkappa} does not need to
%be solved as an $n\times n$ eigenvalue problem. Define
%$A_{\kappa}^{Q_{\kappa}}$ as the $n_{\kappa}\times n_{\kappa}$ submatrix
%consisting of the nonzero entries of $A_{\kappa}$ restricted to
%$Q_{\kappa}$ with corresponding eigenvalue problem
%\[
%A_{\kappa}^{Q_{\kappa}} X^{Q_{\kappa}}_{\kappa} = 
%X^{Q_{\kappa}}_{\kappa} D_{\kappa},
%\]
%then $X_{\kappa}$ can be obtained by padding $X^{Q_{\kappa}}_{\kappa}$}
%
%therefore the 
%eigen-decomposition~\eqref{eqn:eigenAkappa} only considers $X_{\kappa}$
%with nonzero values on
%$Q_{\kappa}$.  
Define
\begin{equation}
  f_{\sigma,\mu}(A_{\kappa}) \equiv X_{\kappa}
  f_{\sigma,\mu}(D_{\kappa})
  X_{\kappa}^{*}.
  \label{eqn:fAkappa}
\end{equation}
Using Theorem~\ref{thm:truncate}, $f_{\sigma,\mu}(A)_{Q_{\kappa},E_{\kappa}}$ can be approximated by
$f_{\sigma,\mu}(A_{\kappa})$, in the sense that
\[
\abs{f_{\sigma,\mu}(A)_{ij}-f_{\sigma,\mu}(A_{\kappa})_{ij}} \le
2K\rho^{\frac{m}{2}+1}, \quad \forall i\in Q_{\kappa},\quad j\in E_{\kappa}.
\]
\REV{Since $f_{\sigma,\mu}$ is spectrally localized,}
in practice not all
eigenvalues and eigenvectors of $A_{\kappa}$ as
in~\eqref{eqn:eigenAkappa} are needed. Instead only a \textit{partial
eigen-decomposition} is needed to compute all eigenvalues of
$A_{\kappa}$ in the interval $(\mu-c\sigma,\mu+c\sigma)$.  Due to the
fast decay properties of Gaussian functions, in practice $c$ can be
chosen to be $2\sim 4$ to be sufficiently accurate.  \REV{We denote by
$s_{\kappa}$ the column dimension of $X_{\kappa}$ in the partial
eigen-decomposition of $A_{\kappa}$.}

The factorized representation in Eq.~\eqref{eqn:fAkappa} also allows the
computation of a set of vectors approximately spanning the column space
of $f_{\sigma,\mu}(A_{\kappa})$, through a local singular value
decomposition (SVD) procedure, i.e.
\begin{equation}
  \norm{f_{\sigma,\mu}(D_{\kappa})
  ((X_{\kappa})_{Q_{\kappa},:})^{*} -
  \wt{U}_{\kappa} \wt{S}_{\kappa} \wt{V}_{\kappa}^{*}}_{2} \le
  \wt{\tau}.
  \label{eqn:localSVD}
\end{equation}
Here $\wt{\tau}$ is SVD truncation criterion. \REV{The size of the
matrix for the SVD
decomposition is $s_{\kappa} \times |Q_{\kappa}|$.}
In practice $\wt{\tau}$
may also be chosen using a relative criterion as $\wt{\tau}=\tau
(\wt{S}_{\kappa})_{1,1}$ in used in our numerical experiment, where we
assume $(\wt{S}_{\kappa})_{1,1}$ is the largest singular value in
Eq.~\eqref{eqn:localSVD}.
In practice this can
be performed by only keeping the singular values in the diagonal matrix
$\wt{S}_{\kappa}$ that are larger than $\wt{\tau}$.  Then we can define
\begin{equation}
  U_{\kappa} = X_{\kappa} \wt{U}_{\kappa}, \quad V_{\kappa} =
  \wt{S}_{\kappa} \wt{V}_{\kappa}^{*}.
  \label{eqn:UV}
\end{equation}
We combine all $U_{\kappa}$ together
\begin{equation}
  U\equiv [U_{1},\ldots,U_{M}],
  \label{eqn:LSSbasis}
\end{equation}
and $U$ is the LSS basis set that is both spectrally localized and
spatially localized.  We denote by $n_{b}$ the total number of
columns of $U$, which is also referred to as the size of the LSS basis
set.  Using the LSS basis set, an approximation to the LSS operator is
defined as
\begin{equation}
  \wt{f}_{ij} = \begin{cases}
    (U_{\kappa})_{i,:}(V_{\kappa})_{:,j},& i\in Q_{\kappa},j\in
    E_{\kappa}, \quad \mbox{for some~} \kappa,\\
    0,&i\notin Q_{\kappa},j\in E_{\kappa},\quad \mbox{otherwise}.
  \end{cases}
  \label{eqn:approxLSS}
\end{equation}
$\wt{f}$ is an $n\times n$ sparse matrix, and the error in the max norm for
approximating the LSS operator $f_{\sigma,\mu}(A)$ is given in
Theorem~\ref{thm:lss}.

\begin{theorem} 
  Let $A$ be an $n\times n$ Hermitian matrix with eigenvalues in
  $(-1,1)$, and the induced graph is
  partitioned into $M$ elements $\{E^{\kappa}\}$. For each element
  $E_{\kappa}$, there is an extended element $Q_{\kappa}$ given
  in~\eqref{eqn:Qkappa}, a submatrix $A_{\kappa}$ given in
  ~\eqref{eqn:Akappa},
  and matrices $U_{\kappa},V_{\kappa}$ satisfying
  \eqref{eqn:localSVD} and \eqref{eqn:UV}. Let $\wt{f}$ be an
  $n\times n$ matrix defined in Eq.~\eqref{eqn:approxLSS}, 
  then 
  \begin{equation}
    \norm{f_{\sigma,\mu}(A) - \wt{f}}_{\max} \le
    2K\rho^{\frac{m}{2}+1}+\wt{\tau}.  
    \label{eqn:maxerror}
  \end{equation}
  \label{thm:lss}
\end{theorem}
\begin{proof}
  For each element $\kappa$, from Eq.~\eqref{eqn:localSVD} we have
  \begin{equation}
    \begin{split}
      &\max_{i\in Q_{\kappa},j\in E_{\kappa}}
      \abs{f_{\sigma,\mu}(A_{\kappa})_{ij}-\wt{f}_{ij}}=
      \max_{i\in Q_{\kappa},j\in
      E_{\kappa}}\abs{f_{\sigma,\mu}(A_{\kappa})_{ij}-(U_{\kappa})_{i,:}(V_{\kappa})_{:,j}}\\
      \le &\norm{f_{\sigma,\mu}(A_{\kappa})-U_{\kappa}V_{\kappa}}_{2}
      \le \norm{X_{\kappa}}_{2} \wt{\tau} = \wt{\tau}.
    \end{split}
    \label{eqn:lssestimate1}
  \end{equation}
  Using Theorem~\ref{thm:truncate} and the definition of the extended
  element~\eqref{eqn:Qkappa}
  \begin{equation} 
    \max_{i\in Q_{\kappa},j\in E_{\kappa}}
    \abs{f_{\sigma,\mu}(A)_{ij}-f_{\sigma,\mu}(A_{\kappa})_{ij}} \le  2K
    \rho^{\frac{m}{2}+1}.
    \label{eqn:lssestimate2}
  \end{equation}
  For vertices $i\notin Q_{\kappa},j\in E_{\kappa}$,  $\wt{f}_{ij}=0$.
  Then from Theorem~\ref{thm:decaygauss} and use $\rho<1$
  \begin{equation} 
    \max_{i\notin Q_{\kappa},j\in E_{\kappa}}
    \abs{f_{\sigma,\mu}(A)_{ij}-\wt{f}_{ij}} =
    \abs{f_{\sigma,\mu}(A)_{ij}} \le K \rho^{m+1} \le K \rho^{m/2+1}.
    \label{eqn:lssestimate3}
  \end{equation}
  Combining
  Eqs.~\eqref{eqn:lssestimate1},\eqref{eqn:lssestimate2},\eqref{eqn:lssestimate3},
  we have
  \[
  \begin{split}
  &\norm{f_{\sigma,\mu}(A) - \wt{f}}_{\max} = 
  \max_{1\le i,j\le n}\abs{f_{\sigma,\mu}(A)_{ij} - \wt{f}_{ij}} \\
  =& \max_{\kappa} \left\{\max\left\{\max_{i\in Q_{\kappa},j\in
  E_{\kappa}}\abs{f_{\sigma,\mu}(A)_{ij} - \wt{f}_{ij}}, \max_{i\notin Q_{\kappa},j\in
  E_{\kappa}}\abs{f_{\sigma,\mu}(A)_{ij} - \wt{f}_{ij}}\right\}  \right\} \\
  =&
  \max_{\kappa}\left\{\max\{2K \rho^{\frac{m}{2}+1}+\wt{\tau}, K
  \rho^{m/2+1}\}  \right\} = 2K \rho^{\frac{m}{2}+1}+\wt{\tau}.  \end{split}
  \]
\end{proof}

\begin{remark}
Theorem~\ref{thm:lss} indicates that in order to accurately approximate
the LSS operator, the SVD truncation criterion $\wt{\tau}$ must be small
enough.  However, this may not necessarily be the case for approximating
interior eigenvalues. This will be discussed in
section~\ref{sec:numer}.
\end{remark}

Finally, we summarize the algorithm for finding the divide-and-conquer
method for constructing the LSS basis set in Algorithm~\ref{alg:LSS}.

\begin{algorithm}  
  \begin{small}
    \begin{center}
      \begin{minipage}{5in}
        \KwIn{\begin{tabular}{l} 
          (1) 
          \begin{minipage}[t]{4.0in} 
            Sparse Hermitian matrix $A$, center $\mu$, width $\sigma$, SVD
            truncation tolerance $\wt{\tau}$.
          \end{minipage}\\
          (2) 
          \begin{minipage}[t]{4.0in} 
            Number of elements $M$, partition of elements
            $\{E_{\kappa}\}_{\kappa=1}^{M}$ and extended elements
            $\{Q_{\kappa}\}_{\kappa=1}^{M}$.
          \end{minipage}
        \end{tabular}
        }
        \KwOut{\begin{tabular}{l} 
          \begin{minipage}[t]{4.0in} 
            LSS basis set $\{U_{\kappa}\}_{\kappa=1}^{M}$.
          \end{minipage}\\
        \end{tabular}
        }

        \For{$\kappa=1,\ldots,M$}{
        Compute the (partial) eigen-decomposition according
        to~\eqref{eqn:eigenAkappa}\;
        Compute the local SVD
        decomposition according to~\eqref{eqn:localSVD} and only keep singular
        vectors with singular values larger than $\wt{\tau}$\;
        Compute $U_{\kappa}$ with matrix multiplication
        according to~\eqref{eqn:UV}\;
        } 
      \end{minipage}
    \end{center}
  \end{small}
  \caption{Localized spectrum slicing basis set.}
  \label{alg:LSS}
\end{algorithm}

%\subsection{Eliminate linearly independent columns}
%
%
%\LL{Algorithm for constructing interior eigenvalue problem}
%
%
%
%The global procedure for eliminating linearly independent columns have
%the advantage that the condition number can be controlled to arbitrary
%level, and effectively reduce the over-completeness of the choice of
%basis functions.  But it can also lead to loss of accuracy, and
%therefore should be balanced.

\subsection{Complexity}\label{subsec:complexity}

In order to simplify the analysis of the complexity of the
Algorithm~\ref{alg:LSS} for finding the LSS basis set, we make the assumption that
the set of $n$ vertices is equally divided into $M$
elements, so that $\abs{E_{\kappa}} = \frac{n}{M} \equiv \abs{E}$.
\REV{As $n$ increases we assume $|E|$ can be kept as a constant, i.e.
the number of elements $M$ increases proportionally with respect to
$n$.}
$\abs{Q_{\kappa}} =
\frac{c_{Q} n}{M} \equiv c_{Q} \abs{E}$, where $c_{Q}$ is a small number denoting the ratio
between the size of the extended element and the size of the element.
\REV{For instance, for the discretized 1D and 2D Laplacian operators in
the numerical examples, $c_{Q}$ is set to be $3$ and $9$, respectively.}

Denote by $s_{\kappa}$ the column dimension of $X_{\kappa}$ in the
partial eigen-decomposition of $A_{\kappa}$, and by
$t_{\kappa}$ the column dimension of $U_{\kappa}$ with
$t_{\kappa}\le s_{\kappa}$. For simplicity we assume
$\{s_{\kappa}\},\{t_{\kappa}\}$ are uniform i.e.
$s_{\kappa}=s, t_{\kappa}=t,\kappa=1,\ldots,M$.
If $A_{\kappa}$ is treated as a dense matrix for the computation of the
local eigen-decomposition of $A_{\kappa}$, then the cost is
$c_{\mathrm{Eig,d}}\abs{Q_{\kappa}}^{3}$.  
The cost of the SVD decomposition is $c_{\mathrm{SVD}}\abs{E_{\kappa}}
s_{\kappa}^2$.
The cost of matrix multiplication to obtain $U_{\kappa}$ is
$c_{\mathrm{MM}}  \abs{Q_{\kappa}} s_{\kappa} t_{\kappa}$.
So the total cost for finding the LSS basis set is proportional to
\begin{equation}
\sum_{\kappa=1}^{M} c_{\mathrm{Eig,d}}\abs{Q_{\kappa}}^{3} +
c_{\mathrm{SVD}}\abs{E_{\kappa}}s_{\kappa}^2  +
c_{\mathrm{MM}}\abs{Q_{\kappa}}s_{\kappa}t_{\kappa} =
n \left( c_{\mathrm{Eig,d}} c_{Q}^{3} \abs{E}^2 + c_{\mathrm{SVD}} s^2 +
c_{\mathrm{MM}} c_{Q} st \right).
  \label{eqn:complexLSS}
\end{equation}
If we assume that as $n$ increases, the spectral radius of $A$ does not
increase, then 
all constants in the parenthesis in the right hand side of
Eq.~\eqref{eqn:complexLSS} are independent of $n$, and the overall
computational complexity for finding the LSS basis set is $\Or(n)$.

%\begin{remark}
%%  Although $s$ is the average rank of the eigenvectors defined locally,
%%  it does not mean that $s$ can be arbitrarily small.  This is because
%%  the localized spectrum slicing requires $Q_{\kappa}$ to be
%%  sufficiently large compared to each element $E_{\kappa}$.  
%The size of the extended element $Q_{\kappa}$ carries similar meaning to
%the parameter $m$ in Theorem~\ref{thm:truncate}.  If the size of
%$Q_{\kappa}$ is too small, the LSS operator may not be accurate enough.
%\end{remark}

In practice the constant for the finding the local eigen-decomposition
can be large due to the term $\abs{E}^2$ in
Eq.~\eqref{eqn:complexLSS}.  Since $A_{\kappa}$ is still a sparse
matrix on $Q_{\kappa}$, iterative methods can be used to reduce the
computational cost to $c_{\mathrm{Eig,i}} \abs{Q_{\kappa}}
s_{\kappa}^2$. This modifies the overall complexity to be
\[
n \left( c_{\mathrm{Eig,i}} c_{Q} s^2 + c_{\mathrm{SVD}} s^2 +
c_{\mathrm{MM}} c_{Q} st \right).
\]
However, it should be noted that the preconstant $c_{\mathrm{Eig,i}}$
might be larger than $c_{\mathrm{Eig,d}}$. Whether direct or iterative
method should be used to solve the local eigenvalue problem may depend
on a number of practical factors such as the size of the local
problem, and the availability of efficient preconditioner on the local
domain etc.

\subsection{Compute interior eigenvalues}\label{sec:interior}

Using the LSS basis set in~\eqref{eqn:LSSbasis}, one may compute the
interior eigenvalues near $\mu$ together with its associated
eigenvectors. This can be done by using the projected matrices
$A_{U},B_{U}$ according
to Eq.~\eqref{eqn:projAB}. 
Due to the spatial sparsity of
$U$, $A_{U},B_{U}$ are also sparse matrices,  
and can be assembled efficiently with local computation. First, the matrix
multiplication $Z = A U$ can be performed locally.  This is because each
column of $U_{\kappa}$ is localized in $Q_{\kappa}$, \REV{then}
%. In particular, if
%each column of $U_{\kappa}$ vanishes near the boundary of $Q_{\kappa}$, then
\begin{equation}
  Z_{\kappa} = A U_{\kappa} \approx A_{\kappa}U_{\kappa}.
  \label{eqn:Zkappa}
\end{equation}
Second, denote by 
\[
(A_{U})_{\kappa',\kappa}=U_{\kappa'}^* Z_{\kappa}, \quad
(B_{U})_{\kappa',\kappa}=U_{\kappa'}^* U_{\kappa},
\]
then for each $\kappa$ it is sufficient to loop over elements
$E_{\kappa'}$ so that $Q_{\kappa'}\bigcap Q_{\kappa}$ is non-empty.
The details for constructing the projected matrices are given in
Algorithm~\ref{alg:assembly}.

\begin{algorithm}  
  \begin{small}
    \begin{center}
      \begin{minipage}{5in}
        \KwIn{\begin{tabular}{l} 
          (1) 
          \begin{minipage}[t]{4.0in} 
            Sparse Hermitian matrix $A$.
          \end{minipage}\\
          (2) 
          \begin{minipage}[t]{4.0in} 
            Number of elements $M$, partition of elements
            $\{E_{\kappa}\}_{\kappa=1}^{M}$, extended elements
            $\{Q_{\kappa}\}_{\kappa=1}^{M}$, submatrices
            $\{A_{\kappa}\}_{\kappa=1}^{M}$, LSS basis set
            $\{U_{\kappa}\}_{\kappa=1}^{M}$ with total number of basis
            functions $n_{b}$.
          \end{minipage}
        \end{tabular}
        }
        \KwOut{\begin{tabular}{l} 
          \begin{minipage}[t]{4.0in} 
            Projected matrices $A_{U},B_{U}$.
          \end{minipage}\\
        \end{tabular}
        }

        Let $A_{U},B_{U}$ be zero matrices of size $n_{b}\times
        n_{b}$.

        \For{$\kappa=1,\ldots,M$}{
        Compute $Z_{\kappa}\gets A_{\kappa} U_{\kappa}$\;
        \For{$\kappa'$ so that $Q_{\kappa'}\bigcap
        Q_{\kappa}\ne \emptyset$ }{

        Compute $(A_{U})_{\kappa',\kappa}\gets 
        U_{\kappa'}^{*} Z_{\kappa}$\;
        Compute $(B_{U})_{\kappa',\kappa}\gets 
        U_{\kappa'}^{*} U_{\kappa}$\;
        }}

        Symmetrize $A_{U}\gets \frac12 (A_{U}+A_{U}^{*}),\quad
        B_{U}\gets \frac12 (B_{U}+B_{U}^{*})$.
      \end{minipage}
    \end{center}
  \end{small}
  \caption{Assembly of the projected matrices.}
  \label{alg:assembly}
\end{algorithm}

After $A_{U},B_{U}$ are assembled, the eigenvalues and corresponding
eigenvectors near $\mu$ can be solved in various ways.  When the size of
the LSS basis set $n_{b}$ is small, one can treat $A_{U},B_{U}$ as dense
matrices and solve the generalized eigenvalue problem
\begin{equation}
  A_{U} C = B_{U} C \Theta,
  \label{eqn:generalEig}
\end{equation}
and only keep the Ritz values
$\Theta=\mathrm{diag}[\theta_{1},\ldots,\theta_{n_{b}}]$ and corresponding Ritz vectors
$C$ near $\mu$.  Each column of the Ritz vector $C_{j}$ can be partitioned
according to the element partition $\{E_{\kappa}\}$ as
\[
  C_{j} = [C_{1,j},\ldots,C_{M,j}]^{T}.
\]
Then an approximate eigenvector for $A$ can be computed as
\begin{equation}
  \wt{X}_{j} = U C_{j} = \sum_{\kappa} U_{\kappa} C_{\kappa,j}.
  \label{eqn:approxEigvec}
\end{equation}

We remark that in the computation of interior eigenvalues, 
spurious eigenvalues may appear. A spurious eigenvalue is
a Ritz value $\theta_{j}$ near the vicinity of $\mu$ as obtained from
Eq.~\eqref{eqn:generalEig}, but the
corresponding vector $\wt{X}_{j}$ as
given in Eq.~\eqref{eqn:approxEigvec} is not an approximate
eigenvector.  The appearance of spurious eigenvalue is also referred to
as spectral pollution~\cite{GruberRappaz1985,Knyazev1997}, and
could be identified by computing the residual
\begin{equation}
  R_{j} = A \wt{X}_{j} - \wt{X}_{j} \theta_{j}.
  \label{eqn:residual}
\end{equation}
A Ritz value $\theta_{j}$ corresponding to large residual norm
$\norm{R_{j}}_{2}$ should be removed.  Note that the residual can also
be computed with local computation
\begin{equation}
  R_{j} = \sum_{\kappa} \left(Z_{\kappa} C_{\kappa,j} - U_{\kappa}
  C_{\kappa,j} \theta_{j}\right),
  \label{eqn:residuallocal}
\end{equation}
where $Z_{\kappa}$ is given in~\eqref{eqn:Zkappa}. \REV{Our numerical
experience indicates that the use of residual is an effective way for
identifying spurious eigenvalues when the LSS basis set is accurate
enough for approximating the subspace spanned by the eigenvectors to be
computed.  In such case the norm of the residual for most Ritz values is small and
the norm of the residual for the spurious eigenvalue stands out. When
the basis set cannot accurately capture all the eigenvalues in the
prescribed interval especially for those clustered near the
boundary of the interval, it becomes more difficult to identify all the
spurious eigenvalues.}

\REV{
\subsection{Domain partitioning for general sparse matrices}\label{subsec:lssgeneral}

For a general sparse matrix $A$, we discuss here the strategy to partition the
associated undirected graph $G=(\mc{V},\mc{E})$ into $M$ elements
$\{E_{\kappa}\}_{\kappa=1}^M$.
Intuitively we would like to choose a
partition that keeps all $E_{\kappa}$ to have similar
sizes, while minimizing the number of edges that is being cut by the
partition, i.e. $\sum_{\kappa,\kappa'=1}^{M} \sum_{i\in E_{\kappa},j\in
E_{\kappa'}} w_{ij}$. Here $w_{ij}=1$ if $A_{ij}\ne 0$ and $0$
otherwise. This is called a minimal $M$-cut problem. It is known that
the minimal $M$-cut problem is NP-hard. Various heuristic methods have
been developed. Here we use the nested
dissection approach~\cite{George1973} as implemented in the
METIS~\cite{KarypisKumar1998} package.
The nested dissection approach can find an approximate minimal $2$-cut of
the graph, and then recursively partitions each part of the graph, with
iterative adjustment of the size of $E_{\kappa}$. For each $\kappa$ we
define a neighbor list $N_{\kappa}$, which consists of 
$\kappa$ itself, as well as other element indices
$\kappa'$ such that there exists at least one pair of indices $i\in
E_{\kappa},j\in E_{\kappa'}$ and $A_{ij}\ne 0$. 
Then the extended element $Q_{\kappa}$ is defined as the collection of
all indices in $E_{\kappa'}$ such that $\kappa'\in N_{\kappa}$.
Algorithm~\ref{alg:partition} gives a pseudo-code for generating the
elements $\{E_{\kappa}\}$, the neighbor lists $\{N_{\kappa}\}$, and the
extended elements $\{Q_{\kappa}\}$. In terms of implementation, the
partition of the graph is given by a \textit{graph partition map} $\xi$
such that $E_{\kappa} = \{i\in \mc{V} \vert \xi(i) = \kappa\}$, and $\xi$ can be
directly returned from a graph partitioning package such as METIS.
}

\begin{algorithm}  
  \begin{small}
    \begin{center}
      \begin{minipage}{5in}
        \KwIn{\begin{tabular}{l} 
          \begin{minipage}[t]{4.0in}
            Sparse Hermitian matrix $A$ of size $n\times n$. Number of elements $M$.
          \end{minipage} 
        \end{tabular}
        }
        \KwOut{\begin{tabular}{l} 
          \begin{minipage}[t]{4.0in}
            $E_{\kappa},N_{\kappa},Q_{\kappa},\kappa=1,\ldots,M$.
          \end{minipage} 
        \end{tabular}
        }

        $\xi = \mathrm{GraphPartition}(A)$.

        $E_{\kappa} = \{i\in \mc{V} \vert \xi(i) = \kappa\},\kappa=1,\ldots,M$.

        $N_{\kappa} = \{\kappa\}\cup \{\kappa' \vert \exists i\in E_{\kappa'},j\in
        E_{\kappa}, A(i,j) \ne 0\},\kappa=1,\ldots,M$.

        $Q_{\kappa} = \{i\in \mc{V} \vert i\in E_{\kappa'}, \kappa'\in
        N_{\kappa}\},\kappa=1,\ldots,M.$
      \end{minipage}
    \end{center}
  \end{small}
  \caption{Generating the set of elements $\{E_{\kappa}\}$ and
  neighboring elements for a general sparse matrix.}
  \label{alg:partition}
\end{algorithm}

\section{Numerical results}\label{sec:numer}

%In this section we demonstrate the performance of the localized spectrum
%slicing procedure for a matrix obtained from a discretized second order PDE in
%Eq.~\eqref{eqn:Quantum}. We compute the matrix function $f_{\sigma,\mu}(A)$,
%and compute the interior eigenvalues in the interval
%$(\mu-\sigma,\mu+\sigma)$.
%To simplify the discussion, we let the global computational domain
%$\Omega$ be a one-dimensional domain with periodic boundary
%condition.  Let $\mc{T}$ be a collection of geometrically conforming,
%regular, quasi-uniform partitions of $\Omega$:
%\begin{equation}
%  \mc{T} = \{E_1, E_2, \cdots, E_M \}.
%\end{equation}
%Each $E_{K}$ is called an element of $\Omega$ (see
%Fig.~\ref{fig:element}).  For each $K\in \mc{T}$, we introduce an
%associated {\em extended element} $Q_{K} \supset E_{K}$. The
%computational domain $\Omega$ is discretized into a set of uniform grid
%points, and correspondingly the Laplacian operator is discretized using
%a 3 point finite difference formula.  

%\LL{Improve below}
% For instance,
%$Q_{K}$ can be defined to be the set of vertices so that it contain all
%indices such that $d(i,j)\le m$, where $i\in E_{K},j\in \Omega$. For
%the solution of PDEs, for a $d$-dimensional space, it is also convenient
%to take $Q_{K}$ to be simply the union of nearest neighbor elements of
%$E_{K}$ which is $3^{d}$ in total.  We define the matrix
%$B_{Q_{K}}$ to be the same as the Dirichlet modification of the matrix
%$A$ inside $Q_{K}$.

In this section we demonstrate the accuracy and efficiency of the
divide-and-conquer procedure for computing the LSS operator and the LSS
basis set, and for computing interior eigenvalues. All the
computation is performed on a single computational thread of an Intel i7
CPU processor with $64$ gigabytes (GB) of memory using MATLAB.  The
matrix $A$ is obtained from a discretized second order partial
differential operator $-\Delta+V$ in one-dimension (1D) and in
two-dimension (2D) with periodic boundary conditions, \REV{and a general
matrix from the University of Florida matrix collection.}

\subsection{One-dimensional case}

In the 1D case, the global domain is
$\Omega=[0,L]$. The Laplacian operator is
discretized using a 3-point finite difference stencil.  The domain is
uniformly discretized into $n=c_{n} M$ grid points so that $x_{i}=(i-1)
h$, with the grid spacing $h\equiv L/n=0.1$.  
All the $n$ grid points (vertices) are uniformly and contiguously
partitioned into $M$ elements $\{E_{\kappa}\}_{\kappa=1}^{M}$. For
simplicity let $Q_{\kappa}$ be the union of $E_{\kappa}$ and its two
neighbors taking into account the periodic boundary condition, i.e.
\[
Q_{\kappa}=\begin{cases}
  E_{M}\bigcup E_{1}\bigcup E_{2}, & \kappa=1,\\
  \bigcup_{\kappa'=\kappa-1}^{\kappa+1} E_{\kappa}, &
  \kappa=2,\ldots,M-1,\\
  E_{M-1}\bigcup E_{M}\bigcup E_{1}, & \kappa=M.
\end{cases}
\]
The potential $V(x)$ is given by the sum of $n_{w}$ exponential
functions as
\begin{equation}
  V(x) = -\sum_{i=1}^{n_{w}} a_{i}
  e^{-\frac{\mathrm{dist}(x,R_{i})}{\delta_{i}}}.
  \label{eqn:Vx}
\end{equation}
Here $\{R_{i}\}$ are a set of equally spaced points. The distance between
two points $x$ and $x'$ is defined to be the minimal distance between
$x$ and all the periodic images of $x'$, i.e.
\[
\mathrm{dist}(x,x') = \min_{\wt{x}'=x'+kL,k\in \mathbb{Z}}
\abs{x-\wt{x}'}.
\]
In order to study the performance of the algorithm for systems of
increasing sizes, we set 
$L=20 n_{w}$  so that the length of the computational domain is
proportional to the number of potential wells $n_{w}$.  To show that we do not take advantage of
the periodicity of the potential, we introduce some randomness in each
exponential function. We choose $a_{i}\sim \mathcal{N}(5.0,1.0)$, which
is a Gaussian random variable with a mean value $5.0$ and a standard
deviation $1.0$. Similarly the width of the exponential function
$\delta_{i}\sim \mathcal{N}(2.0,0.2)$.  One realization of the potential
with $n_{w}=8$ is given in Fig.~\ref{fig:Vx1D} (a), with the partition of
elements indicated by black dashed lines.  For the choice of parameter
$\mu=2.0$ and $\sigma=1.0$, Fig.~\ref{fig:Vx1D} (b) shows the function
$f_{\sigma,\mu}(\lambda)$ evaluated on the eigenvalues of $A$
\REV{plotted in log-scale in the interval $(-5,10)$}, and
the LSS operator $f_{\sigma,\mu}(A)$ is spectrally localized.
\REV{Fig.~\ref{fig:Vx1D} (c) demonstrates the histogram of the
eigenvalues (unnormalized spectral density) for all eigenvalues of
$A$.}

\begin{figure}[h]
  \begin{center}
    \subfloat[]{\includegraphics[width=0.3\textwidth]{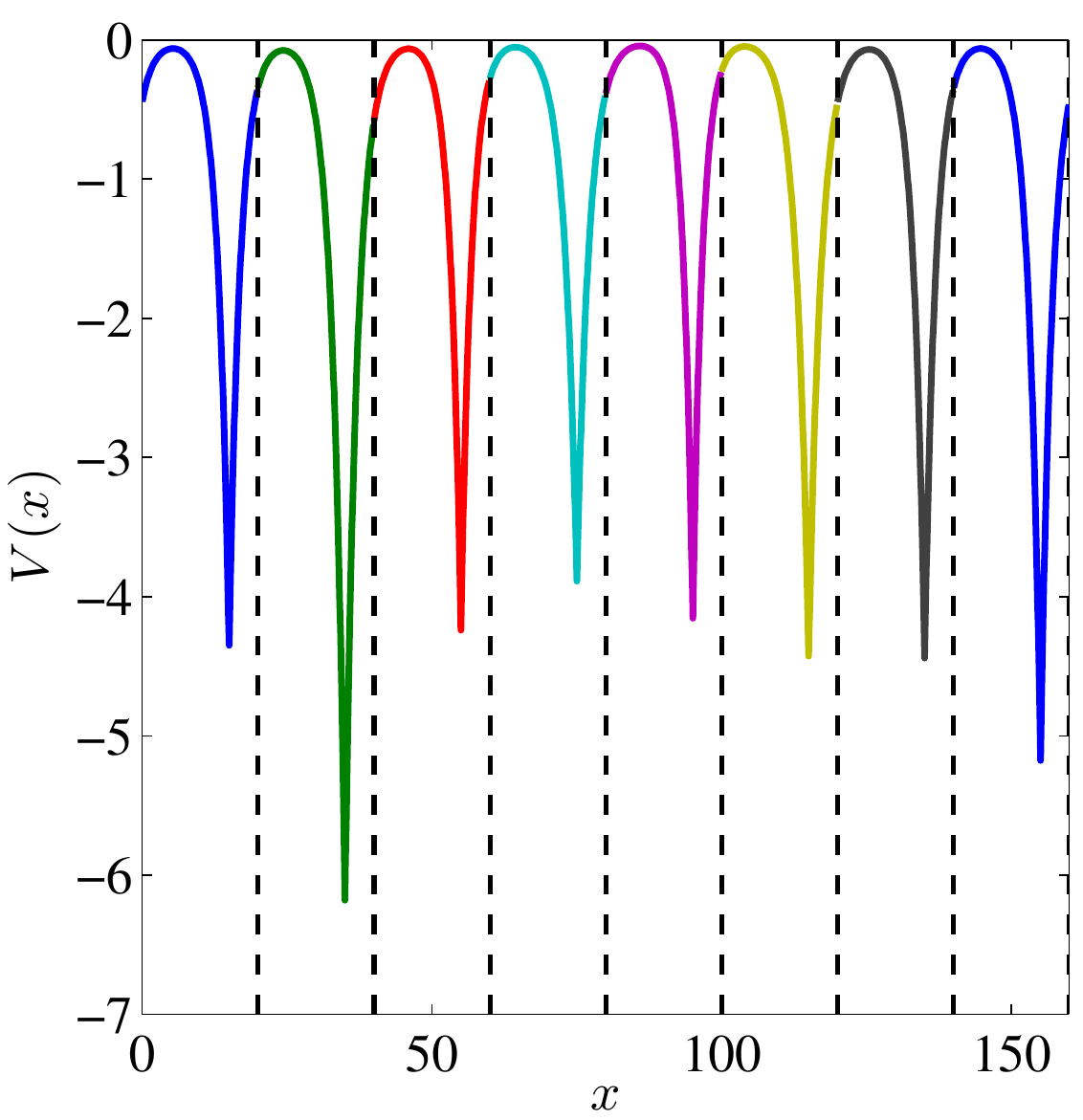}}
%    \subfloat[]{\includegraphics[width=0.3\textwidth]{slicemu2p0}}
    \subfloat[]{\includegraphics[width=0.335\textwidth]{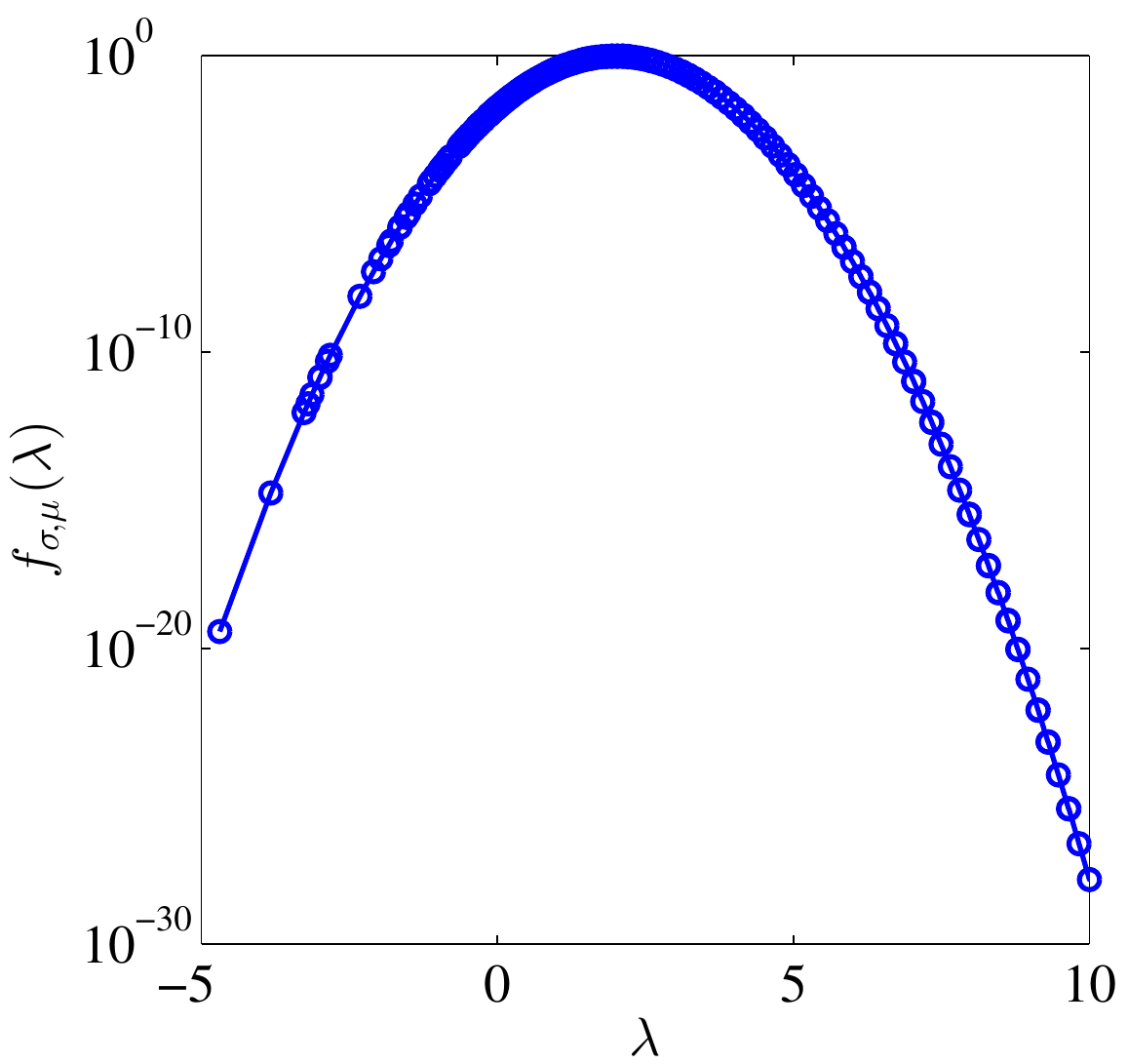}}
    \subfloat[]{\includegraphics[width=0.295\textwidth]{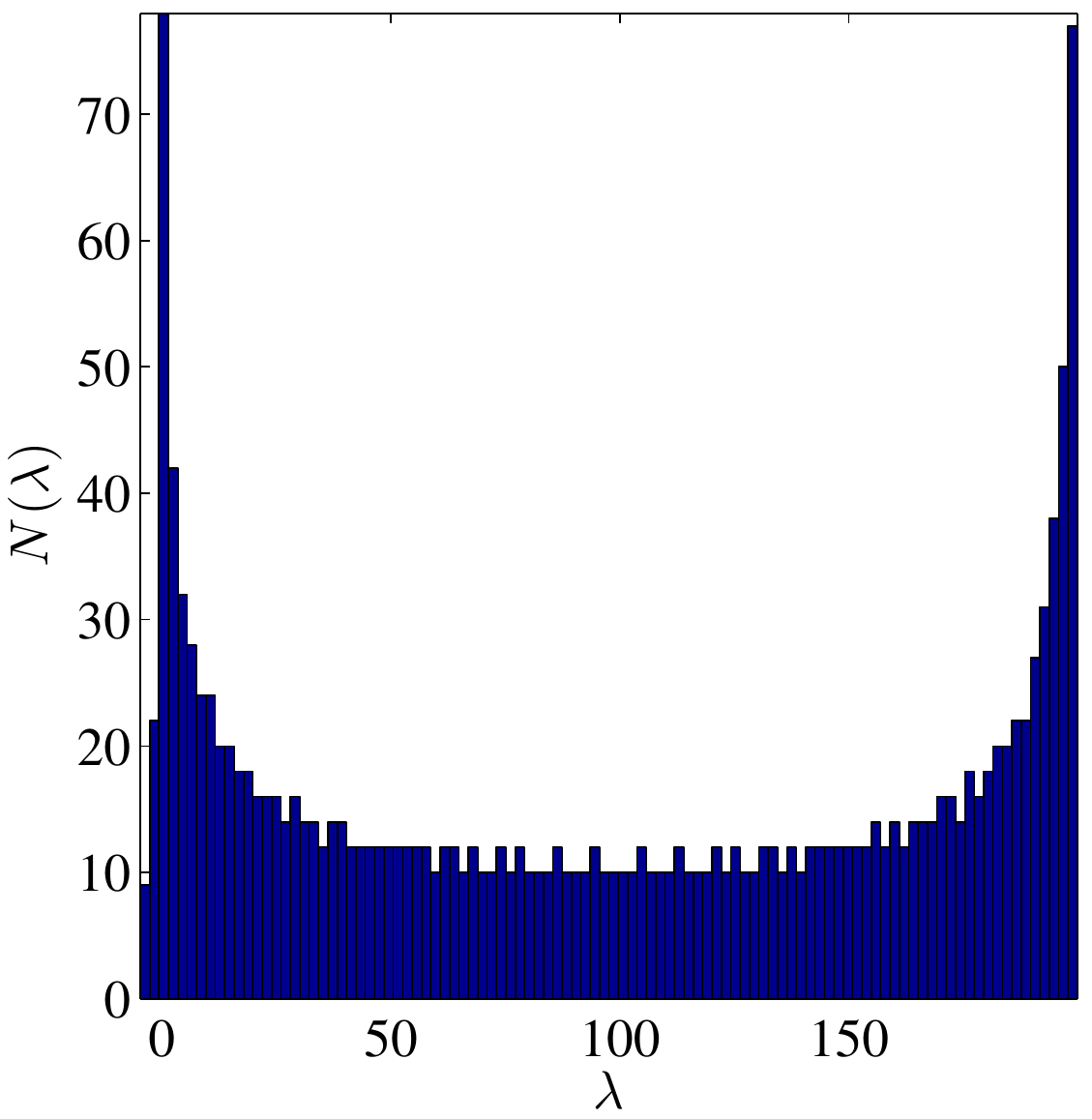}}
  \end{center}
  \caption{\REV{(a) One realization of the 1D potential with $n_{w}=8$. The
  domain is partitioned into $8$ equally sized elements separated by
  black dashed lines. (b) The function $f_{\sigma,\mu}(A)$ with $\sigma=1.0$ and
  $\mu=2.0$ viewed spectrally in the interval $(-5,15)$ plotted in the
  log scale.  The spectral radius of $A$ is $199.89$. (c) The histogram
  of the eigenvalues of $A$.} }
  \label{fig:Vx1D}
\end{figure}

Fig.~\ref{fig:fmuError} (a)-(c) demonstrates the behavior of the exact
LSS operator $f_{\sigma,\mu}(A)$ with $\sigma=1.0$ and increasing value
of $\mu$.  In Fig.~\ref{fig:fmuError}, $[f_{\sigma,\mu}(A)](x,y)$ should
be interpreted using its discretized matrix element
$[f_{\sigma,\mu}(A)]_{ij}$ for $x=(i-1)h,y=(j-1)h$. We find that as
$\mu$ increases, the off-diagonal elements of $f$ decays rapidly and
remains to be well approximated by a banded (and therefore sparse) matrix with
increasing bandwidth.  Fig.~\ref{fig:fmuError} (d)-(f) demonstrates the
quality of the divide-and-conquer approximation $\wt{f}$ to
the LSS operator.  Here we first demonstrate the accuracy of $\wt{f}$
without the truncation using SVD decomposition (i.e. the SVD truncation
criterion $\wt{\tau}=0$ as in Eq.~\eqref{eqn:localSVD}).  When
$\mu=-2.0$, the approximation is nearly exact, while when $\mu$
increases to $20.0$ the relative error is around $10\%$ since the
support size of each column of $f$ already extends beyond each extended
element $Q_{\kappa}$.

\begin{figure}[h]
  \begin{center}
    \subfloat[]{\includegraphics[width=0.3\textwidth]{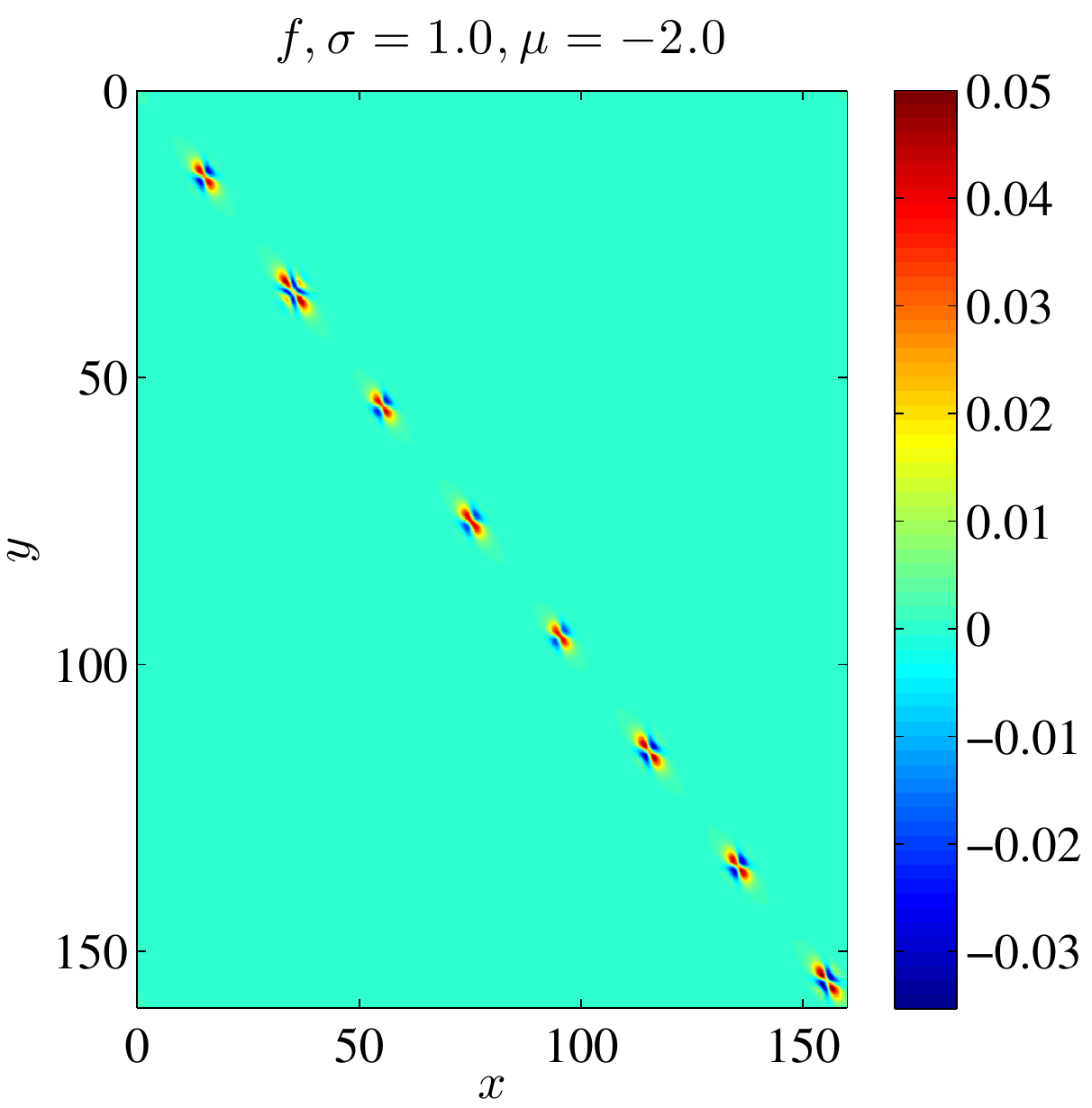}}
    \subfloat[]{\includegraphics[width=0.3\textwidth]{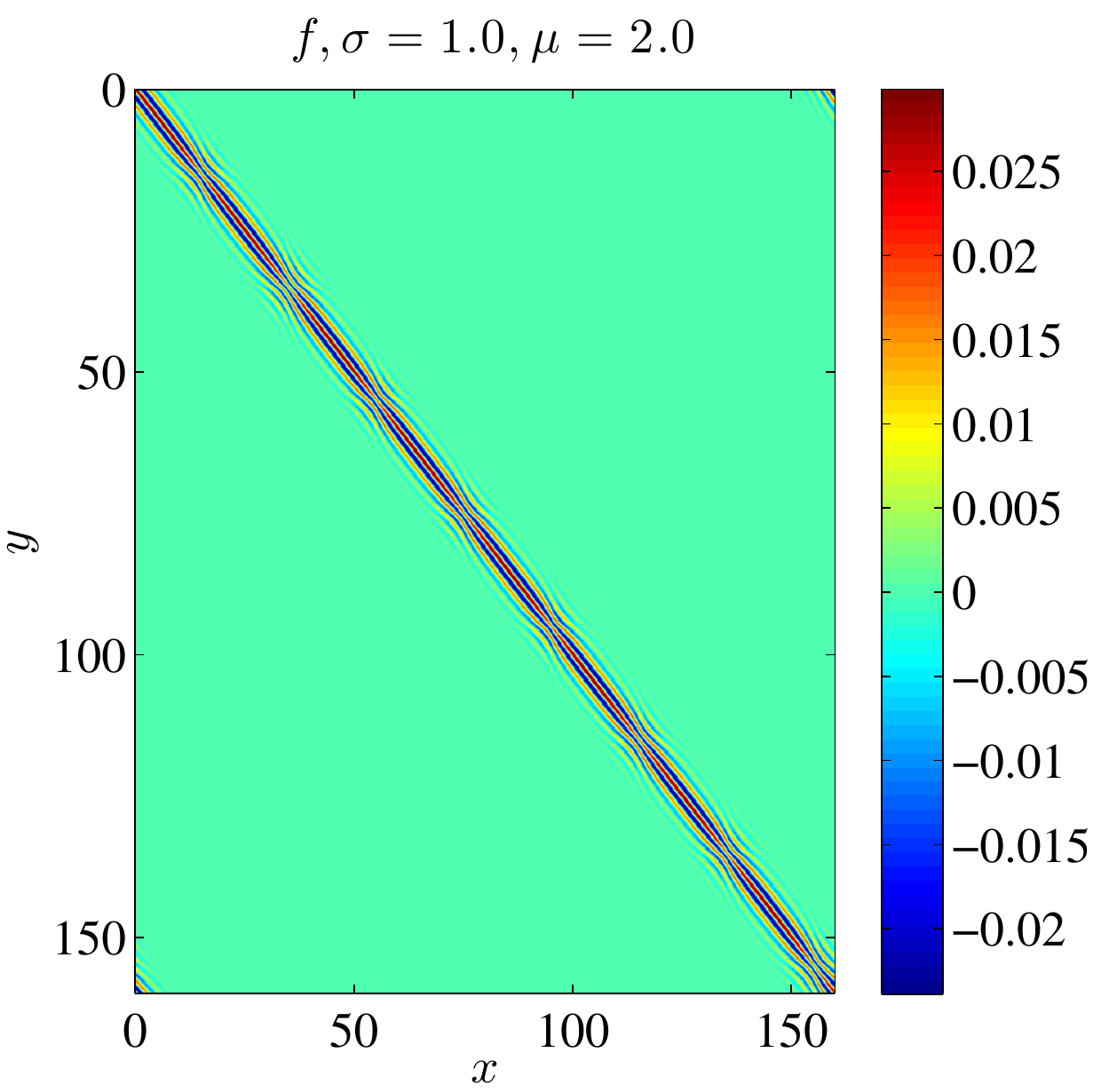}}
    \subfloat[]{\includegraphics[width=0.3\textwidth]{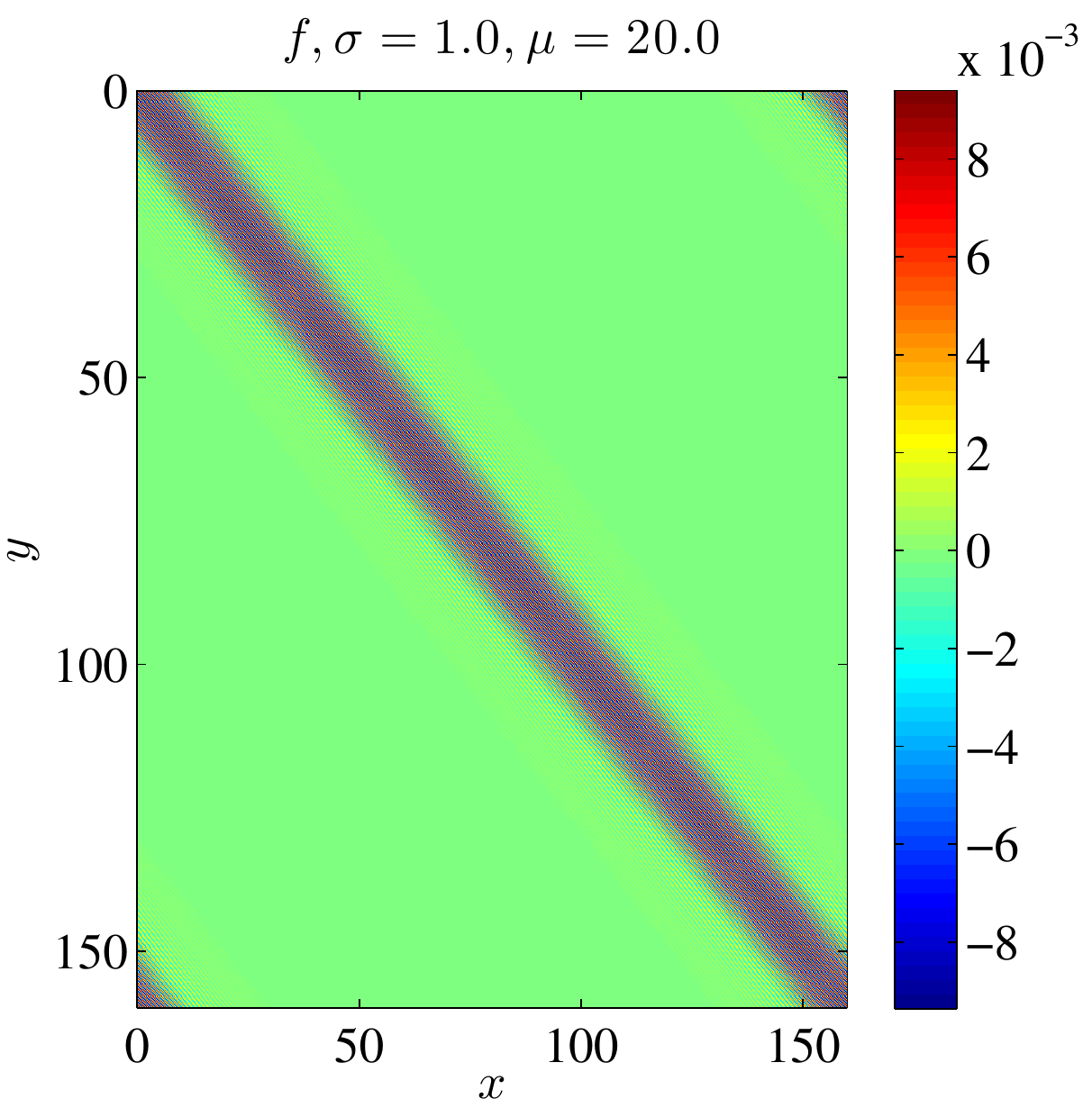}}
    
    \subfloat[]{\includegraphics[width=0.3\textwidth]{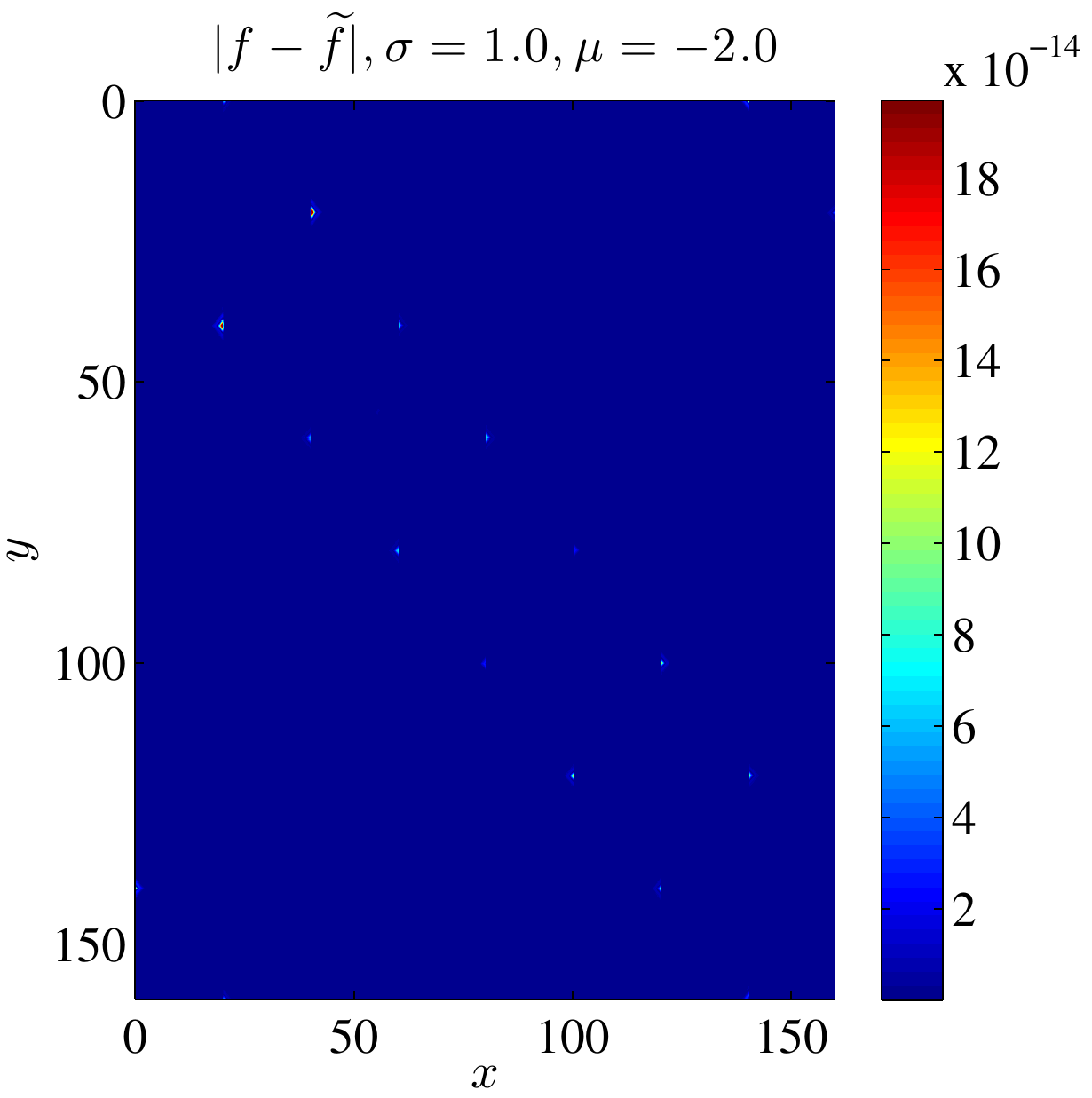}}
    \subfloat[]{\includegraphics[width=0.3\textwidth]{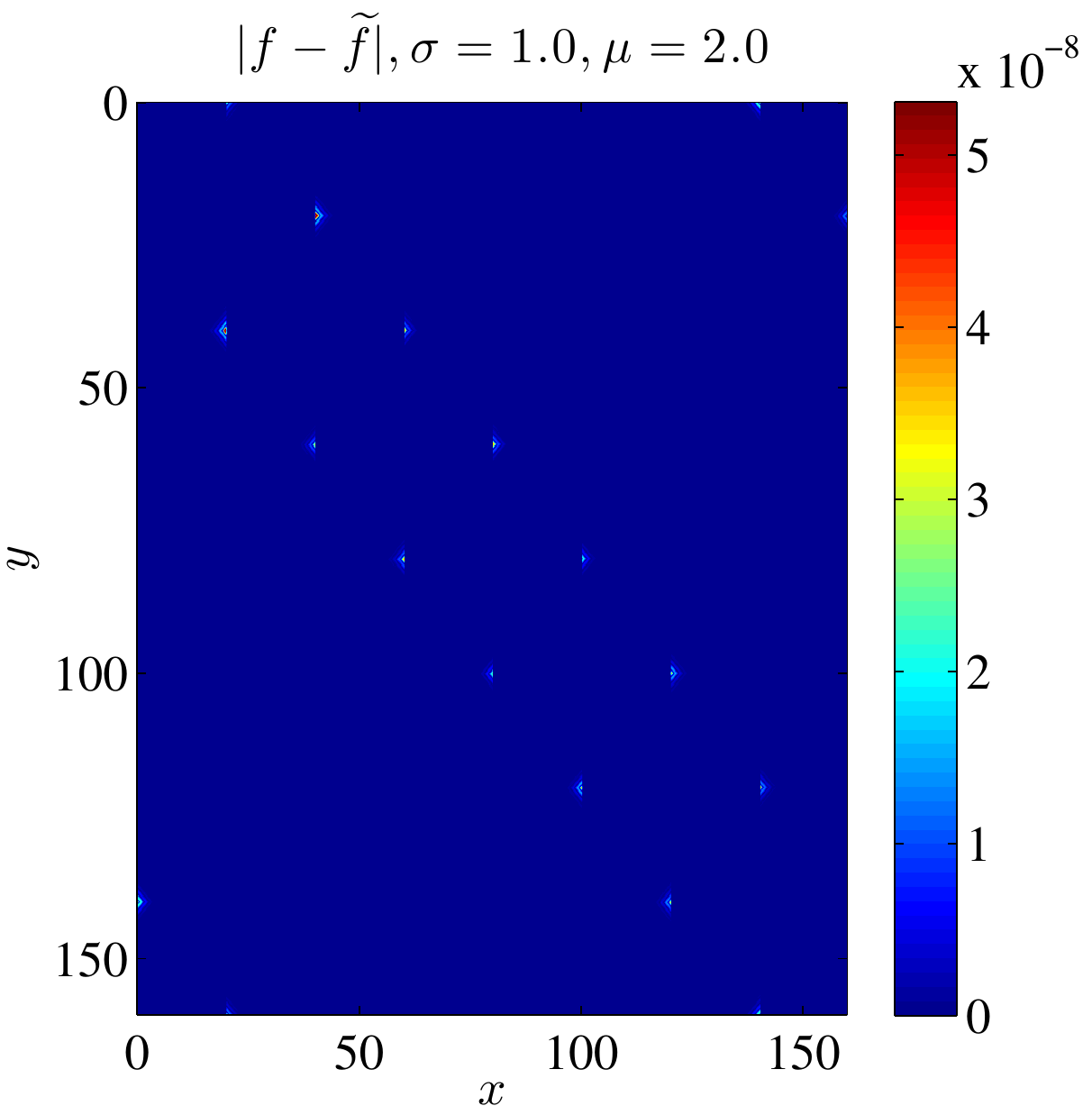}}
    \subfloat[]{\includegraphics[width=0.3\textwidth]{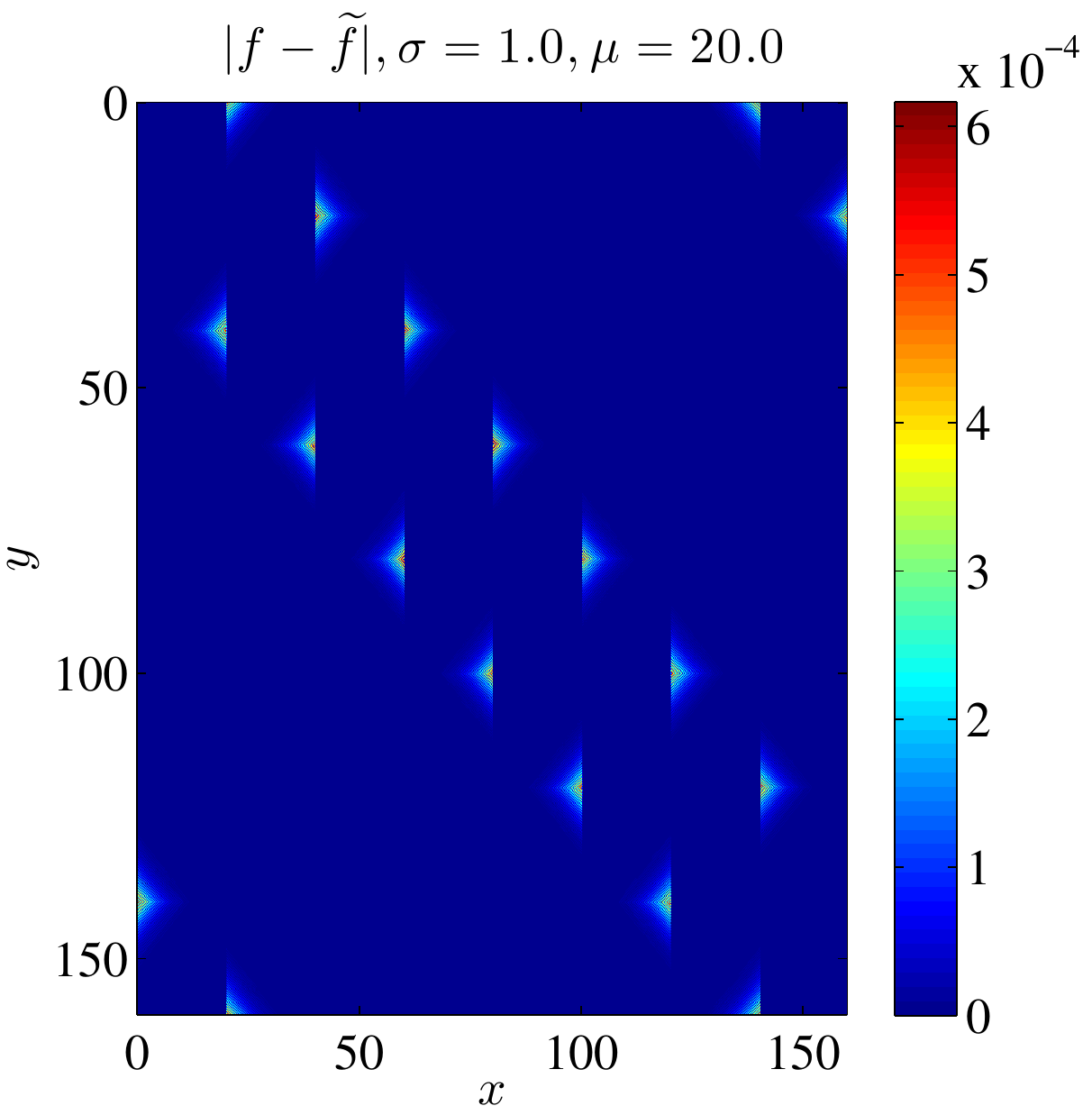}}
  
  \end{center}
  \caption{The LSS operator $f_{\sigma,\mu}(A)$ with $\sigma=1.0$ for (a) $\mu=-2.0$ (b)
  $\mu=2.0$ (c) $\mu=20.0$. The max error between the LSS
  operator and its divide-and-conquer approximation $\wt{f}_{\sigma,\mu}(A)$
  (d) $\mu=-2.0$ (e) $\mu=2.0$ (f) $\mu=20.0$.}
  \label{fig:fmuError}
\end{figure}

A more complete picture of the $\mu$-dependence for approximating the
LSS operator is given in Fig.~\ref{fig:fmuError2}.
Fig.~\ref{fig:fmuError2} (a) shows the max norm error of the
divide-and-conquer approximation to the LSS operator for $\mu$
traversing the entire spectrum of $A$ from $-3.0$ to $200.0$.  The error
increases rapidly as $\mu$ initially increases, achieves its maximum
at $\mu=100$ and then starts to decrease.  Fig.~\ref{fig:fmuError2} (b)
shows the same picture but zooms into the interval near $\mu=0$.  As
$\mu$ increases above $10.0$, the vectors spanning columns of $f_{\sigma,\mu}(A)$
are approximately linear combination of high frequency Fourier modes,
and Fig.~\ref{fig:fmuError2} (a) shows that the Fourier modes are
increasingly more difficult to localize as the frequency increases.
Fig.~\ref{fig:fmuError2} (c)-(d) shows similar behavior for
$\sigma=2.0$.  The profile of the error with respect to $\mu$ closely
resembles a Gaussian function.  Compared to the case with $\sigma=1.0$
the error significantly reduces for all $\mu$, indicating the balance
between spatial locality and spectral locality with varying $\sigma$.

\begin{figure}[h]
  \begin{center}
    \subfloat[]{\includegraphics[width=0.3\textwidth]{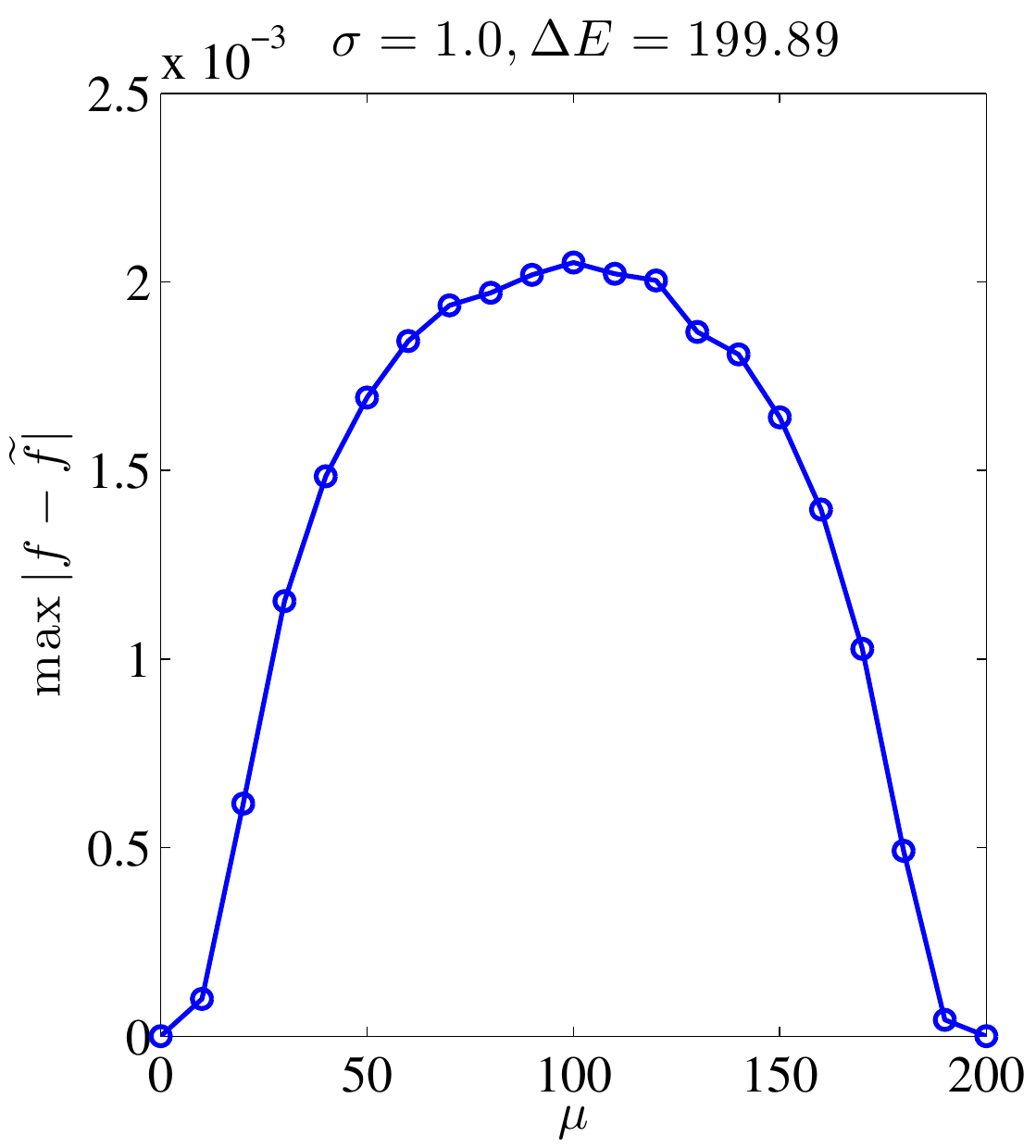}}\quad
    \subfloat[]{\includegraphics[width=0.3\textwidth]{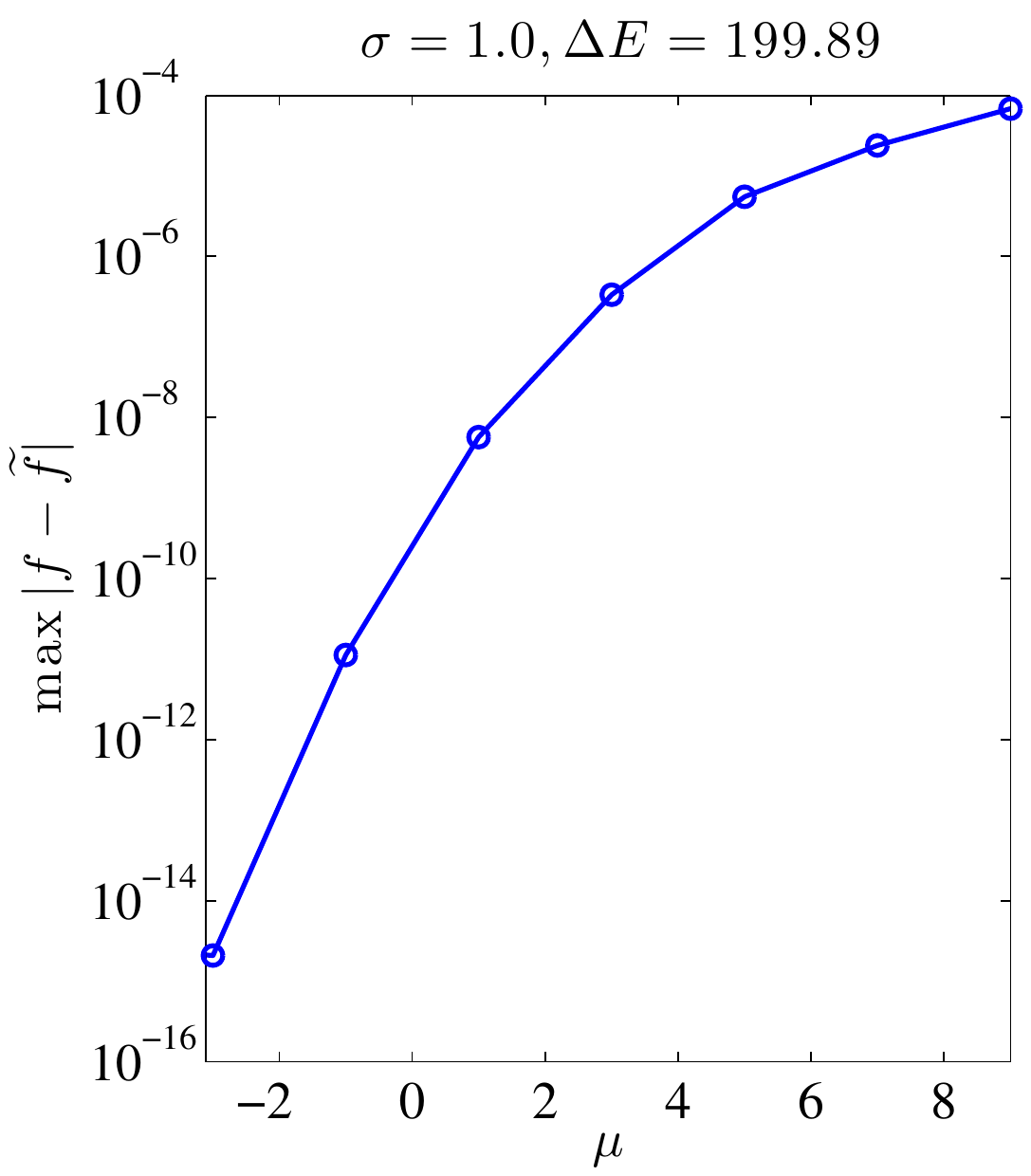}}

    \subfloat[]{\includegraphics[width=0.3\textwidth]{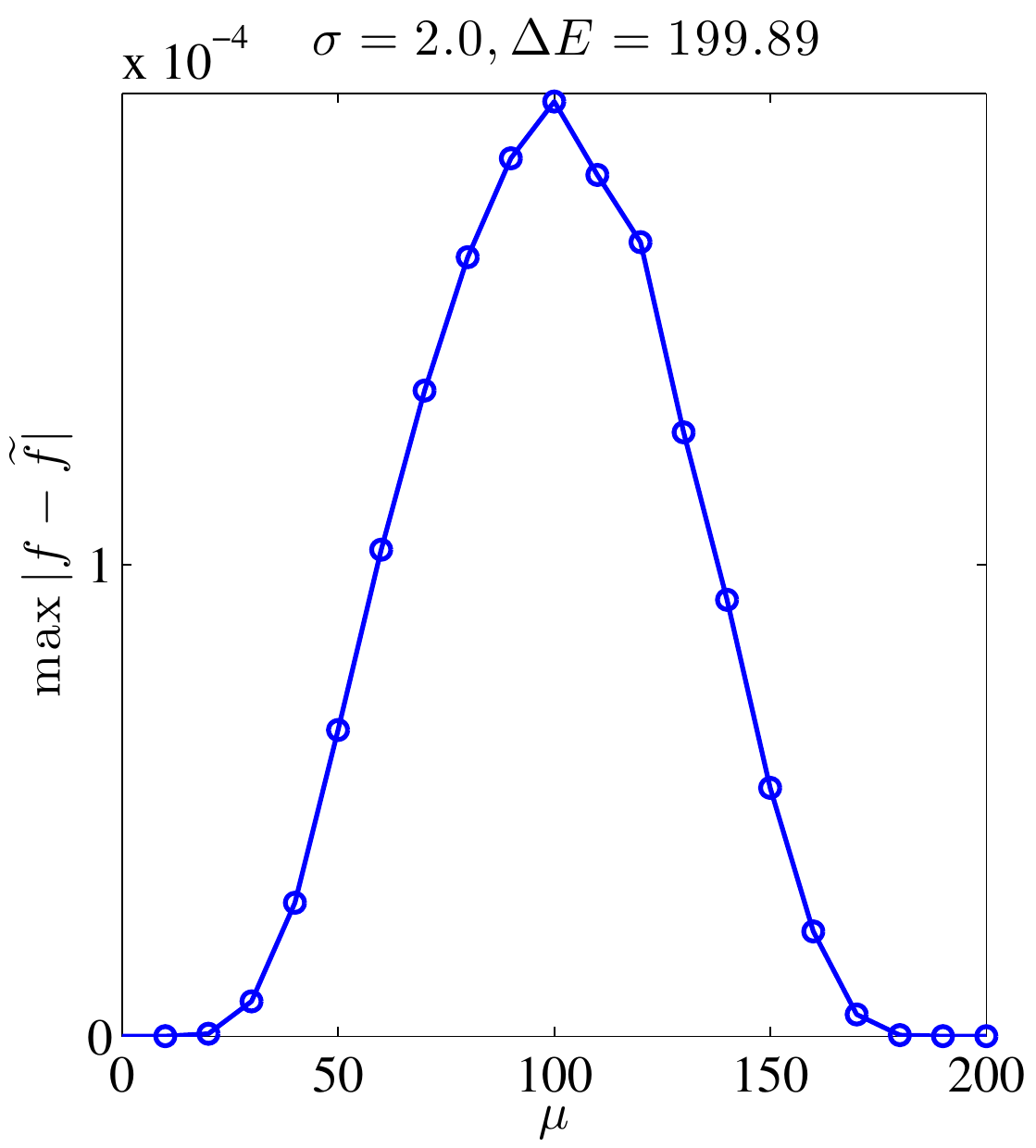}}\quad
    \subfloat[]{\includegraphics[width=0.3\textwidth]{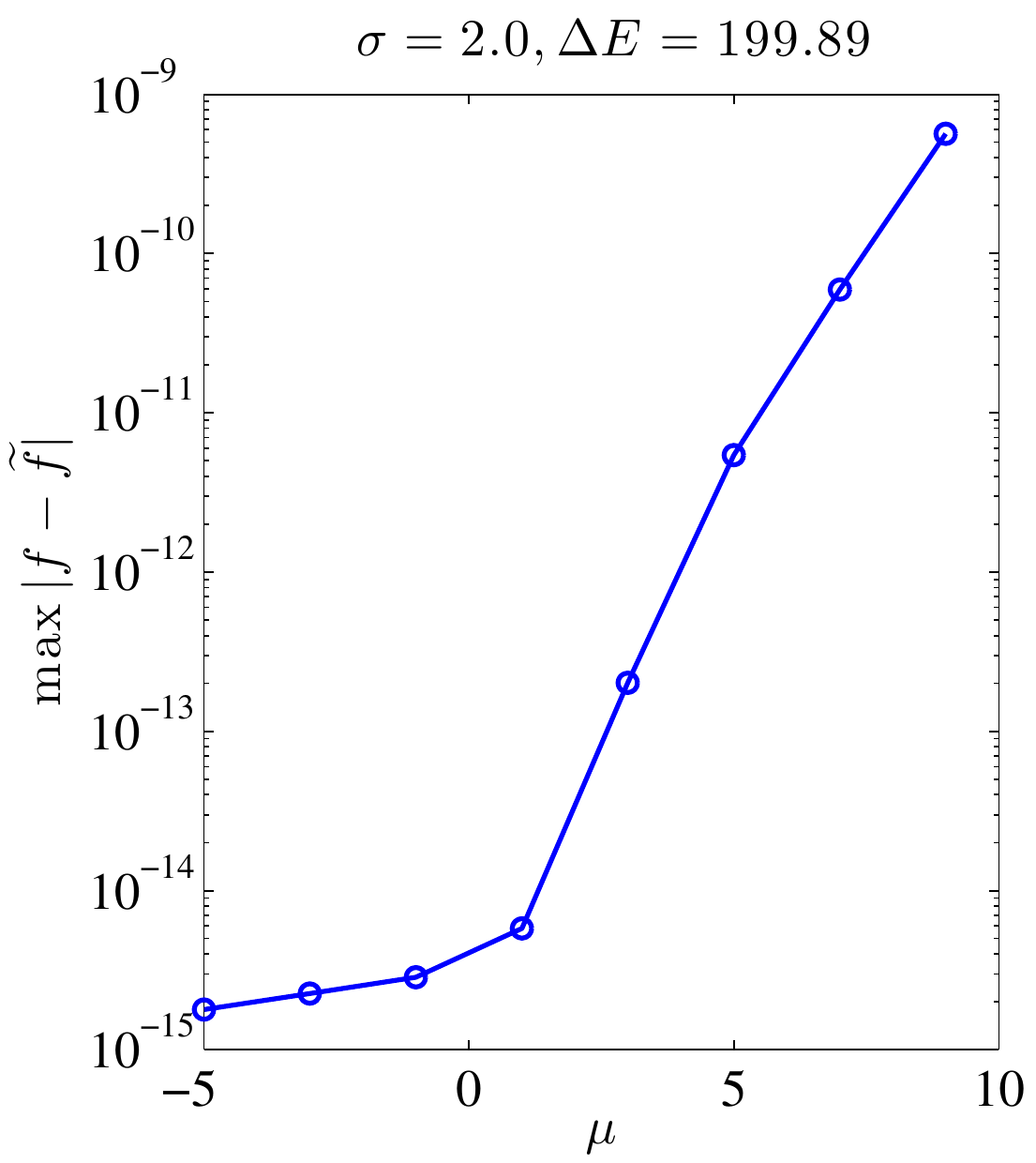}}
  \end{center}
  \caption{Max norm error of the LSS operator traversing the entire
  spectrum of $A$ for (a) $\sigma=1.0$; (b) $\sigma=1.0$, a zoomed in
  view; and (c) $\sigma=2.0$; (d) $\sigma=2.0$, a zoomed in
  view.  }
  \label{fig:fmuError2}
\end{figure}

Fig.~\ref{fig:fErrorSigma} (a) demonstrates the max norm error of the
LSS operator for $\mu=2.0$ with increasing value of $\sigma$. When
$\sigma$ is less than $0.25$ the LSS operator is very localized
spectrally, but the matrix is almost dense.  Therefore the
divide-and-conquer approximation leads to large error. As $\sigma$
increases above $0.25$, the max norm error decreases exponentially
with the increase of $\sigma$.  We observe that the choice of $\sigma$
is crucial: by varying $\sigma$ from $0.5$ to $1.5$, the error is
reduced by over $6$ orders of magnitude from $10^{-4}$ to below $10^{-10}$. 

Next we study the effect of grid refinement by varying the grid size
from $h=0.20$ to $h=0.033$. For 3-point finite difference stencil the
spectral radius of $A$, denoted by $\Delta E$ is proportional to
$1/h^2$, and in practice $\Delta E$ increases from $50$ to $1800$.  We
note that Theorem~\ref{thm:truncate} indicates that the error should be
determined by the ratio $\sigma/\Delta E$, and therefore the size of the
extended element as characterized by the geodesic distance $m$ should
increase proportionally to $\Delta E$ to preserve accuracy. Here instead we fix
the number of elements to be $8$ as the grid refines.  Therefore $m\sim
1/h \sim \sqrt{\Delta E}$, and we should expect that the error increases
as the grid refines.  Fig.~\ref{fig:fErrorSigma} (b) shows that max norm
error of the LSS operator for $\mu=2.0$,$\sigma=1.0$, with increasing
$\Delta E$. As the ratio $\sigma/\Delta E$ decreases over one order of
magnitude, the max norm error does not increase, but rather decreases by
more than a factor of $2$. We note that this numerical result does not
contradict the theoretical prediction, since
Theorem~\ref{thm:approxGaussian} only provide an \textit{upper bound} of
the decay rate, and the \textit{actual} decay rate can be faster.  Note
that as the grid refines, the change towards the high end of the
spectrum is often larger than the change at the low end of the spectrum.
Fig.~\ref{fig:fErrorSigma} indicates that the accuracy of the LSS
operator is relatively insensitive to the change in the high end of the
spectrum, and it may be possible to construct the LSS operator with
improved discretization scheme, without sacrificing too much in terms of
the spatial locality.

\begin{figure}[h]
  \begin{center}
    \subfloat[]{\includegraphics[width=0.3\textwidth]{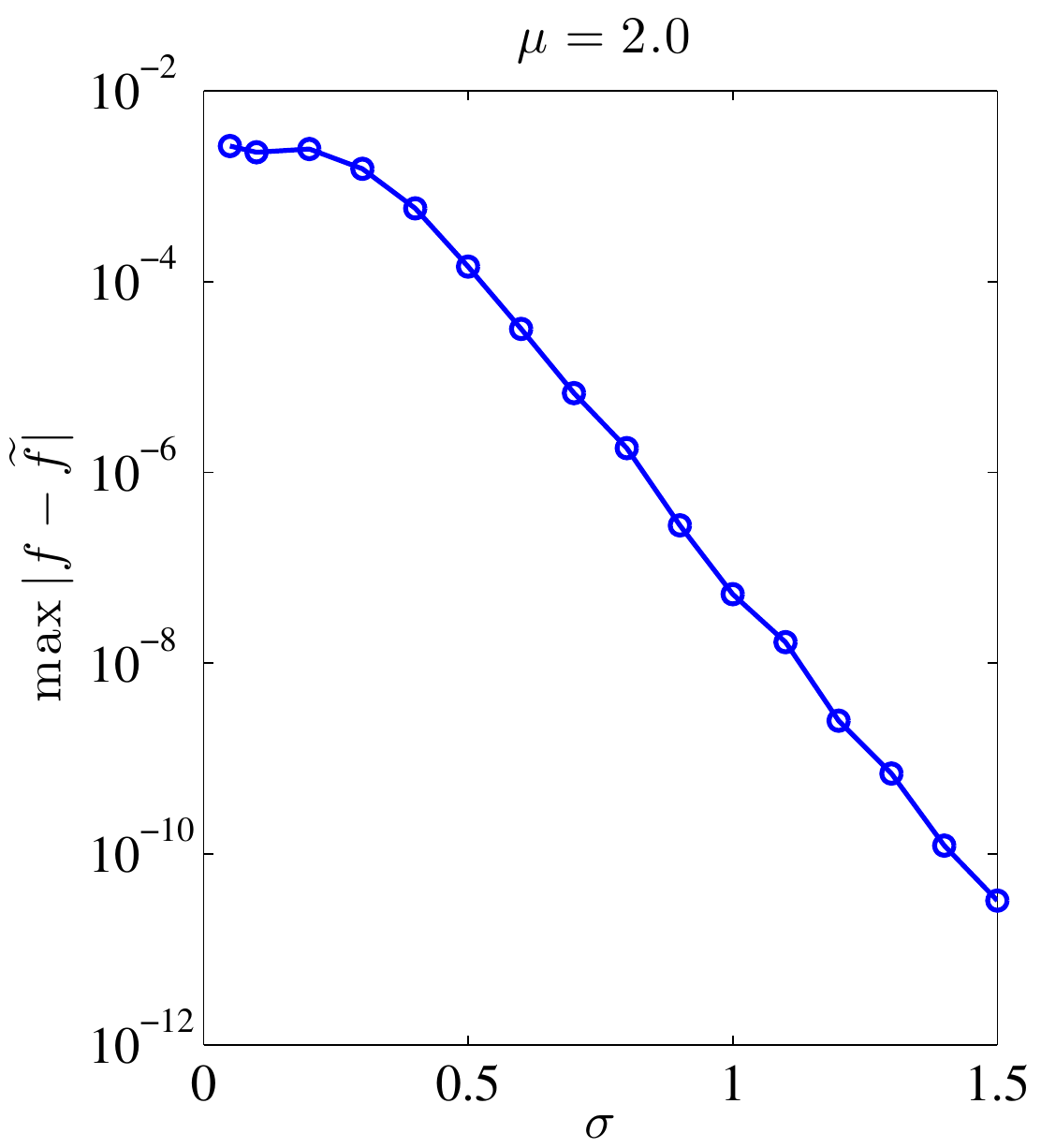}}\quad
    \subfloat[]{\includegraphics[width=0.3\textwidth]{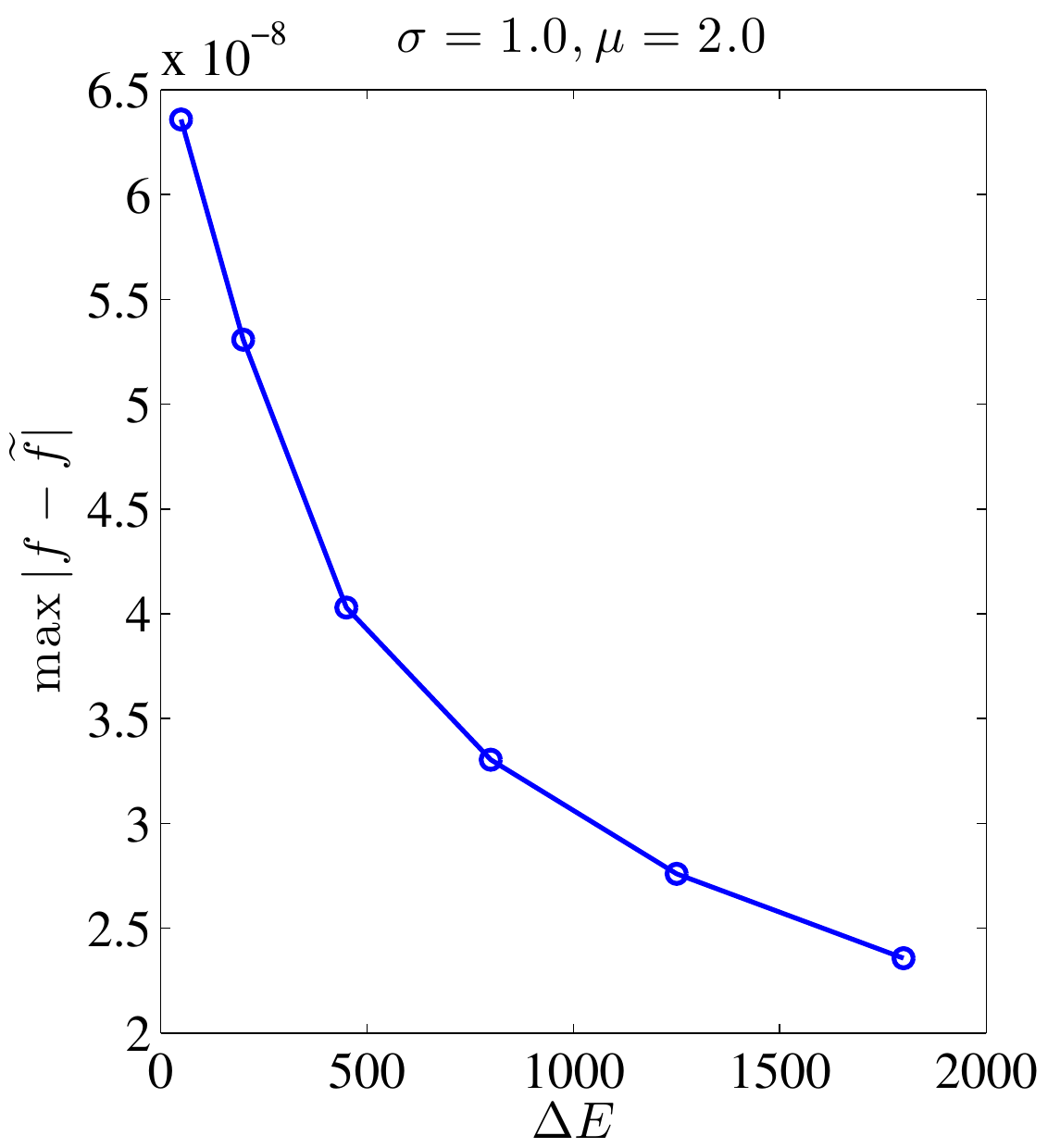}}
  \end{center}
  \caption{Max norm error of the LSS operator for (a)
  $\mu=2.0,\Delta E=199.89$ and
  increasing value of $\sigma$. (b) $\mu=2.0,\sigma=1.0$ and
  increasing value of $\Delta E$.}
  \label{fig:fErrorSigma}
\end{figure}

%Accuracy of the projected operator with truncation.
%
%The error of the projection operator with SVD truncation.
%
%

So far the numerical results are obtained for the divide-and-conquer
approximation to the LSS operator with $\wt{\tau}=0$.  Next we
apply the SVD truncation to obtain the LSS basis set
$\{U_{\kappa}\}_{\kappa=1}^{M}$ for varying SVD relative truncation criterion.
In our numerical experiments, we use $\tau$ as the \textit{relative} SVD
truncation criterion with respect to the largest singular value of
$\wt{S}_{\kappa}$.  Fig.~\ref{fig:fTolError} shows the error of the
approximation to the LSS operator with $\tau$ being $0.001,0.01,0.1$,
respectively. As indicated in Eq.~\eqref{eqn:localSVD}, the max norm
error of the approximation of the LSS operator is approximately
proportional to $\tau$, as $\tau$ becomes dominant in
Eq.~\eqref{eqn:maxerror}.

\begin{figure}[h]
  \begin{center}
    \subfloat[]{\includegraphics[width=0.3\textwidth]{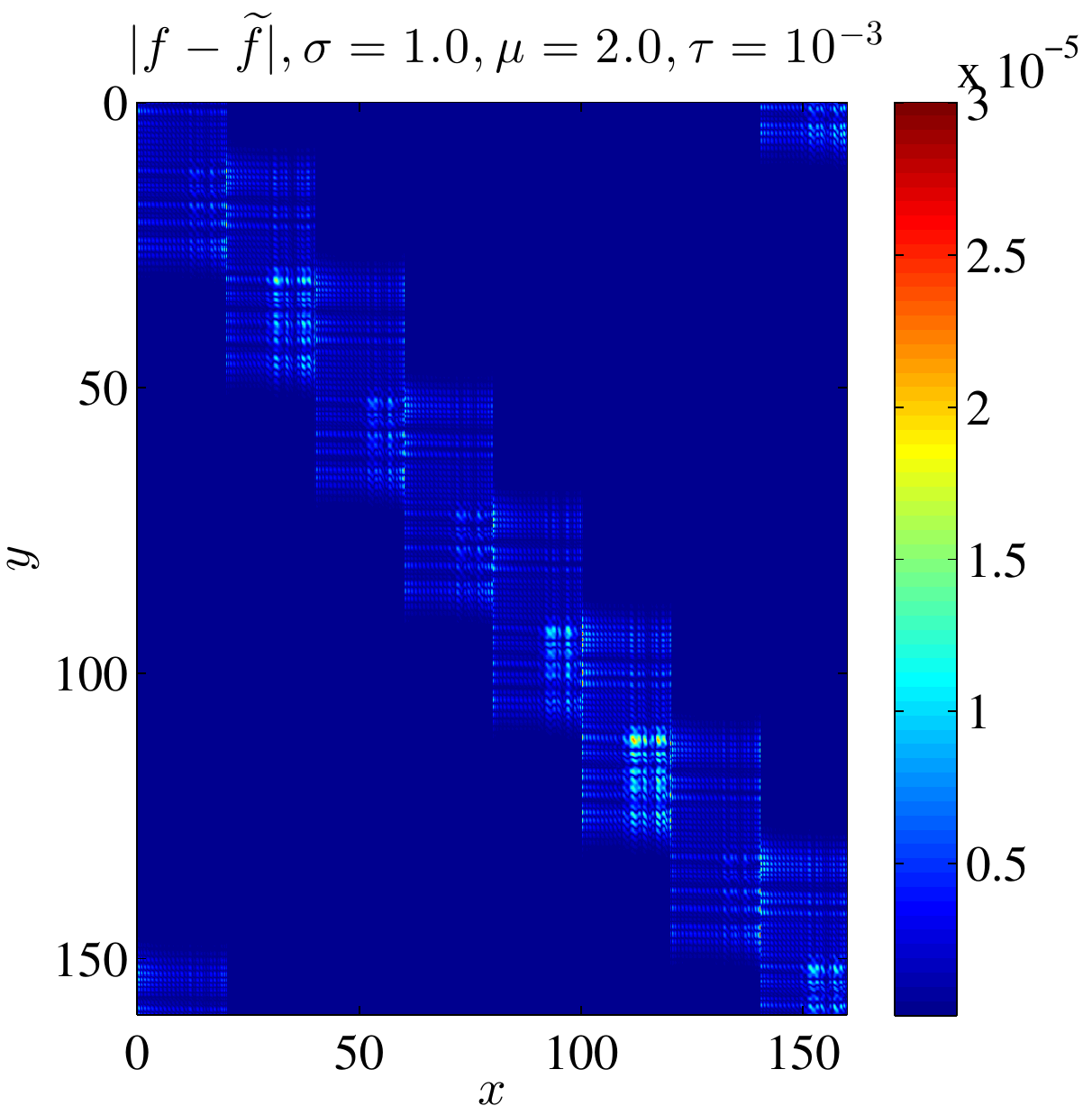}}
    \subfloat[]{\includegraphics[width=0.3\textwidth]{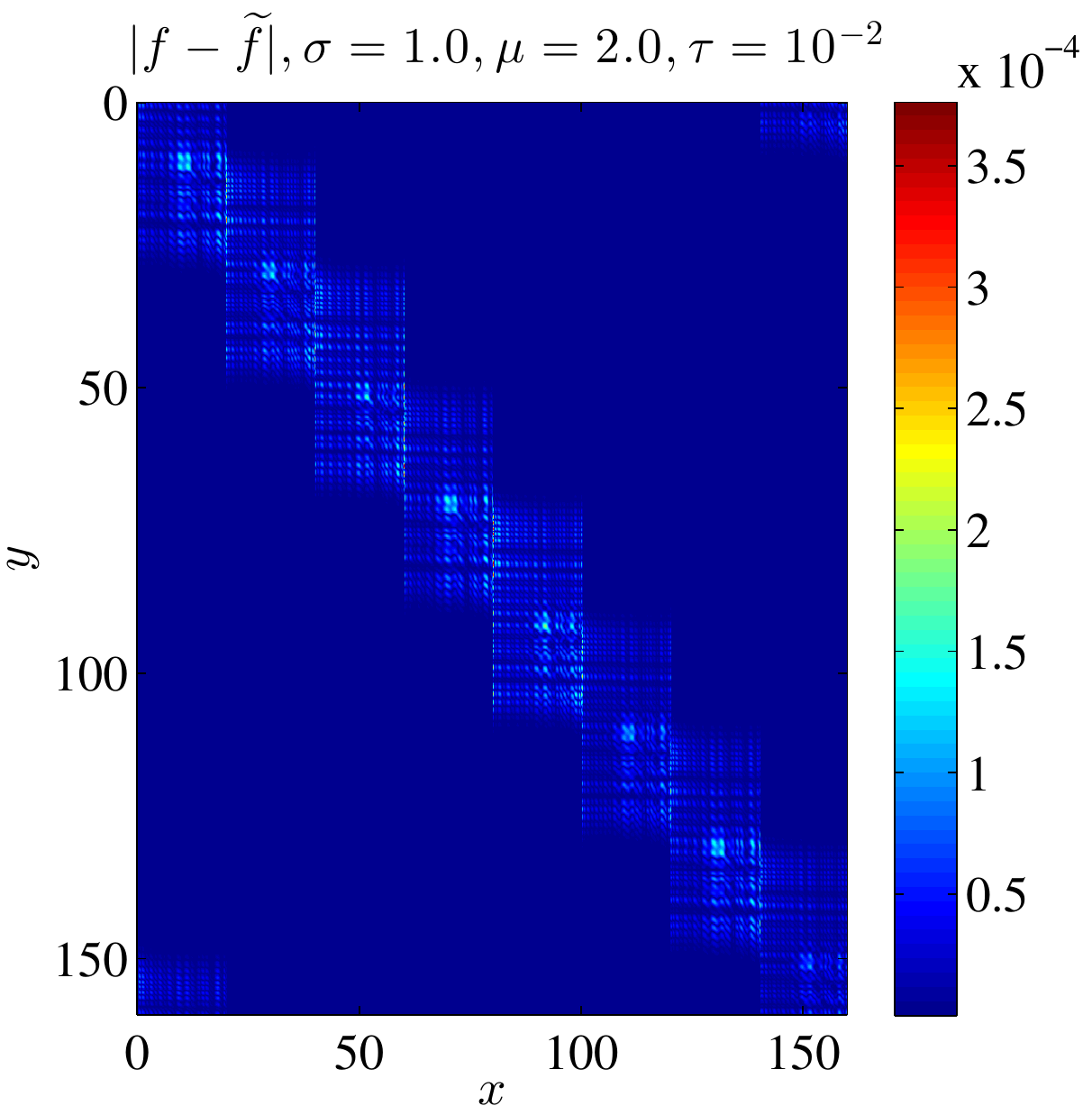}}
    \subfloat[]{\includegraphics[width=0.3\textwidth]{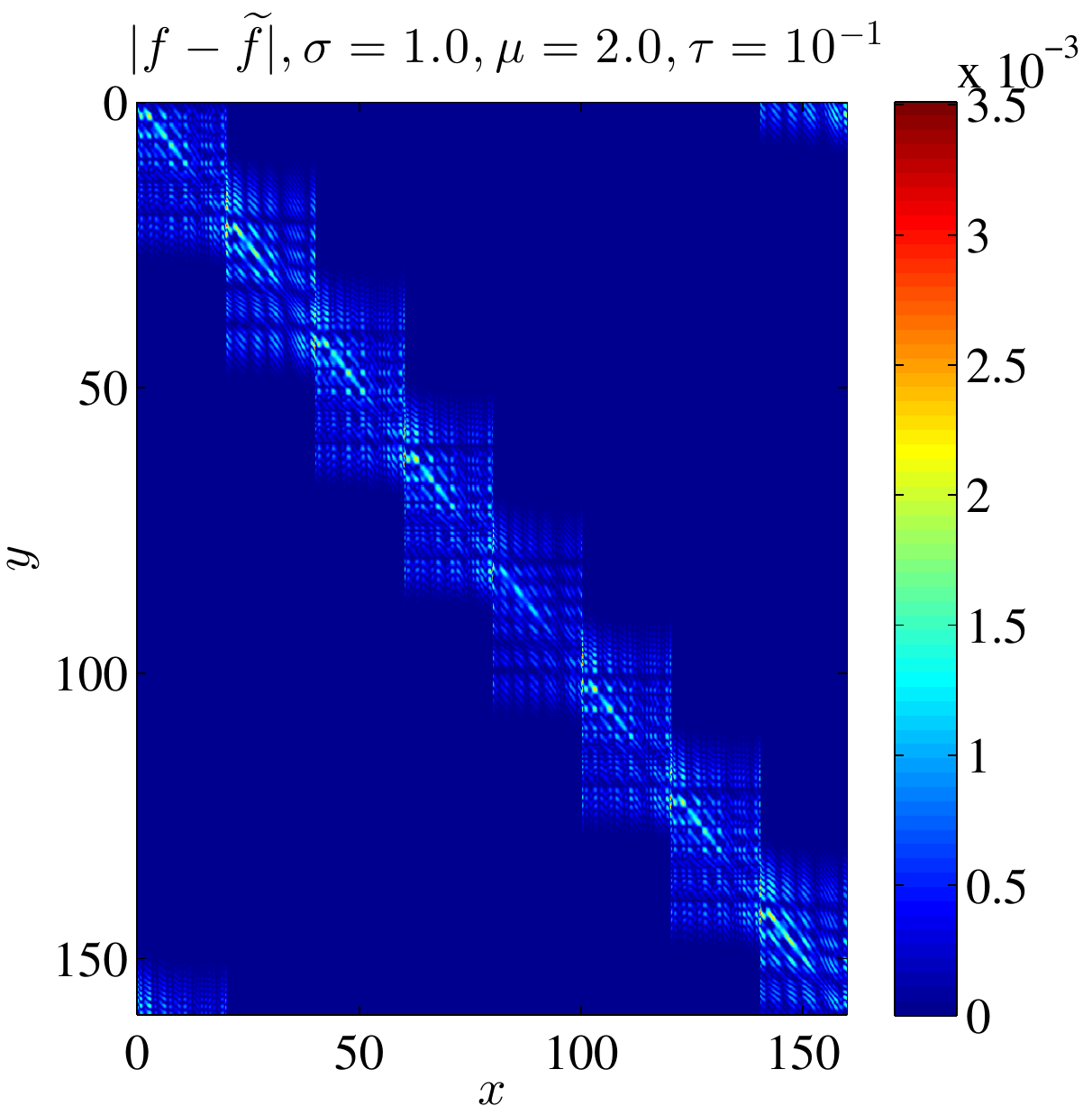}}
  \end{center}
  \caption{Error of the divide-and-conquer approximation to the LSS
  operator with $\sigma=1.0,\mu=2.0$ and different
  SVD relative truncation criterion (a)
  $\tau=10^{-3}$ (b) $\tau=10^{-2}$ (c) $\tau=10^{-1}$.}
  \label{fig:fTolError}
\end{figure}

The LSS basis set comes from the SVD decomposition of $\wt{f}$ on each
element.  Fig.~\ref{fig:basis1D} (a) shows the $1$-st LSS basis function
on two elements $\kappa=2$ and $\kappa=6$, respectively, and
Fig.~\ref{fig:basis1D} (b) shows the $5$-th LSS basis function on the
same two elements for $\mu=2.0,\sigma=1.0$.  It is clear that each LSS
basis function is well localized in each extended element
$Q_{\kappa}$ and its center is in $E_{\kappa}$.

\begin{figure}[h]
  \begin{center}
    \subfloat[]{\includegraphics[width=0.35\textwidth]{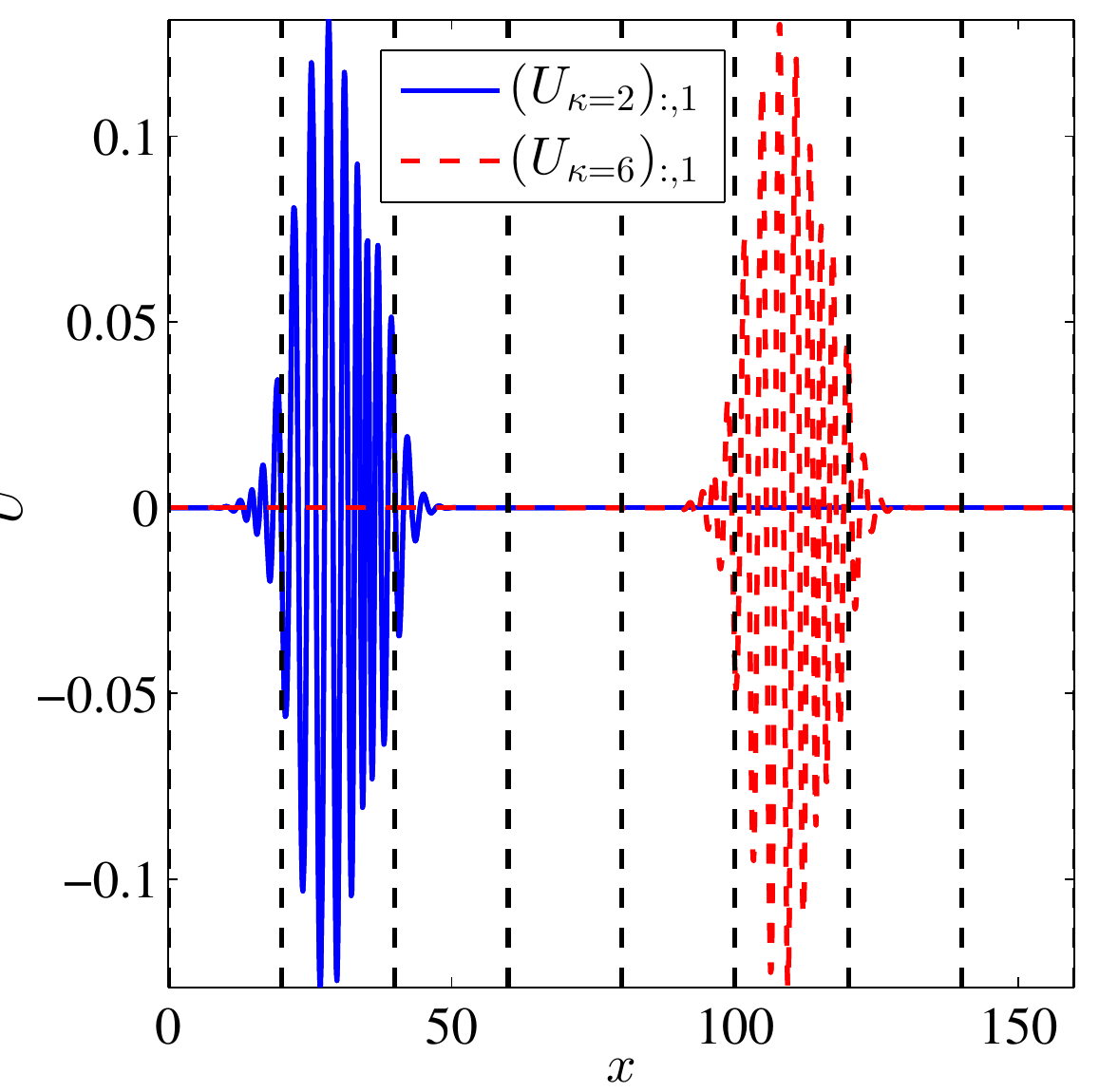}}\quad
    \subfloat[]{\includegraphics[width=0.35\textwidth]{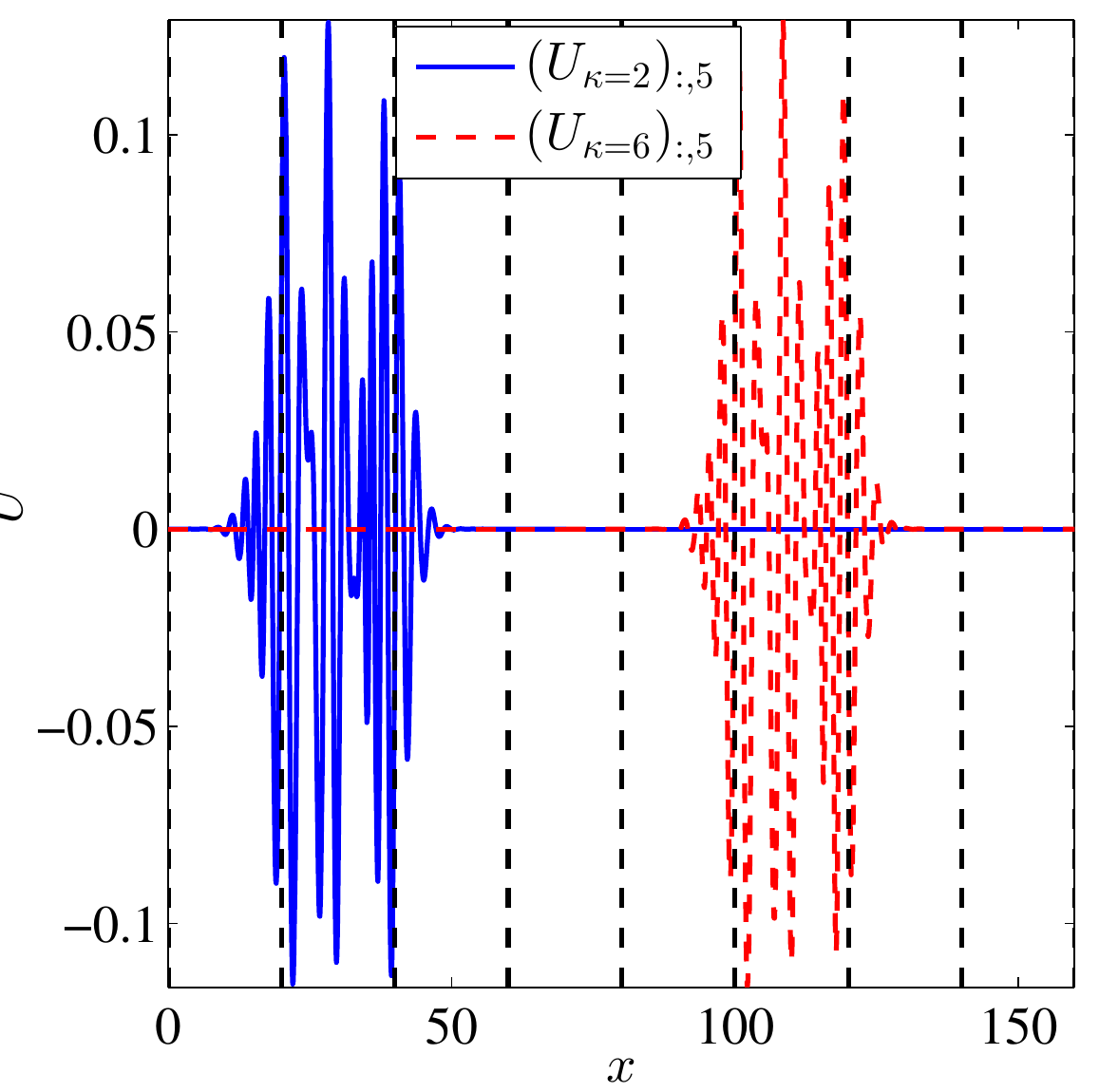}}
  \end{center}
  \caption{Example of the LSS basis function on two elements $\kappa=2$
  and $\kappa=6$ for (a) the $1$-st LSS basis function and (b) the
  $5$-th LSS basis function.}
  \label{fig:basis1D}
\end{figure}

Fig.~\ref{fig:fTolError} seems to suggest that in order to accurately
compute the interior eigenvalues, a very tight SVD criterion $\tau$ is
needed.  However, we note that many of the LSS basis functions associated with
the small singular values actually corresponds to the tail of the
Gaussian function in~\eqref{eqn:gaussian} which are away from $\mu$.
Therefore in order to compute the interior eigenvalues near $\mu$
accurately, it is possible to use a much larger value of $\tau$. 
Fig.~\ref{fig:errorEigTau} (a) shows the difference between the
$24$ eigenvalues of $A$ within the interval
$(\mu-0.5\sigma,\mu+0.5\sigma)$ and the corresponding Ritz values of $A$ with
$\tau=0.1$.  The computed Ritz values are highly accurate and the
maximum error is under $5\times 10^{-6}$ even though a large SVD
truncation criterion $\tau$ is used.  Section~\ref{sec:interior} discusses the
identification of spurious eigenvalues using the residual for each
computed Ritz value.  Indeed within the interval
$(\mu-0.5\sigma,\mu+0.5\sigma)$ we find $25$ Ritz values, and the
$1$ additional Ritz value should be a spurious eigenvalue.
Fig.~\ref{fig:errorEigTau} (b) shows 
$\norm{R_{j}}_{2}$ for each Ritz value, and we identify that the
$11$-th Ritz value has a much larger residual than the rest and should
be removed. After removing this spurious eigenvalue, the remaining Ritz
values become accurate approximation to the eigenvalues as indicated in
Fig.~\ref{fig:errorEigTau} (a).

\begin{figure}[h]
  \begin{center}
    \subfloat[]{\includegraphics[width=0.3\textwidth]{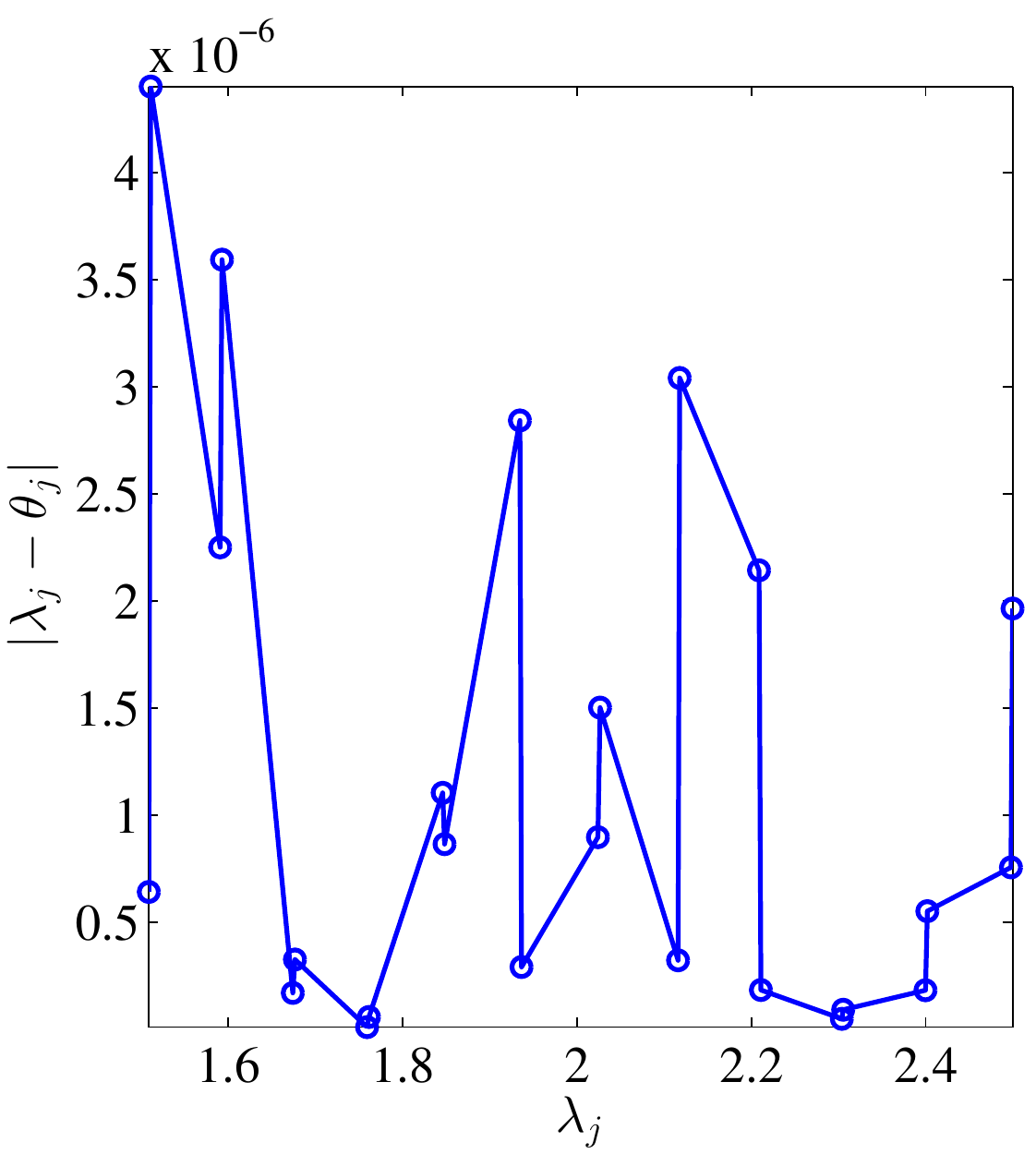}}\quad
    \subfloat[]{\includegraphics[width=0.3\textwidth]{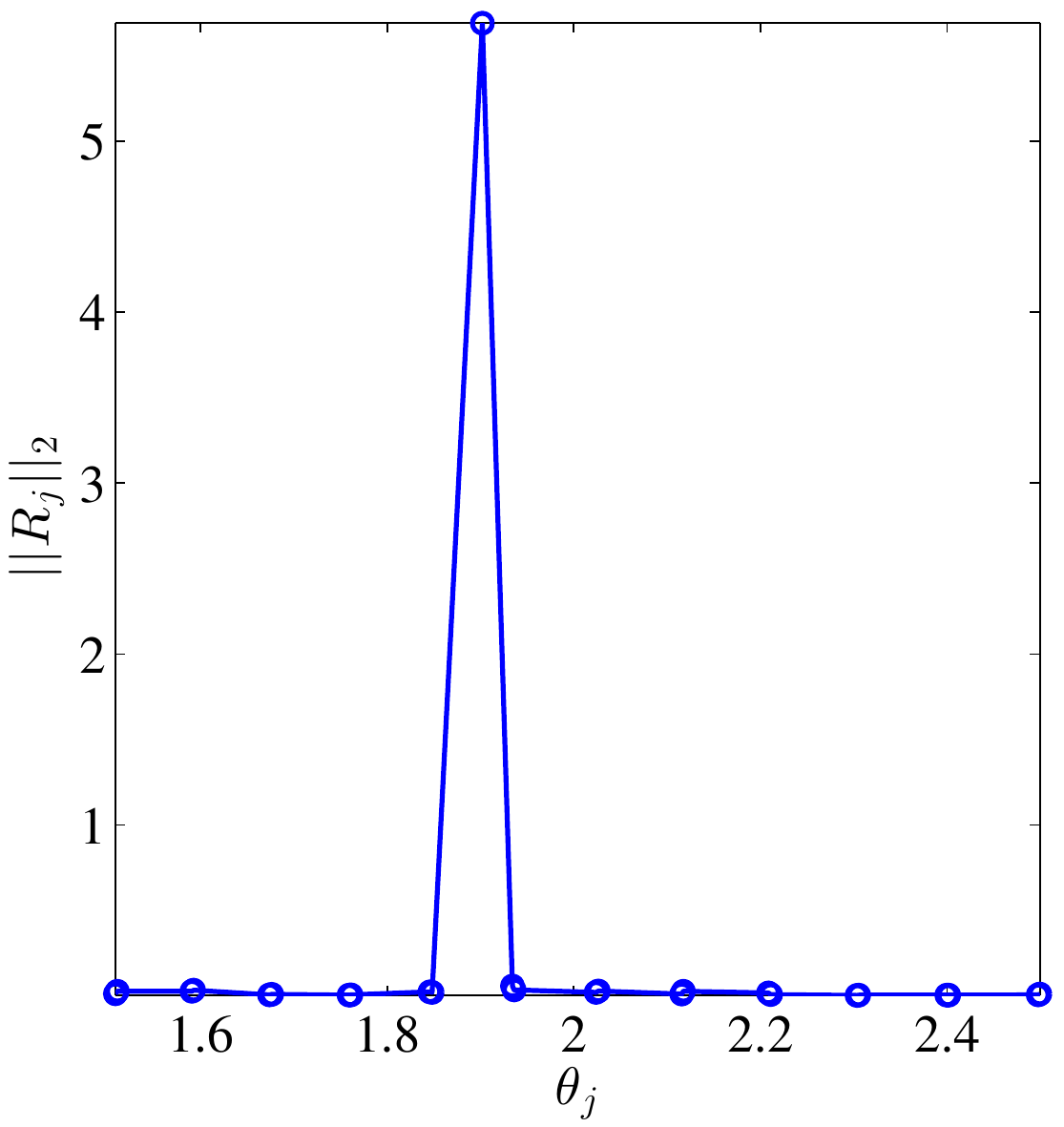}}
  \end{center}
  \caption{(a) Difference between the $24$ eigenvalues and corresponding Ritz values within
  the interval $(\mu-0.5\sigma,\mu+0.5\sigma)$ with
  $\sigma=1.0,\mu=2.0$. (b) The $2$-norm of the residual for each
  of the $25$ Ritz values. The $11$-th Ritz value has a large residual
  norm and is a spurious eigenvalue.}
  \label{fig:errorEigTau}
\end{figure}

While the accuracy of the divide-and-conquer
approximation to the LSS operator improves as the SVD truncation
criterion $\tau$ decreases, using a very small value of $\tau$ may
result in ill-conditioned projection matrices $A_{U}$ and $B_{U}$, i.e.
some of the LSS basis functions can be approximately represented as the
linear combination of other LSS basis functions.
Fig.~\ref{fig:cond_nb_tau} (a) shows the condition number of $A_{U}$,
$B_{U}$ with respect to $\tau$.  The condition numbers are below
$10^{4}$ when $\tau \ge 0.1$, and increase very rapidly to
$10^{13}$ for $\tau=10^{-3}$.  In the latter case, numerical results
obtained from the generalized eigenvalue solver cannot be trusted.
Decreasing $\tau$ also leads to increase of the size of the LSS basis set.
As $\tau$ decreases from $10^{-1}$ to $10^{-3}$, the number of LSS basis
functions increase from $87$ to $173$. The accuracy of the LSS basis set
for different values of $\tau$ is given in Table~\ref{tab:errorTau}.
When $\tau$ is too small, the number of computed Ritz values is less
than $24$ due to the very large condition number of the generalized
eigenvalue problem, and the difference between the eigenvalues and the
Ritz values is not a meaningful quantity to report and is reported as
N/A. The error of the
Ritz values reaches its minimum near $\tau=0.032$ at only $7.59\times
10^{-8}$, and then starts to increase as $\tau$ increases.  We observe
that even if $\tau=0.316$, the absolute (and relative) error of
the Ritz values is still within $0.2\%$. For this case the dimension of
the projected generalized eigenvalue problem is $62$, which is much
smaller compared to the dimension of $A$ which is $1600$.

\begin{figure}[h]
  \begin{center}
    \subfloat[]{\includegraphics[width=0.3\textwidth]{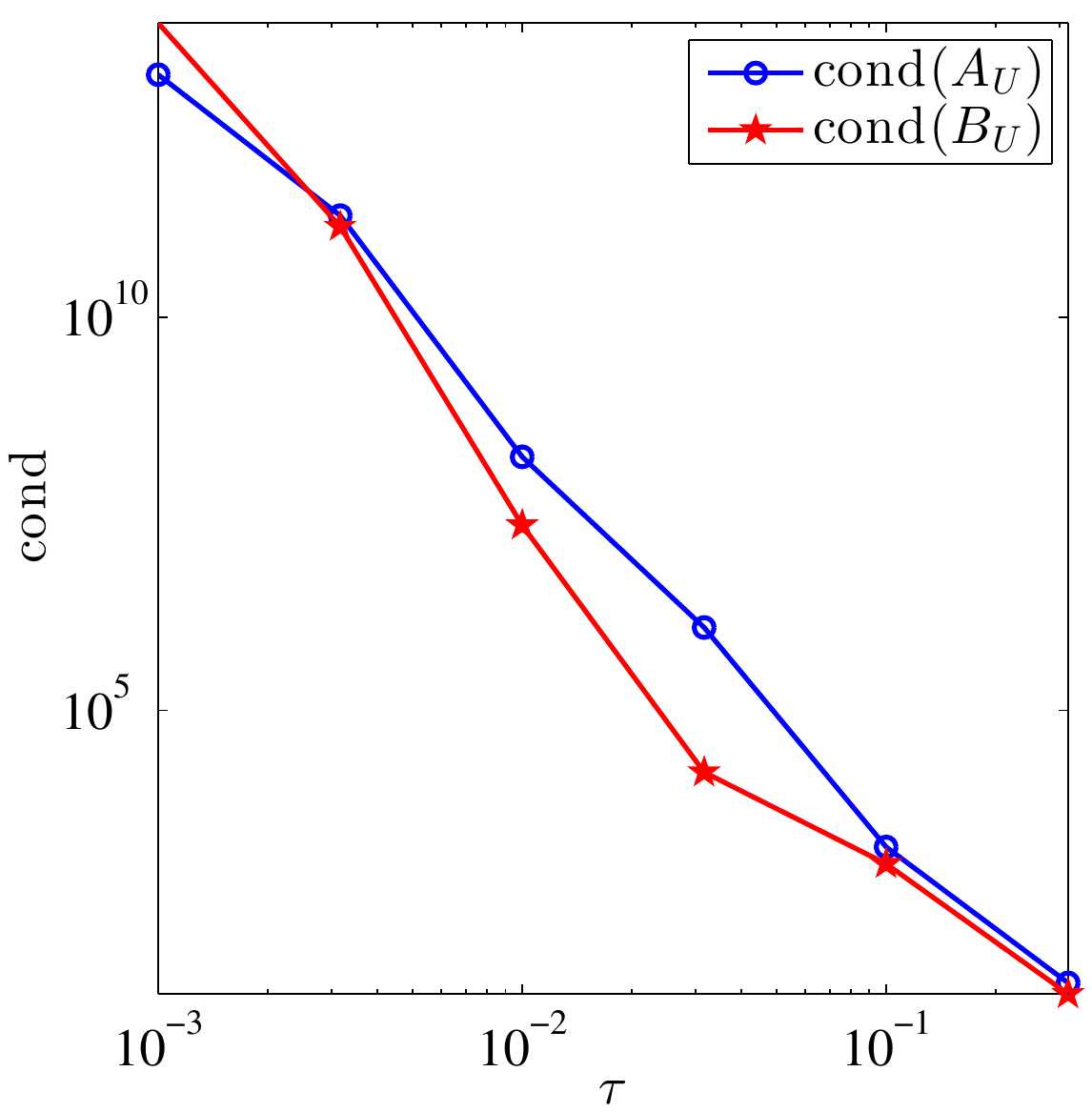}}\quad
    \subfloat[]{\includegraphics[width=0.3\textwidth]{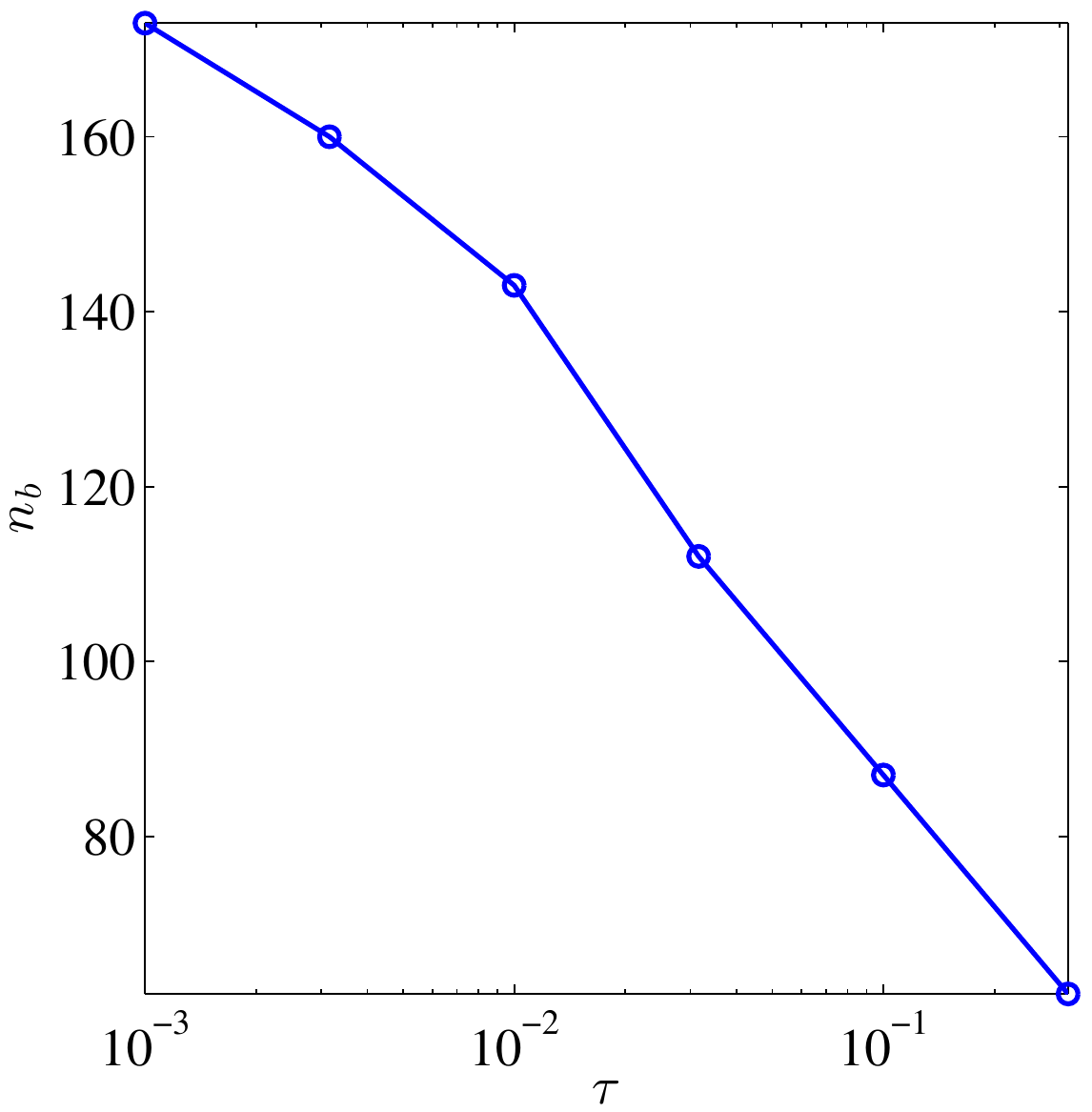}}
  \end{center}
  \caption{Error of the interior eigenvalue with $\mu=2.0$, $\sigma=1.0$
  and varying SVD relative truncation criterion $\tau$.}
  \label{fig:cond_nb_tau}
\end{figure}

\begin{table}[ht]
  \centering
  \begin{tabular}{c|c|c}
    \hline
    $\tau$ & \# Ritz values & $\max_{j}\abs{\lambda_{j}-\theta_{j}}$\\
    \hline
    $0.001$ & $1$  & N/A\\ 
    $0.003$ & $19$ & N/A\\ 
    $0.010$ & $24$ & $2.49\times 10^{-6}$ \\ 
    $0.032$ & $24$ & $7.59\times 10^{-8}$ \\ 
    $0.100$ & $24$ & $4.40\times 10^{-6}$ \\ 
    $0.316$ & $24$ & $1.50\times 10^{-3}$ \\ 
    \hline
  \end{tabular}
  \caption{The number of computed Ritz values in the interval 
  $(\mu-0.5\sigma,\mu+0.5\sigma)$ with $\sigma=1.0,\mu=2.0$
  (spurious eigenvalues removed).  If the number of Ritz values match
  the number of eigenvalues in the interval ($24$), then the third column gives the maximum difference
  between the eigenvalues and the Ritz values. Otherwise the third
  column gives N/A.}
  \label{tab:errorTau}
\end{table}

Even for the 1D simple example, the LSS basis set can be an efficient
way to compute interior eigenvalue problems compared to the solution of
the eigenvalue problem directly. For comparison of efficiency and
accuracy, MATLAB's sparse eigenvalue solver \textsf{eigs} is used for
the matrix $A$.  We acknowledge that
\textsf{eigs} may not be the best eigensolver to use for large interior
eigenvalue problems, and other choices such as preconditioned conjugate
gradient type of solvers, or Jacobi-Davidson type of solvers may give better
results.  We also remark that the current implementation of the LSS
solver is only for proof of principle, and many of its components can be
further optimized before a more thorough performance study is to be
performed.  Here we consider systems
of increasing size by changing $n_{w}$ in the potential function in
Eq.~\eqref{eqn:Vx} from $8$ to $256$. Correspondingly the number of grid
points $n$ increases from $1600$ to $51200$, and the number of elements
increases proportionally from $8$ to $256$.
$\mu=2.0,\sigma=1.0,\tau=3\times 10^{-2}$ is used for all systems to
compute the eigenvalues within the interval
$(\mu-0.5\sigma,\mu+0.5\sigma)$.  Fig.~\ref{fig:eiglarge1d} shows the
time for computing the interior eigenvalues near $\mu$ using MATLAB's
sparse eigenvalue solver \textsf{eigs} (``Global total''), and the time
using the LSS basis set (``LSS total''). The tolerance for
\textsf{eigs} is set to $10^{-5}$.  
The breakdown of the time cost for the LSS solver
includes the time for constructing the LSS basis set (``LSS basis''),
the time for assembling the projected matrix (``Assembly''), and the
time for solving the projected eigenvalue problem (``LSS solve'').
Fig.~\ref{fig:nnzAU} shows the sparsity pattern for $A_{U}$ for
$n=6400$, and the sparsity pattern for $B_{U}$ is by definition the
same.  The number of nonzero elements is $15.6\%$ of the total
number of elements in $A_{U}$.  The sparsity of the projected matrices
is not used in our example here, but can be exploited using alternative methods.

Since the size of the local problem is small, the local eigenvalue
problem on each $Q_{\kappa}$ is performed using MATLAB's dense
eigenvalue solver \textsf{eig}, and so is the solution of the
generalized eigenvalue problem for the projected matrix.  The time for
the global solver scales cubically with respect to $n$, and the
constructing the LSS basis and the assembly of the projected matrix
increases linearly with respect to $n$.  The solution of the generalized
eigenvalue problem also scales cubically with respect to $n$, and
therefore does not dominate in the LSS solver until $n=51200$.  The
cross-over time between the LSS solver and the global solver is around
$n=10000$. For $n=51200$, the LSS solver costs $46.6$ sec, which
is $11.2$ times faster than the global solver which costs $520.8$ sec.

Fig.~\ref{fig:eiglarge1d} (b) shows the accuracy of the LSS
solver.  The Ritz values remain as accurate approximation to the
eigenvalues as the number of eigenvalues in the interval increases from
$24$ to $706$.

%We have observed that the accuracy of the interior eigenvalues depend on
%the choice of the SVD criterion $\tau$.  Decreasing $\tau$ increases the
%accuracy of the interior eigenvalue estimation before it reaches the
%acuracy of the localized spectrum slicing operator $\varepsilon$.
%However, in practice there is another limiting factor for the choice of
%$\tau$, which is the condition number of the basis functions.
%Fig.~\ref{fig:} reports the $SVD$ criterion tau, together with the
%total number of significant basis functions, the accuracy and the
%condition number.  It is found that the condition number can increase
%significantly as $\tau$ decreases.  This introduces another numerical
%stability problem.  

\begin{figure}[h] \begin{center}
    \subfloat[]{\includegraphics[width=0.33\textwidth]{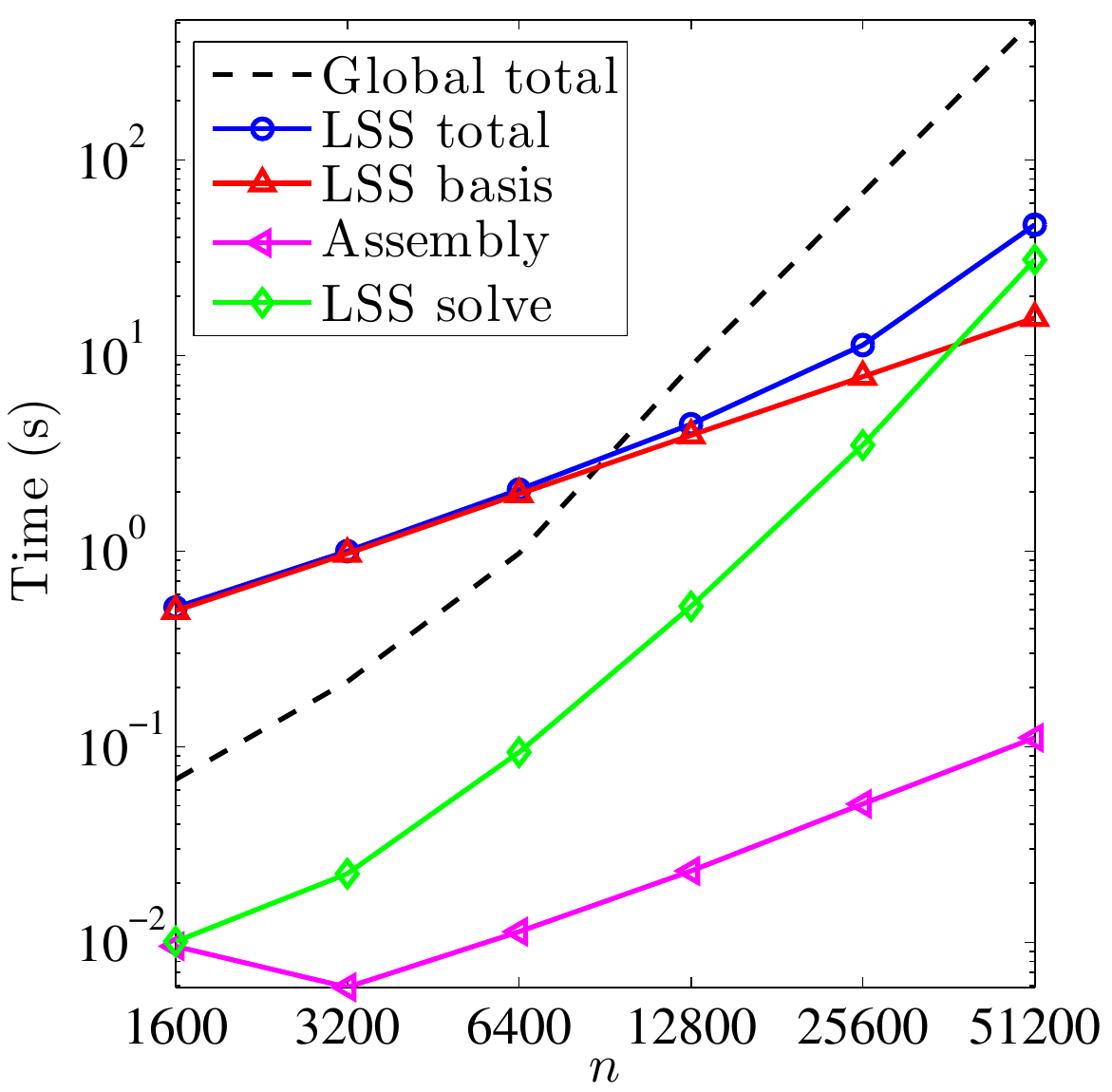}}\quad
    \subfloat[]{\includegraphics[width=0.31\textwidth]{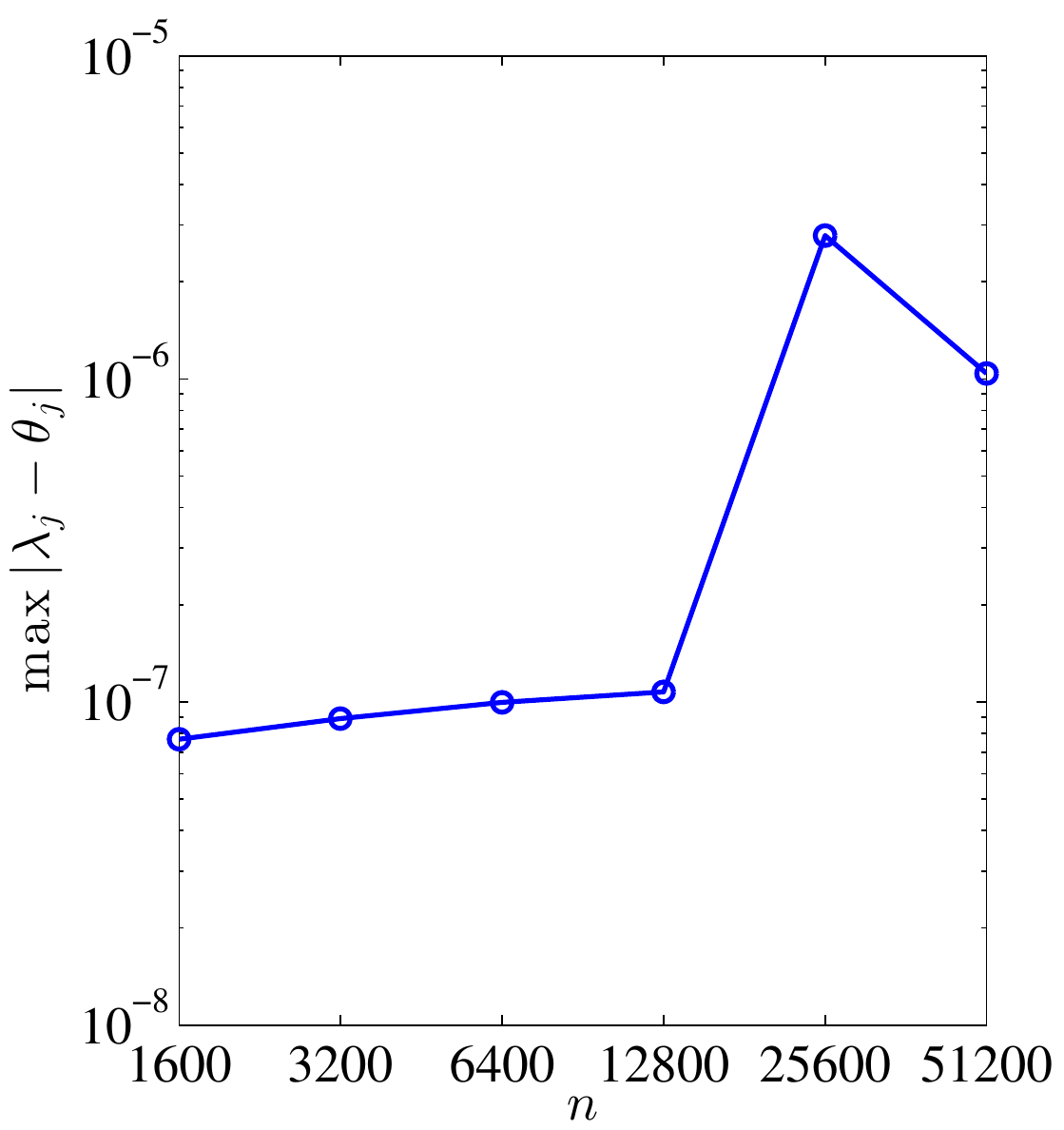}}
  \end{center}
  \caption{(a) Comparison of time cost between the global solver and the LSS
  solver for 1D interior eigenvalue problem with
  increasing system size. See text for details of the comparison. (b)
  Maximum error of the Ritz values.}
  \label{fig:eiglarge1d}
\end{figure}

\begin{figure}[h] 
  \begin{center}
    \includegraphics[width=0.33\textwidth]{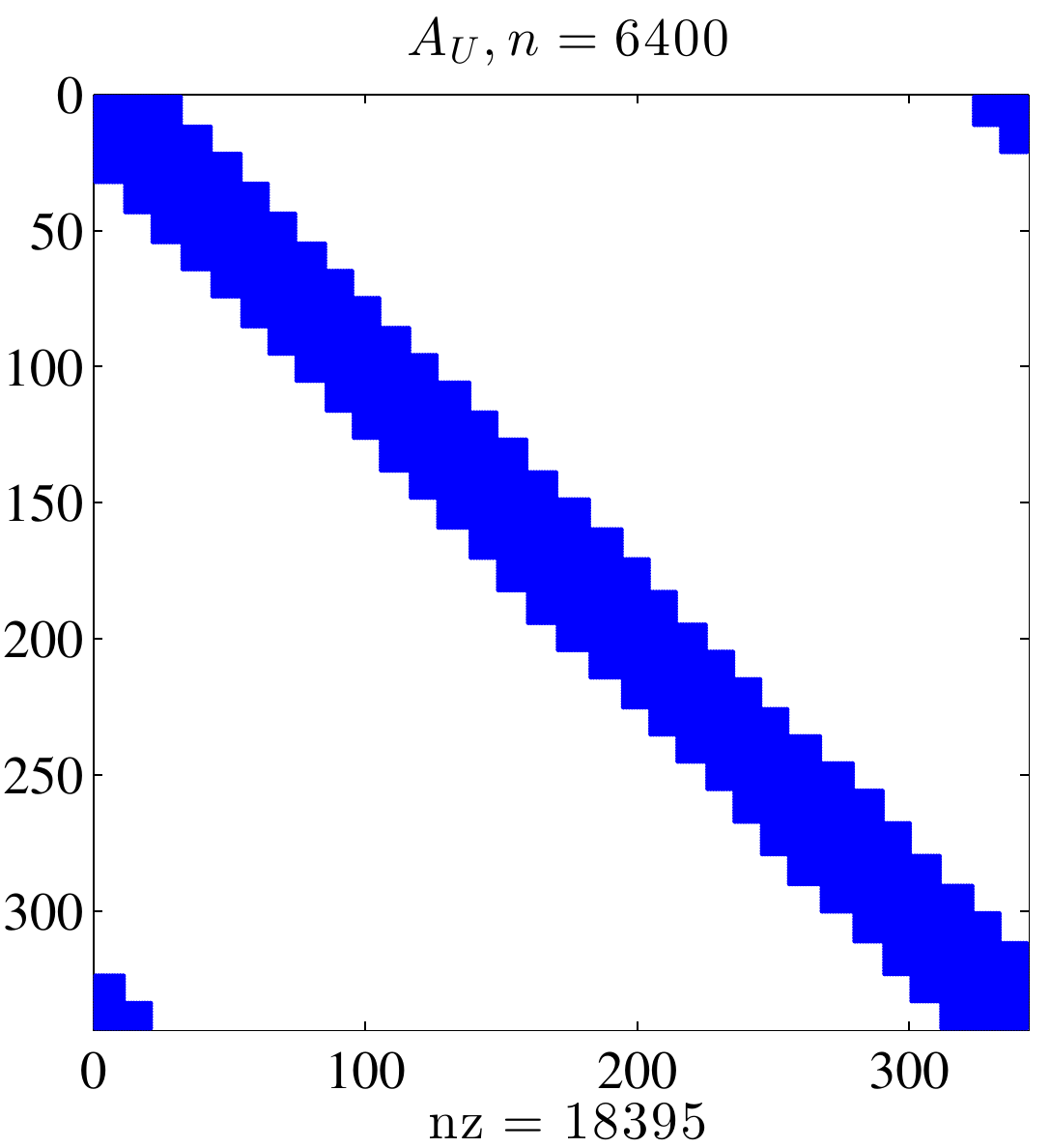}
  \end{center}
  \caption{Sparsity pattern for $A_{U}$ for $n=6400$.}
  \label{fig:nnzAU}
\end{figure}

\subsection{Two-dimensional case}

The setup of the 2D example is similar to that in 1D.  The global domain
is $\Omega=[0,L]\times [0,L]$, and the Laplacian operator is discretized
using a 5-point finite difference stencil.  The grid spacing is
chosen to be $h=1.0$.
The potential function $V(x,y)$ is given by sum of periodized
exponential functions \REV{with random perturbation in terms of heights,
widths and positions of the exponential functions. This can be viewed as
a model potential for a crystal under thermal noise.}  One realization
of this potential is given in Fig.~\ref{fig:V2D}.  
Let the number of elements $M$ is a square number and the number of grid
points $n$ is divisible by $M$. Then all $n$ grid points (vertices) are
uniformly partitioned into $\sqrt{M}\times \sqrt{M}$ elements. We also
assume each extended element $Q_{\kappa}$ contains $E_{\kappa}$ and its
$8$ nearest neighbor elements. Fig.~\ref{fig:V2D} shows the partition of
the 2D domain into $8\times 8=64$ elements separated by black dashed
lines.

%Here we demonstrate the idea for 2D system for the operator
%$A=-\Delta+V$ discretized using 5-point finite difference stencil.
%Fig.~\ref{fig:V2D} shows the potential of $V(x,y)$ which is given by a
%sum of Gaussian functions.  The height of the Gaussian is random at each
%site is $\mathcal{N}(-5.0,1.0)$, and the width of the Gaussian is given
%by $\mathcal{N}(4.0, 0.2)$.  

\begin{figure}[h]
  \begin{center}
    \includegraphics[width=0.35\textwidth]{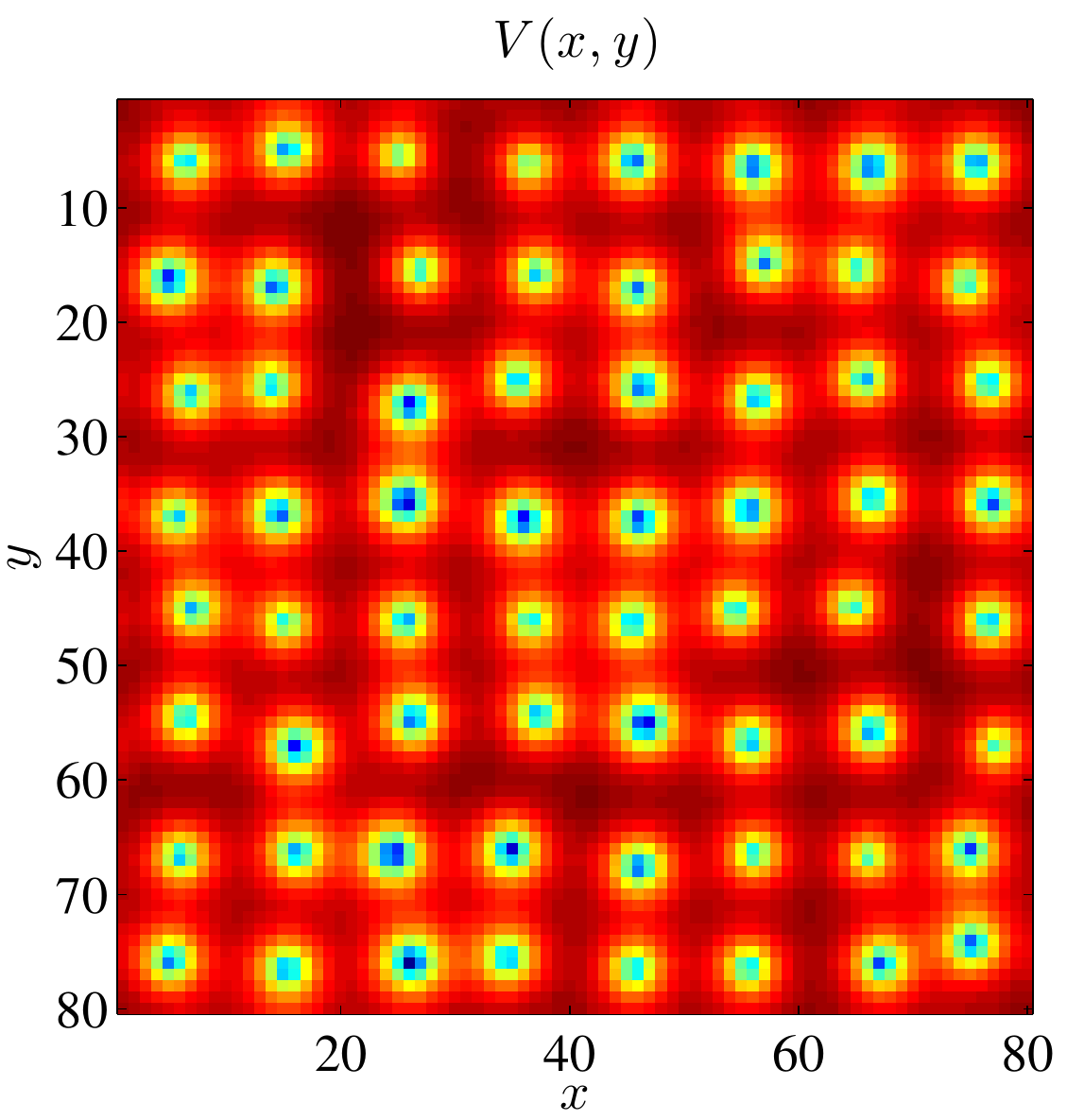}
  \end{center}
  \caption{One realization of the 2D potential.  The domain is
  partitioned into $8\times 8=64$ elements separated by black dashed
  lines.}
  \label{fig:V2D}
\end{figure}

\REV{We compare the accuracy of the LSS basis set by comparing the
eigenvalues within the interval $(\mu-\sigma,\mu+\sigma)$ with
$\mu=-1.0,\sigma=1.0$.} The SVD relative truncation criterion $\tau$ is set to be
$10^{-1}$.  Fig.~\ref{fig:accuracy2D} (a) shows the error of Ritz values
compared to all the $828$ eigenvalues within the interval, and the error
of all Ritz values is very small, within \REV{$7\times 10^{-5}$}.
Fig.~\ref{fig:accuracy2D} (b) shows the residual of the Ritz values.
For all the Ritz values the residual are below \REV{$7\times 10^{-3}$} and
no spurious eigenvalue is identified for this case. 

%Fig.~\ref{fig:2Daccuracy} shows the timing comparison of MATLAB's eigs
%function for computing the interior eigenvalue problems, and the timing
%comparison of LSS.  The number of elements is proportional to the length
%of the domain atlong each direction.  The eigenfunctions are computed
%locally using MATLAB's eig function due to the relatively small system
%size, and can be changed to more efficient local solvers.
%Fig.~\ref{fig:accuracy2D} shows the accuracy of the eigenvalues computed
%in the interior domain for $N_{x}=N_{y}=80$ corresponding to the
%potential Fig.~\ref{fig:V2D}.   $\sigma=1.0$ is used to construct the
%local basis functions.

\begin{figure}[h]
  \begin{center}
    \subfloat[]{\includegraphics[width=0.3\textwidth]{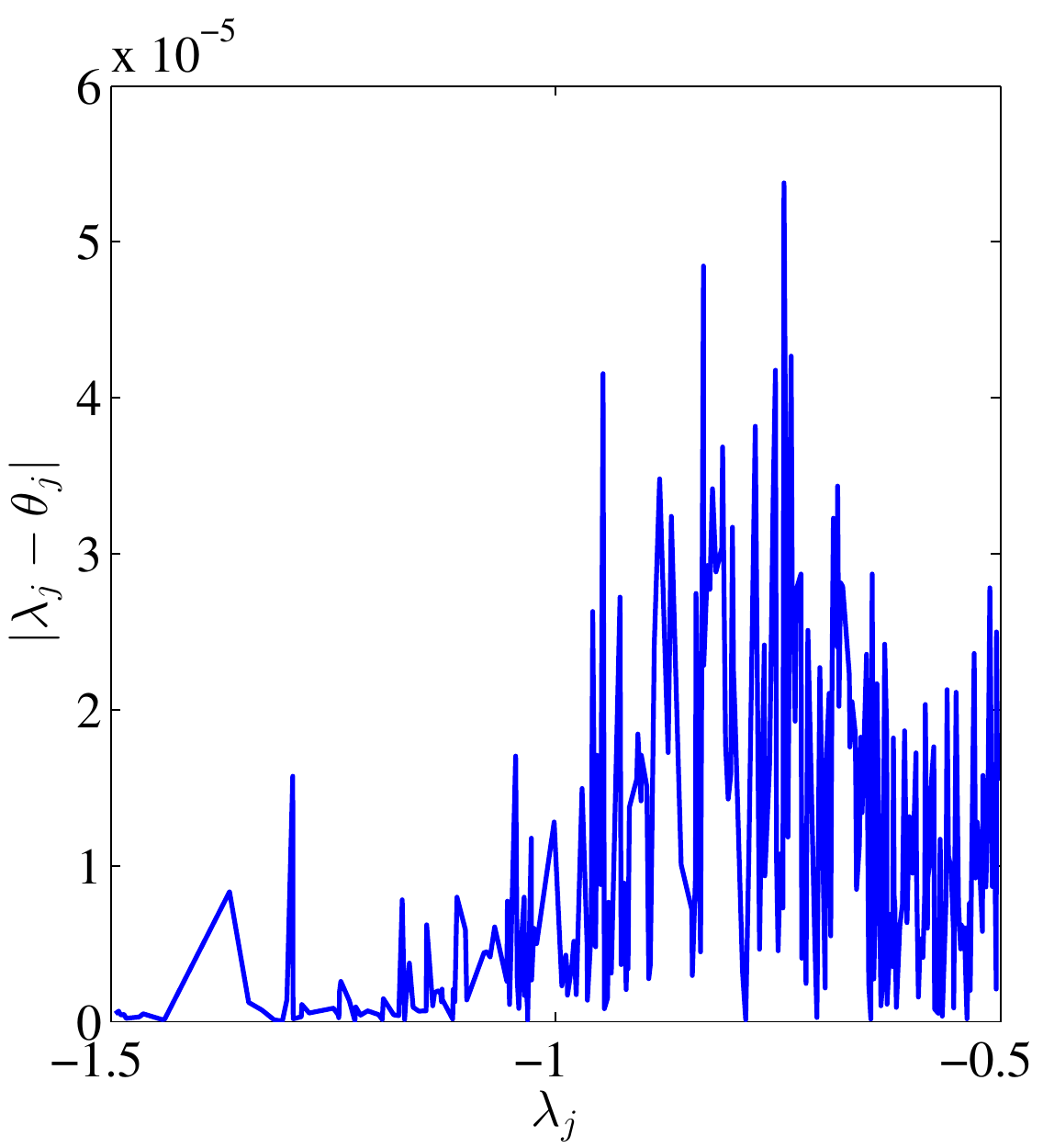}}\quad
    \subfloat[]{\includegraphics[width=0.3\textwidth]{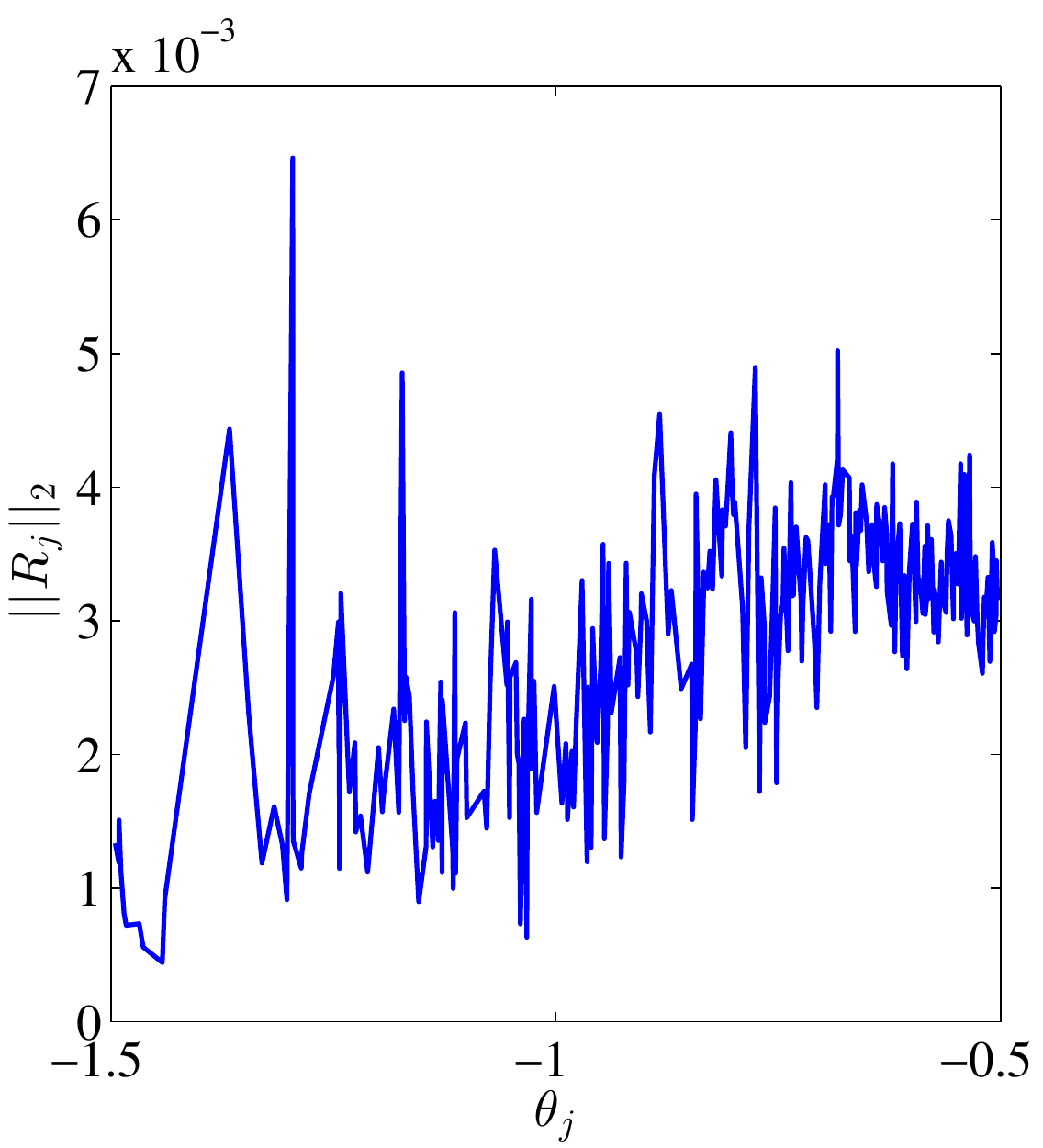}}
  \end{center}
  \caption{\REV{(a) Error of the Ritz values (b) The $2$-norm of residual
  corresponding to Ritz values for the 2D problem with
  $\mu=-1.0,\sigma=1.0,\tau=10^{-1}$.}}
  \label{fig:accuracy2D}
\end{figure}

Finally we demonstrate the performance of the LSS solver for a 2D
problem with increasing size.  The number of grid points $n$ increases
from $1600$ to $25600$, and the number of elements increases
proportionally from $16$ to $256$.
 Fig.~\ref{fig:eiglarge2d} shows the time
for computing the interior eigenvalues near $\mu$ using MATLAB's sparse
eigenvalue solver \textsf{eigs} (``Global total''), and the time using
the LSS basis set (``LSS total'').  The tolerance for \textsf{eigs} is
set to $10^{-5}$.  The breakdown of the LSS solver includes the time for
constructing the LSS basis set (``LSS basis''), the time for assembling
the projected matrix (``Assembly''), and the time for solving the
projected eigenvalue problem (``LSS solve'').  Again the local
eigenvalue problem on each $Q_{\kappa}$ is performed using MATLAB's
dense eigenvalue solver \textsf{eig}, and so is the solution of the
generalized eigenvalue problem for the projected matrix.  The crossover
point between the global solver and the LSS solver is around $n=3000$.
\REV{For $n=25600$, the LSS solver costs $143$ sec, which is $8.3$
times faster than the global solver which costs $1183$ sec.}

Fig.~\ref{fig:eiglarge2d} (b) shows the accuracy of the LSS
solver.  The Ritz values remain as accurate approximation to the
eigenvalues as the number of eigenvalues in the interval increases with
respect to the system size and no spurious eigenvalue is observed for
all cases.

\begin{figure}[]
  \begin{center}
    \subfloat[]{\includegraphics[width=0.33\textwidth]{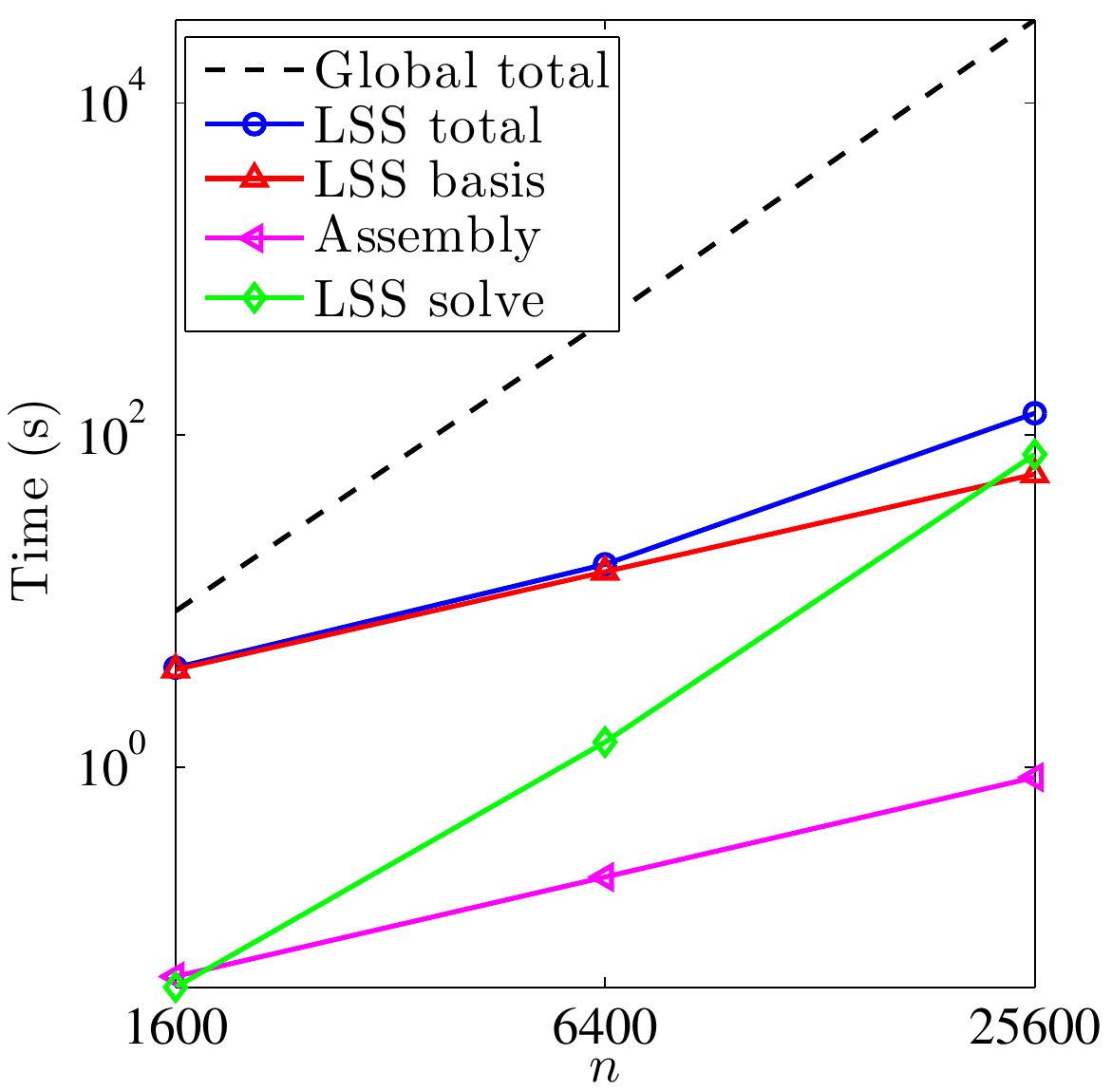}}
    \subfloat[]{\includegraphics[width=0.31\textwidth]{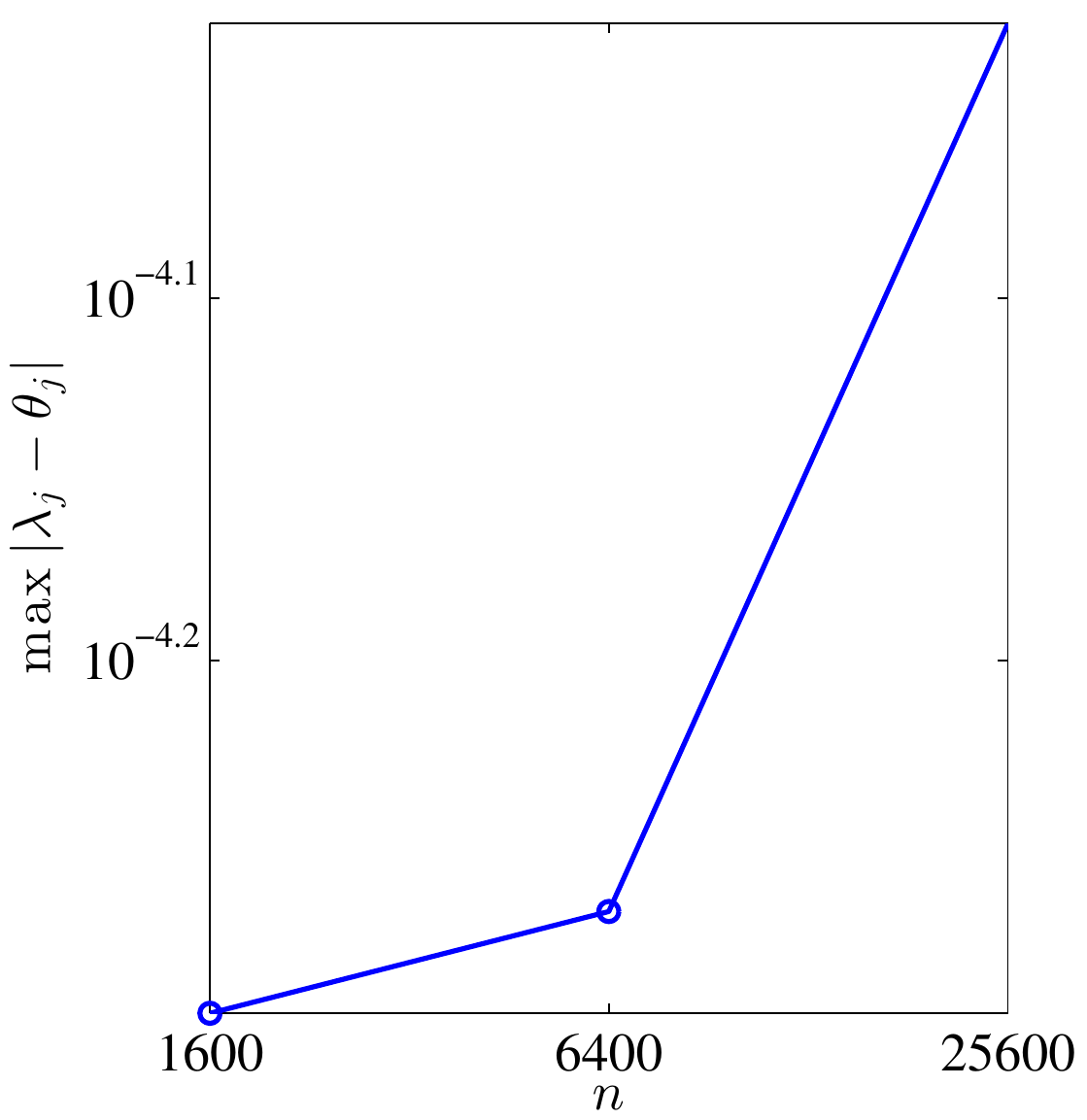}}
  \end{center}
  \caption{\REV{(a) Comparison of time cost between the global solver and the LSS
  solver for 2D interior eigenvalue problem with
  increasing system size. See text for details of the comparison. (b)
  Maximum error of the Ritz values.}}
  \label{fig:eiglarge2d}
\end{figure}

\REV{
\subsection{Sparse matrix with general sparsity pattern}

%As discussed in section~\ref{subsec:lssgeneral}, the LSS basis set for
%general sparse matrices can be performed similarly by first performing
%an approximate $M$-cut of the graph associated with the matrix
%$A$. 
For a general sparse matrix, we take the \texttt{turon-m} matrix 
from the University of Florida matrix collection~\cite{FloridaMatrix}.  The
dimension of the matrix is 189924, with 1690876 number of nonzeros. 
The LU factorization procedure for this matrix is relatively expensive.
Using the approximate minimum degree 
(AMD) ordering strategy provided through the \texttt{symamd} command in
MATLAB~\cite{DavisGilbertLarimoreEtAl2004}. The number of nonzeros in
$L$ and $U$ are $364176421$ with a fill-in ratio (i.e. the ratio between
the number of nonzeros in $L,U$ and the number of nonzeros in $A$) is $215$.
The LU factorization takes $952$ sec, and each triangular solve
$U^{-1}(L^{-1}b)$ for a random right hand side vector $b$ takes $0.52$
sec, compared to each matrix vector multiplication $Ax$ which takes
$0.006$ sec.  The spectral radius of this matrix is $86$.  The sparsity
pattern of this matrix, together with the
histogram of the eigenvalues (unnormalized spectral density) in the
interval $(1,7)$ is given in Fig.~\ref{fig:turon_m} (a) (b),
respectively.

\begin{figure}[h]
  \begin{center}
%    \subfloat[]{\includegraphics[width=0.33\textwidth]{spy_stokes128}}
%    \subfloat[]{\includegraphics[width=0.31\textwidth]{spec_stokes128}}
    \subfloat[]{\includegraphics[width=0.31\textwidth]{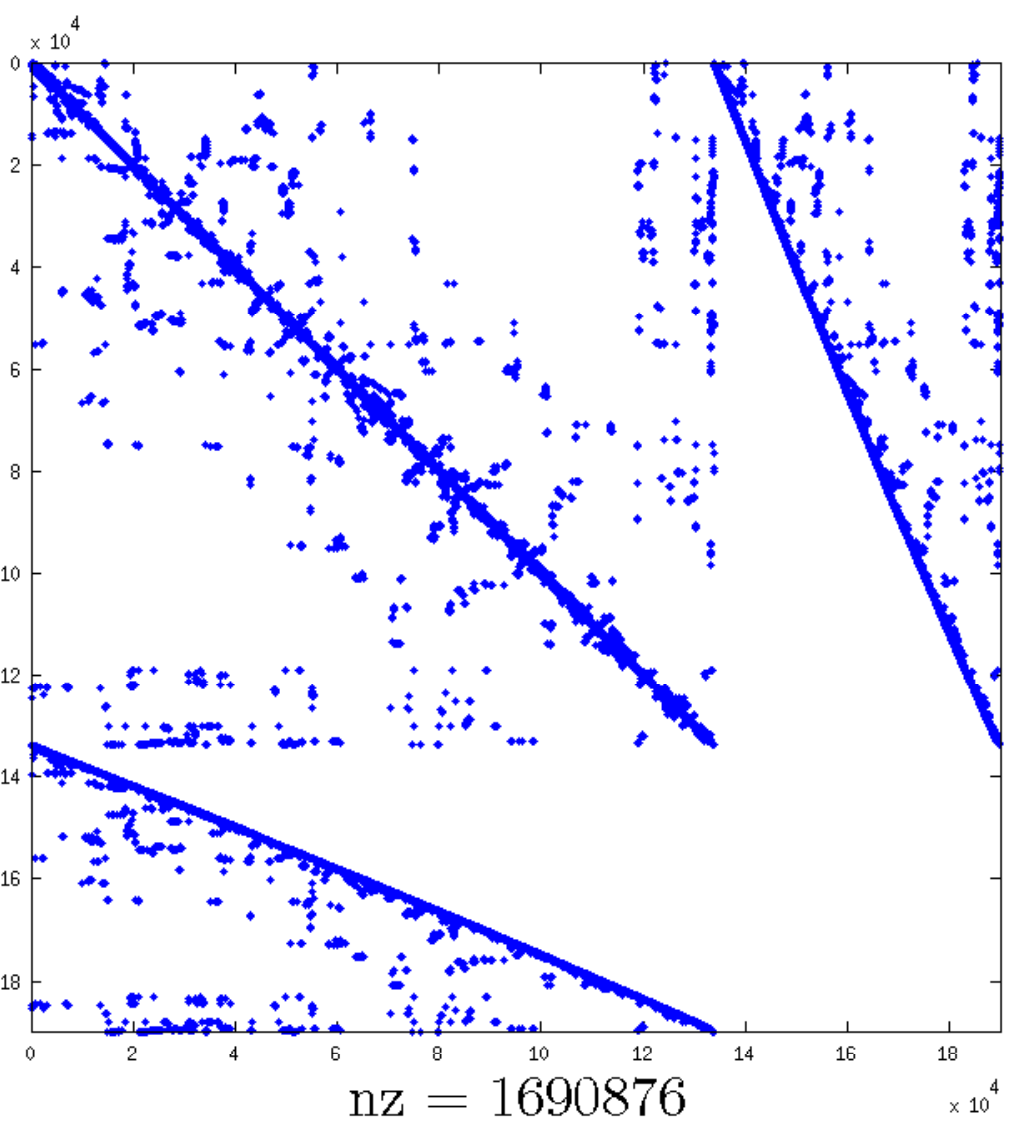}}
    \subfloat[]{\includegraphics[width=0.33\textwidth]{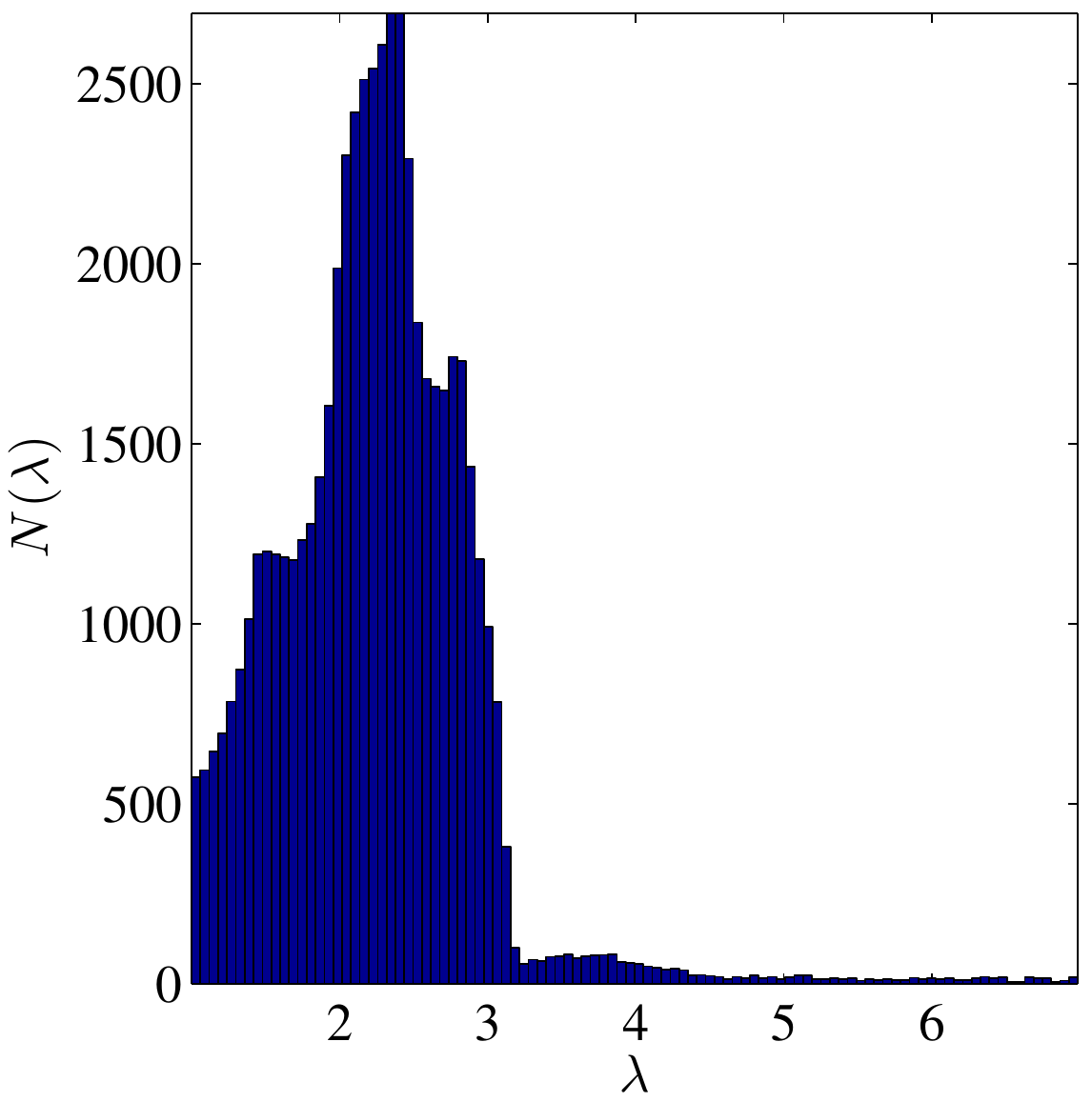}}
    \subfloat[]{\includegraphics[width=0.32\textwidth]{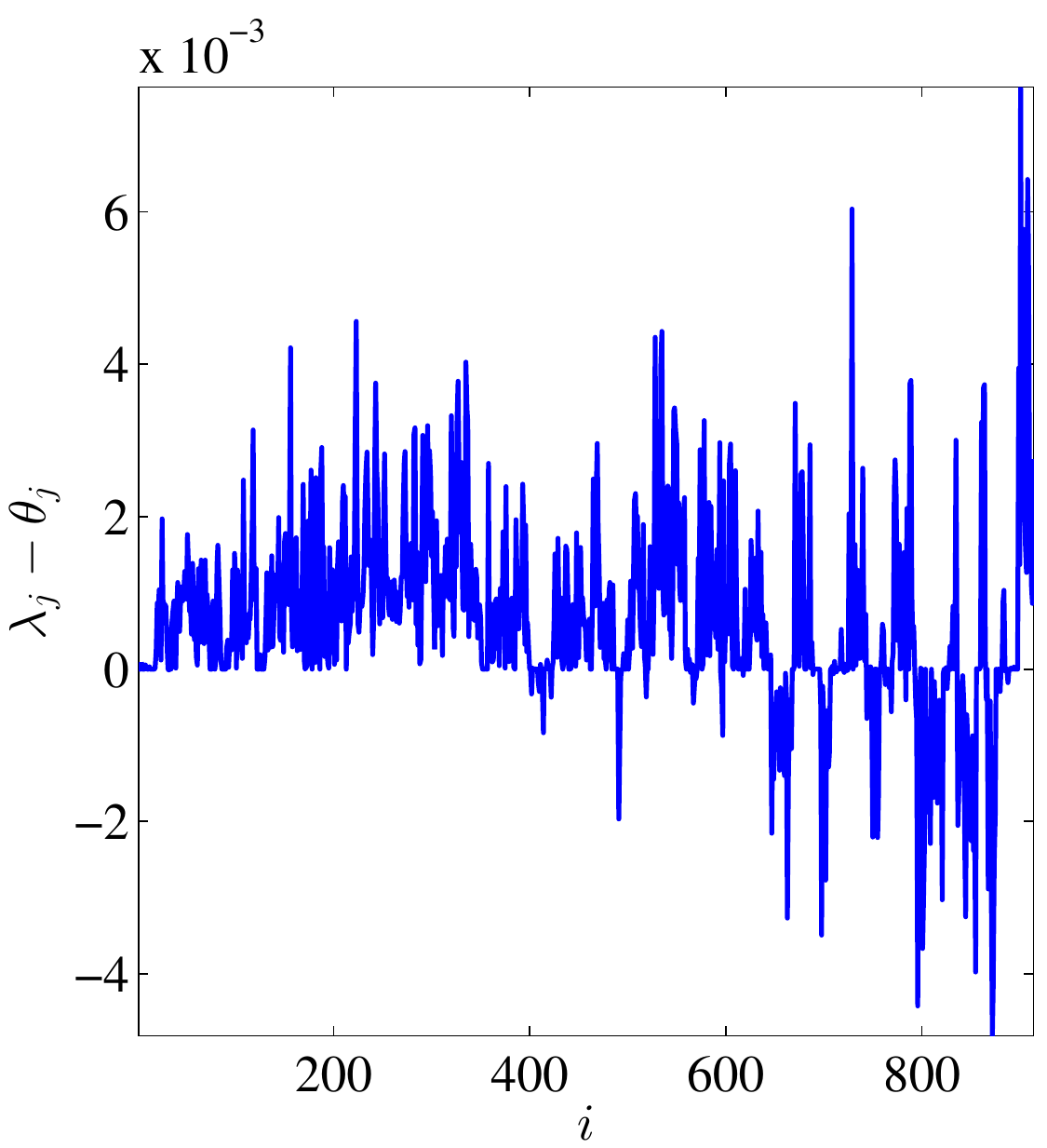}}
  \end{center}
  \caption{(a) Sparsity pattern and (b) histogram of the eigenvalues in
  the interval $(1,7)$ (c) Accuracy of the Ritz values for the interior
  eigenvalues in the interval $(3.5,4.5)$ of the
  \texttt{turon-m} matrix.}
  \label{fig:turon_m}
\end{figure}

In order to apply the LSS method to this unstructured matrix, we use the
strategy in section~\ref{subsec:lssgeneral} and use the
METIS~\cite{KarypisKumar1998}
package interfaced by the \texttt{metismex}
program\footnote{{https://github.com/dgleich/metismex}} with MATLAB
for generating the graph partitioning map $\xi$. 

We set $\mu=4.0,\sigma=0.5$. As in Fig.~\ref{fig:turon_m} (b), $\mu=4.0$ 
indeed corresponds to interior eigenvalues. We select this region mainly
because the spectral density is relatively low so that the computation
can be treated on a single computational core.  The matrix is
partitioned into $16$ elements using METIS. The matrix partition routine
is efficient and only takes $0.55$ sec. Due to the large size of the
submatrix 
on a single extended element, we use \textsf{eigs} to solve $500$
eigenvalues on each extended element with tolerance set to $10^{-5}$, and
set the SVD relative truncation criterion $\tau$ to be
$0.05$.   The size of the projected matrix is $8000$, which is much
reduced compared to the dimension of $A$. The projected
generalized eigenvalue problem is solved with the dense eigenvalue
solver \texttt{eig}.

We compare the accuracy of the LSS basis set by comparing the
eigenvalues within the interval $(\mu-\sigma,\mu+\sigma)=(3.5,4.5)$. There are
$914$ eigenvalues in this interval, and \texttt{eigs} takes $1886$ sec to
converge to tolerance with $10^{-5}$.  For LSS, the time for computing
the basis functions for all $16$ elements is $3989$ sec . The time for
constructing the projected matrix is $12$ sec, and the time for solving
the projected matrix is $93$ sec. For the projected matrix, we find
$919$ eigenvalues in total, and identified $5$ spurious spurious
eigenvalues. After removing the spurious eigenvalues with the largest
residual, the accuracy of the Ritz values compared to the true
eigenvalues are given in Fig.~\ref{fig:turon_m} (c).  In this case, the
LSS method is more expensive. This is mainly due to the cost for
constructing the LSS basis functions. However, this part can be
potentially performed independently for each element and without
inter-element communication on parallel computers.}
%
%The cost of LSS is
%more expensive xxx. However,xxx
%
%Our numerical experience indicates that for general sparse matrices, it
%becomes more difficult to prune the spurious eigenvalues. This is partly
%due to that for general sparse matrices the exponential decay is
%possible to be slow and  anisotropic, and the anisotropic decay can be
%exacerbated by the graph partitioning method for which the main goal is
%to minimize the $M$-cut heuristically.  The eigenvalues for this matrix
%is also more clustered, and further improvement of the quality of the
%basis functions to reduce the error leads to a large projection matrix
%that is difficult to solve. }

\section{Conclusion}\label{sec:conclusion}

In this paper, we present a method for constructing a novel basis set
called the localized spectrum slicing (LSS) basis set. Each function in
the LSS basis set is localized both spectrally and spatially, and
therefore can be used as an efficient way for representing 
eigenvectors of a general sparse Hermitian matrix corresponding to a
relatively narrow range of eigenvalues. The LSS basis set uses the decay
properties of analytic matrix functions, and can be constructed in a
divide-and-conquer method.   We show that by carefully tuning one
parameter $\sigma$, spatial locality and spectral locality of the basis
functions can be balanced. The projected matrices are both sparse and
have reduced sizes.
%\REV{which can be used to find
%interior eigenvalues without the knowledge of the inverse operator in
%the form of $(A-\mu I)^{-1}$ on the global domain.}

%In terms of the future work, the Gaussian function used in the LSS
%operator is a smooth approximation to the $\delta$ function.  The same
%concept of locality can be used to approximate other matrix functions,
%such as matrix sign functions.  When all the eigenvalues (either low end
%or high end) are included for the approximation, the projected
%eigenvalue problem can be free of spurious eigenvalues. This is, e.g.
%closely related to the recently developed adaptive local basis
%functions~\cite{LinLuYingE2012} and element orbitals~\cite{LinYing2012} for constructing
%efficient basis functions for solving the Kohn-Sham density functional
%theory. However, the error analysis of these generalized basis set is
%usually difficult.  The LSS type of basis set can be useful there as an
%alternative method to construct efficient basis functions, together with
%provable error bound for the approximation. The LSS basis set can also
%be used to efficiently characterize the eigenvectors close to the null
%space of $A$, which can then be used to construct preconditioners to
%accelerate linear solves for indefinite problems.

In terms of the future work, the Gaussian function used in the LSS
operator is a smooth approximation to the $\delta$ function.  The same
concept of locality can be used to approximate other matrix functions,
such as matrix sign functions. %The error bound of the LSS operator here
%explicitly requires that $A$ is sparse.  For dense matrices, it would be
%interesting to see whether the divide-and-conquer method can still be
%used to construct an LSS basis set, which can be used to generate a
%\textit{sparsified representation} of $A$ within certain spectrum range
%of $A$.  
This aspect is, e.g. closely related to the recently developed
adaptive local basis functions~\cite{LinLuYingE2012} and element
orbitals~\cite{LinYing2012} for constructing efficient basis functions
for solving the Kohn-Sham density functional theory. 
The LSS basis set can also be used to efficiently characterize the
eigenvectors close to the null space of $A$, which could potentially be used to
construct preconditioners to accelerate linear solves for indefinite
problems. 

From efficiency point of view, in the current implementation, the local
eigenvalue problem is solved mostly using a dense eigenvalue solver.  This is
still feasible for the 1D and 2D model problems presented in the
numerical section in this paper, but for 3D problems this is going to be
too expensive.  Efficient iterative solvers, or local Chebyshev
expansion based schemes should be used instead. Another practical issue
is to control the condition number of the LSS basis set when the SVD
truncation criterion is small.  An efficient way to identify a subset of
well conditioned LSS basis functions is needed to be more robust.

The balance between spatial and spectral locality is an important topic
in Fourier analysis and multi-resolution analysis.  Because the
construction of the LSS basis set is completely algebraic and can be
applied to any sparse Hermitian matrix, it is possible to extend the
current work to construct multi-resolution basis functions tailored for
given matrices, or multi-resolution basis functions for operators on graphs. 

\section*{Acknowledgments}

This work was supported by Laboratory Directed Research and Development
(LDRD) funding from Berkeley Lab, provided by the Director, Office of
Science, of the U.S. Department of Energy under Contract No.
DE-AC02-05CH11231, the DOE Scientific Discovery through the Advanced
Computing (SciDAC) program and the DOE Center for Applied Mathematics
for Energy Research Applications (CAMERA) program.

\end{document}